\documentclass[12pt]{article}
\usepackage{amsmath}
\usepackage{latexsym}
\usepackage{amssymb}
\newtheorem{thm}{Theorem}[section]
\newtheorem{la}[thm]{Lemma}
\newtheorem{Defn}[thm]{Definition}
\newtheorem{Remark}[thm]{Remark}
\newtheorem{Note}[thm]{Note}
\newtheorem{prop}[thm]{Proposition}

\newtheorem{cor}[thm]{Corollary}
\newtheorem{Example}[thm]{Example}
\newtheorem{Examples}[thm]{Examples}
\newtheorem{Problems}[thm]{Problems}

\newtheorem{Problem}[thm]{Problem}
\newtheorem{Convention}[thm]{Convention}
\newtheorem{Number}[thm]{\!\!}
\newenvironment{defn}{\begin{Defn}\rm}{\end{Defn}}

\newenvironment{example}{\begin{Example}\rm}{\end{Example}}

\newenvironment{rem}{\begin{Remark}\rm}{\end{Remark}}
\newenvironment{numba}{\begin{Number}\rm}{\end{Number}}
\newenvironment{proof}{{\noindent\bf Proof.}}%
                  {\nopagebreak\hspace*{\fill}$\Box$\medskip\medskip\par}   
\newcommand{\Punkt}{\nopagebreak\hspace*{\fill}$\Box$}
\newcommand{\wb}{\overline}
\newcommand{\ve}{\varepsilon}
\newcommand{\at}{\symbol{'100}}

\newcommand{\wt}{\widetilde}
\newcommand{\tensor}{\otimes}
\newcommand{\impl}{\Rightarrow}

\newcommand{\mto}{\mapsto}

\newcommand{\isom}{\cong}

\newcommand{\N}{{\mathbb N}}
\newcommand{\R}{{\mathbb R}}

\newcommand{\K}{{\mathbb K}}

\newcommand{\Z}{{\mathbb Z}}

\newcommand{\cU}{{\cal U}}

\newcommand{\sub}{\subseteq}
\DeclareMathOperator{\GL}{GL}

\newcommand{\aeq}{\Leftrightarrow}

\DeclareMathOperator{\pr}{pr}

\DeclareMathOperator{\id}{id}

\newcommand{\cA}{{\cal A}}

\newcommand{\cS}{{\cal S}}

\newcommand{\cL}{{\mathcal L}}
\newcommand{\obs}{{\footnotesize\rm s}}
\newcommand{\obc}{{\footnotesize\rm c}}
\newcommand{\obu}{{\footnotesize\rm u}}
\newcommand{\obcs}{{\footnotesize\rm cs}}

\DeclareMathOperator{\Lip}{Lip}

\newcommand{\sbull}{{\scriptscriptstyle \bullet}}

\DeclareMathOperator{\Pol}{Pol}

\begin{document}
\begin{center}
{\Large\bf Invariant Manifolds for Analytic Dynamical\\[2mm]
Systems over Ultrametric Fields}\\[6mm]
{\bf Helge Gl\"{o}ckner}\vspace{4mm}
\end{center}
\begin{abstract}
\hspace*{-7.2 mm}
We give an exposition
of the theory of invariant manifolds around a fixed point,
in the case of time-discrete,
analytic dynamical systems over a
complete ultrametric field $\K$.
Typically,
we consider an analytic manifold $M$ modelled
on an ultrametric Banach space
over $\K$, an analytic diffeomorphism
$f\colon M \to M$,
and a fixed point $p$ of $f$.
Under suitable
assumptions on the tangent map $T_p(f)$, we
construct a centre-stable manifold,
a centre manifold,  respectively,
an $a$-stable manifold
around~$p$,
for a given real number $a\in \;]0,1]$.\vspace{3mm}
\end{abstract}
{\footnotesize {\em Classification}:
37D10 (Primary) 
46S10, 
26E30 (Secondary)\\[2.5mm]
}\\[3.4mm]
{\footnotesize{\bf Contents}\\[1.6mm]
1.\hspace*{3.2mm}Introduction and statement of the main results\dotfill\pageref{secintro}\\
2.\hspace*{3.2mm}Preliminaries and notation\dotfill\pageref{secprepa}\\
3.\hspace*{3.2mm}Centre-stable manifolds\dotfill\pageref{seccs}\\
4.\hspace*{3.2mm}Centre manifolds\dotfill\pageref{secc}\\
5.\hspace*{3.2mm}Mappings between sequence spaces\dotfill\pageref{secseq}\\
6.\hspace*{3.2mm}Construction of local stable manifolds\dotfill\pageref{secirw}\\
7.\hspace*{3.2mm}Global stable manifolds\dotfill\pageref{secglob}\\
8.\hspace*{3.2mm}Local unstable manifolds\dotfill\pageref{secirw3}\\
9.\hspace*{3.2mm}Spectral interpretation of hyperbolicity\dotfill\pageref{secfin}\\
10. Behaviour close to a fixed point\dotfill\pageref{seccon}\\
11. When $W_a^\obs(f,p)$ is not only immersed\dotfill\pageref{notonly}\\
12. Further conclusions in the finite-dimensional case\dotfill\pageref{adepend}\\
13. Specific results concerning automorphisms of Lie groups\dotfill\pageref{seclie}\\[1.6mm]
Appendix A: Proof of Proposition~\ref{thmloccsbb}\dotfill\pageref{sec1app}\\
Appendix B: Proof of Theorem~\ref{locumfd}\dotfill\pageref{sec2app}}
\section{\hspace*{-2.5mm}Introduction and statement of main results}\label{secintro}
In this article, we construct various types of invariant manifolds
for analytic dynamical systems over complete ultrametric fields.
The invariant manifolds are useful in the theory of Lie groups
over local fields, where they allow results to be extended
to ground fields of positive characteristic,
which previously where available only in characteristic $0$
(i.e., for $p$-adic Lie groups). The results also constitute
a first step towards a theory of partially hyperbolic
dynamical systems over complete ultrametric fields.\\[2.5mm]
{\bf Definitions and main results.}
As in the real case, hyperbolicity\linebreak assumptions
are essential for a discussion of invariant manifolds.
To explain the appropriate conditions
in the ultrametric case, let $E$ be an ultrametric
\linebreak
Banach space over a complete
ultrametric field $(\K,|.|)$.
Let $\alpha \colon E\to E$ be a continuous
$\K$-linear map,
and $a\in \; ]0,\infty[$.
%
\begin{defn}\label{defahyp}
We say that $\alpha$ is \emph{$a$-hyperbolic}
if there exist $\alpha$-invariant
vector subspaces $E_{a,\obs}$ and $E_{a,\obu}$ of $E$ such that
$E=E_{a,\obs} \oplus E_{a,\obu}$,
and an ultrametric norm $\|.\|$ on $E$ defining its topology,
with properties (a)--(c):
\begin{itemize}
\item[(a)]
$\|x+y\|=\max\{\|x\|,\|y\|\}$
for all $x\in E_{a, \obs}$ and $y\in E_{a,\obu}$;
\item[(b)]
$\alpha_2:=\alpha|_{E_{a,\obu}}$
is invertible;
\item[(c)]
$\|\alpha_1\|< a$ and $\frac{1}{\|\alpha_2^{-1}\|}>a$
holds for the operator norms with respect to $\|.\|$,
where $\alpha_1:=\alpha|_{E_{a,\obs}}$
(and $\frac{1}{0}:=\infty$).
\end{itemize}
Then $E_{a,\obs}$ is uniquely determined (see Remark~\ref{hypsuniq}).
If $a=1$, we also write $E_\obs:=E_{1,\obs}$
and $E_\obu:=E_{1,\obu}$.
\end{defn}
As is to be expected, $a$-hyperbolicity
can be read off from the spectrum of $\alpha$ if $E$ is finite-dimensional
(see Corollary \ref{spechyper}):
Then $\alpha$ is $a$-hyperbolic
if and only if $a\not=|\lambda|$
for each eigenvalue $\lambda$ of
$\alpha\tensor_\K\id_{\wb{\K}}$ 
in an algebraic closure $\wb{\K}$.\\[2.5mm]
Now consider an analytic manifold $M$ modelled on an ultrametric Banach
space $E$ over $\K$ (as in \cite{Bo1}).
Let $f\colon M\to M$
be an analytic diffeomorphism,
and $p\in M$ be a fixed point of $f$.
%
\begin{defn}\label{defwas}
Given $a\in \;]0,1]$, we define $W_a^\obs(f,p)\sub M$,
the \emph{$a$-stable set} around $p$ with respect to $f$,
as the set of all $x\in M$ such that
%
\begin{equation}\label{dewaseq}
\mbox{$f^n(x)\to p\;$  as $\;n\to\infty\;$
and $\; a^{-n}\|\kappa(f^n(x))\|\to 0\, $,}
\end{equation}
for some (and hence every) chart $\kappa\colon U\to V\sub E$
of $M$ around $p$ such that $\kappa(p)=0$,
and some (and hence every) ultrametric norm $\|.\|$ on $E$
defining its topology.\footnote{See Remark \ref{indep} for the independence
of the choice of $\kappa$ and $\|.\|$.}
\end{defn}
It is clear from the definition that $W_a^\obs:=W_a^\obs(f,p)$
is stable under $f$, i.e., $f(W_a^\obs)=W_a^\obs$.
Now $a$-hyperbolicity of $T_p(f)$ ensures
that $W_a^\obs$ is a manifold, the
\emph{$a$-stable manifold} around $p$ with respect
to $f$ (see Section \ref{secglob}):
%
%
\begin{thm}[Ultrametric Stable Manifold Theorem]\label{ultrastabm}
Let $M$ be an analytic manifold modelled on an ultrametric Banach
space over a complete ultrametric field~$\K$.
Let $f\colon M\to M$
be an analytic diffeomorphism
and $p\in M$ be a fixed point of $f$.
If $a\in \;]0,1]$ and $T_p(f)\colon T_p(M)\to T_p(M)$ is $a$-hyperbolic,
then there exists a unique analytic manifold structure
on $W_a^\obs:=W_a^\obs(f,p)$ such that {\rm(a)--(c)} hold:
\begin{itemize}
\item[\rm(a)]
$W_a^\obs$ is an immersed submanifold of $M$;
\item[\rm(b)]
$W_a^\obs$ is tangent to the $a$-stable subspace $T_p(M)_{a, \obs}$
$($with respect to $T_p(f))$, i.e.,
$T_p(W_a^\obs)= T_p(M)_{a,\obs}$;
\item[\rm(c)]
$f$ restricts to an analytic diffeomorphism $W_a^\obs\to W_a^\obs$.
\end{itemize}
Moreover, each neighbourhood of $p$ in $W_a^\obs$ contains
an open neighbourhood $\Omega$ of $p$ in $W_a^\obs$
which is a submanifold of $M$,
is $f$-invariant $($i.e., $f(\Omega)\sub \Omega)$,
and such that
$W_a^\obs =\bigcup_{n=0}^\infty f^{-n}(\Omega)$.
\end{thm}
In case of $1$-hyperbolicity, one simply speaks of hyperbolicity.
Moreover, $W_1^\obs$ is simply called the \emph{stable manifold}
around~$p$, and denoted~$W^\obs$.\\[2.5mm]
To obtain so-called \emph{centre-stable manifolds} and \emph{centre manifolds}
around a given fixed point $p$, again we need to impose
appropriate conditions on $T_p(f)$.
To formulate these, let
$E$ be an ultrametric Banach space over $\K$.
Moreover, let $\alpha\colon E\to E$
be a continuous linear map,
and $a\in \;]0,\infty[$.
%
%
\begin{defn}\label{defcsub}
An $\alpha$-invariant vector subspace
$E_{a,\obcs} \sub E$ is called an \emph{$a$-centre-stable
subspace} with respect to $\alpha$
if there exists an $\alpha$-invariant
vector subspace $E_{a,\obu}$ of $E$ such that
$E=E_{a,\obcs} \oplus E_{a,\obu}$ and
$\alpha_2:= \alpha|_{E_{a,\obu}}\colon E_{a,\obu}\to E_{a,\obu}$
is invertible,
and there exists an ultrametric norm $\|.\|$ on $E$ defining its topology,
with the following properties:
\begin{itemize}
\item[(a)]
$\|x+y\|=\max\{\|x\|,\|y\|\}$
for all $x\in E_{a,\obcs}$, $y\in E_{a,\obu}$; and
\item[(b)]
$\|\alpha_1\|\leq a$ and $\frac{1}{\|\alpha_2^{-1}\|}>a$
holds for the operator norms with respect to $\|.\|$,
where $\alpha_1:=\alpha|_{E_{a,\obcs}}$.
\end{itemize}
Then $E_{a,\obcs}$ is uniquely determined (see Remark~\ref{unicst}).
\end{defn}
%
\begin{defn}\label{defacentre}
If $\alpha$ is an automorphism,
we say that an $\alpha$-invariant vector subspace
$E_{a,\obc} \sub E$ is an \emph{$a$-centre
subspace} with respect to $\alpha$
if there exist $\alpha$-invariant
vector subspaces $E_{a,\obs}$ and $E_{a,\obu}$ of $E$ such that
$E=E_{a,\obs}\oplus E_{a,\obc} \oplus E_{a,\obu}$,
and an ultrametric norm $\|.\|$ on $E$ defining its topology,
with the following properties:
\begin{itemize}
\item[(a)]
$\|x+y+z \|=\max\{\|x\|,\|y\|,\|z\|\}$
for all $x\in E_{a,\obs}$, $y\in E_{a,\obc}$  and $z\in E_{a,\obu}$;
\item[(b)]
$\|\alpha(x)\|=a\|x\|$ for all $x\in E_{a,\obc}$; and
\item[(c)]
$\|\alpha_1\|<a$ and $\frac{1}{\|\alpha_3^{-1}\|}>a$ hold for
the operator norms with respect
to $\|.\|$, where $\alpha_1:=\alpha|_{E_{a,s}}$ and $\alpha_3:=\alpha|_{E_{a,\obu}}$.
\end{itemize}
Then $E_{a,\obs}$, $E_{a,\obc}$ and
$E_{a,\obu}$ are uniquely determined (Remark~\ref{udecen});
$E_{a,\obs}$ and $E_{a,\obu}$ are
called the \emph{$a$-stable}
and \emph{$a$-unstable} subspaces of $E$ with respect to~$\alpha$, respectively
(and likewise in Definition \ref{defahyp}).
If $a=1$, we simply speak of stable, centre and unstable
subspaces, and write~$E_\obs$, $E_\obc$ and
$E_{\obu}$ instead of
$E_{1,\obs}$, $E_{1,\obc}$ and~$E_{1,\obu}$.
\end{defn}
If $E$ is finite-dimensional,
then an $a$-centre-stable subspace with respect to
a linear map $\alpha\colon E\to E$
always exists, for any $a\in \;]0,\infty[$.
An $a$-centre subspace exists if
$E$ is finite-dimensional and $\alpha$ an automorphism
(see Section~\ref{secfin}).
\begin{numba}\label{thesituintr}
Let $M$
be an analytic manifold modelled on an ultrametric Banach space
over a complete ultrametric field~$\K$.
Let $M_0\sub M$ be open,
$f\colon M_0 \to M$ be an analytic mapping,
$p\in M_0$ be a fixed point of $f$,
and $a\in \;]0,1]$.\vspace{-2mm}
\end{numba}
%
\begin{defn}\label{defcsta}
If $T_p(M)$ has an $a$-centre-stable subspace $T_p(M)_{a,\obcs}$ with
respect to $T_p(f)$,
we call an immersed submanifold $N\sub M_0$
an \emph{$a$-centre-stable manifold} around~$p$ with respect to~$f$
if (a)--(d) are satisfied:
\begin{itemize}
\item[(a)]
$p\in N$;
\item[(b)]
$N$ is tangent to
$T_p(M)_{a,\obcs}$ at $p$, i.e.,
$T_p(N)=T_p(M)_{a,\obcs}$;
\item[\rm(c)]
$f(N)\sub N$; and
\item[\rm(d)]
$f|_N\colon N\to N$ is analytic.
\end{itemize}
If $a=1$, we simply speak of a \emph{centre-stable manifold}.
\end{defn}
%
\begin{defn}\label{defcenma}
If $T_p(f)$ is an automorphism
and $T_p(M)$ has a centre subspace $T_p(M)_\obc$
with respect to $T_p(f)$,
we say that an immersed submanifold $N\sub M_0$
is a \emph{centre manifold} around $p$ with respect to $f$
if (a), (c) and (d) fom Definition \ref{defcsta} hold as well as
\begin{itemize}
\item[(b)$'$]
$N$ is tangent to
$T_p(M)_\obc$ at~$p$, i.e.,
$T_p(N)=T_p(M)_\obc$.
\end{itemize}
\end{defn}
Given a manifold $M$, $p\in M$
and immersed submanifolds
$N_1,N_2\sub M$\linebreak
containing~$p$,
let us write $N_1\sim_p N_2$
if there exists an open neighbourhood $U$
of~$p$ in~$N_1$
which is also an open neighbourhood of~$p$ in~$N_2$,
and on which $N_1$ and $N_2$ induce the same analytic
manifold structure. The $\sim_p$-equivalence class
of an immersed submanifold $N\sub M$ is called
its \emph{germ} at $p$.\\[2.5mm]
The following result is obtained in Section \ref{seccs}:
%
\begin{thm}[Ultrametric Centre-Stable Manifold Theorem]\label{csthm}
\hspace*{2mm}Let\linebreak
$a\in \;]0,1]$
and assume that $T_p(M)$ admits an $a$-centre-stable subspace with\linebreak
respect to $T_p(f)$,
in the situation of {\rm\ref{thesituintr}}.
Then the following holds:
\begin{itemize}
\item[\rm(a)]
There exists an $a$-centre-stable manifold
$N$ around $p$ with
respect to $f$;
\item[\rm(b)]
The germ of $N$ at $p$ is uniquely determined;
\item[\rm(c)]
Every neighbourhood of~$p$ in~$N$ contains
an open neighbourhood~$\Omega$ of~$p$ in~$N$
which is an $a$-centre-stable
manifold and
a submanifold of~$M$.
\end{itemize}
\end{thm}
As concerns centre manifolds, we show (see Section \ref{secc}):
%
%
\begin{thm}[Ultrametric\hspace*{-.7mm} Centre\hspace*{-.7mm} Manifold\hspace*{-.7mm} Theorem]\label{cthm}
\hspace*{-1.4mm}In\hspace*{-.7mm} the\hspace*{-.7mm} setting\linebreak
of {\rm\ref{thesituintr}},
assume that $T_p(f)$ is an automorphism
and assume that $T_p(M)$ has a centre subspace with respect
to $T_p(f)$. Then
\begin{itemize}
\item[\rm(a)]
There exists a centre manifold $N$ around $p$ with respect to $f$;
\item[\rm(b)]
The germ of $N$ at $p$ is uniquely determined;
\item[\rm(c)]
Each neighbourhood of $p$ in $N$
contains an open neighbourhood~$\Omega$ of~$p$ in~$N$
which is a centre manifold,
a submanifold of~$M$,
stable under $f$ $($i.e., $f(\Omega)=\Omega)$,
and for which $f|_\Omega\colon \Omega \to \Omega$ is
a diffeomorphism.
\end{itemize}
\end{thm}
It is essential for the uniqueness
assertions in part~(b) of Theorem~\ref{csthm} and~\ref{cthm}
that all manifolds (and submanifolds) we consider are analytic manifolds.\\[2.5mm]
Local $a$-unstable manifolds
can also be discussed, for $a\geq 1$ (see Section \ref{secirw3}).
In  Sections \ref{seccon} to \ref{adepend},
we describe general
consequences of our results,
and in
Section \ref{seclie}
we draw more specific conclusions concerning Lie groups.
Results from these sections are
vitally used in~\cite{MaZ} and~\cite{SPO},
to obtain information on
automorphisms of finite-dimensional
Lie groups over local fields of positive characteristic.
To explain the motivation for the current article,
and to show the utility of its results,
we now briefly describe two applications which are
only available through the use of invariant manifolds.\\[3mm]
{\bf Applications in Lie theory.}
If $G$ is a totally disconnected, locally compact topological group
with neutral element $1$ 
and $\alpha\colon G\to G$ an automorphism of topological groups,
then
\[
U_\alpha:=\{x\in G\colon \mbox{$\alpha^n(x)\to 1$ as $n\to\infty$}\}
\]
is called the \emph{contraction group} of $\alpha$ and
\[
M_\alpha:=\{x\in G\colon \mbox{$\alpha^\Z(x)$ is relatively compact in $G$} \}
\]
the \emph{Levi factor},
where $\alpha^\Z(x):=\{\alpha^n(x)\colon n\in \Z\}$ (see \cite{BaW}).
Now assume that $G$
is an analytic finite-dimensional Lie group over a local field $\K$
and $\alpha\colon G\to G$ an analytic automorphism.
Since $\alpha(1)=1$, we are in the situation of the current article.
Using invariant manifolds,
one can prove the following results
in arbitrary characteristic
(the $p$-adic case of which is due to Wang \cite{Wan})\footnote{Our results
also enable the calculation of the
``scale'' $s(\alpha)$ (introduced in \cite{Wi1}, \cite{Wi2})
if $U_\alpha$ is closed~\cite{SPO},
and to discuss Lie groups of type $R$
over local fields of arbitrary~characteristic.\linebreak
Previously, this was only possible
in the $p$-adic case (see \cite{SCA} and \cite{Raj},
respectively.~Compare also \cite{BaW}
for the scale of inner automorphisms
of reductive algebraic groups).}:
\begin{itemize}
\item[(a)]
\emph{The group $U_\alpha$ is always nilpotent}
(see \cite[Theorem B]{MaZ}).
\item[(b)]
\emph{If $U_\alpha$ is closed, then $U_\alpha$, $U_{\alpha^{-1}}$
and $M_\alpha$ are Lie subgroups of $G$.
Moreover, $U_\alpha M_\alpha U_{\alpha^{-1}}
$ is an open subset
of $G$ and the ``product map''}
\[
\pi\colon U_\alpha\times M_\alpha \times U_{\alpha^{-1}}\to
U_\alpha M_\alpha U_{\alpha^{-1}}\,,\quad (x,y,z)\mto xyz
\]
\emph{is an analytic diffeomorphism}
(see \cite{SPO}; cf.\ also the sketch in
\cite{SUR}).
\end{itemize}
In fact, the $a_j$-stable manifolds $G_j:=W_{a_j}^\obs(\alpha,1)$
provide a central series
$\{1\}=G_1\sub G_2\sub\cdots \sub G_n=G$
of Lie subgroups of $G$,
for suitable real numbers
$0<a_1<\cdots < a_n<1$ (see \cite{MaZ}).
And to get (b),
one heavily uses
the (stable) manifold structures on
$U_\alpha=W^\obs(\alpha,1)$ and
$U_{\alpha^{-1}}=W^\obs(\alpha^{-1},1)$ constructed here,
and the fact that $M_\alpha$ contains a centre manifold
for $\alpha$ around $1$ (see \cite{SPO}; cf.\ also \cite{SUR}).\\[3mm]
{\bf Methods.}
Using a local chart around $p$,
the constructions of
$a$-centre-stable manifolds
and (local) $a$-stable manifolds
are easily reduced to the case where $M=E$ is an ultrametric
Banach space and $f$ is an analytic $E$-valued
map on an open ball $B^E_r(0)\sub E$,
such that $f(0)=0$.
Write $E=E_1\oplus E_2$\linebreak
and (accordingly) $f=(f_1,f_2)$,
where $E_1$ is the $a$-centre stable (resp., the $a$-stable)
subspace of $E$ and $E_2$ the $a$-unstable subspace.
Now the idea is to construct an $a$-centre-stable manifold
(resp., a local $a$-stable manifold)
as the graph $\Gamma$ of an analytic $E_2$-valued map~$\phi$
on some ball in $E_1$.\\[2.5mm]
\emph{Construction of $a$-centre-stable manifolds.}
In this case, the required $f$-invariance of $\Gamma$
necessitates that
\begin{equation}\label{formal}
f_2(x,\phi(x))=\phi(f_1(x,\phi(x)))
\end{equation}
for small $x\in E_1$. Writing now $f_1$, $f_2$ and $\phi$
as convergent series,  (\ref{formal}) can be read
as an identity for formal series, which enables us to determine
the coefficients of~$\phi$ recursively (see Section \ref{seccs}).\\[2.5mm]
\emph{Construction of $($local\hspace*{.3mm}$)$ $a$-stable manifolds.}
We construct local $a$-stable\linebreak
manifolds
by an adaptation of a method used by M.\,C. Irwin
in the real case \cite{Ir1} .
Instead of constructing the points $z=(x,\phi(x))$
of the local $a$-stable manifold directly,
the central idea of Irwin was to construct,
instead, their orbits $\omega(x):=(f^n(z))_{n\in \N_0}$.
These are elements of a suitable Banach space of sequences,
and satisfy a certain identity
\[
G(\omega(x))\;=\; \big((x,0),(0,0),\ldots\big)\,,
\]
which can be solved for $\omega(x)$
using the inverse function theorem for
$k$ times continuously Fr\'{e}chet differentiable
functions between Banach spaces
(cf.\ \cite{AaM}
for an extension of Irwin's method to real analytic dynamical systems).\\[2.5mm]
As an inverse function theorem is also available
for analytic maps between ultrametric Banach spaces,
we can adapt
Irwin's method to the ultrametric case (see Section \ref{secirw}
for the construction, and Section \ref{secseq}
for auxiliary results concerning sequence
spaces).
Our discussion also profited much from \cite{Wel}.\\[3mm]
{\bf General remarks.}
It should be mentioned that (of course!)
our primary interest lies in the finite-dimensional case.
However, Irwin's method forces us to consider
infinite-dimensional sequence spaces.
Moreover, the discussion of centre-stable (and centre) manifolds
actually becomes easier
if one uses
a coordinate-free formulation
(which avoids the use of multi-indices).\\[2.5mm]
We also wish to mention
that although locally analytic functions
are used as the basis of our studies,
the dynamical systems give rise to
various \emph{global} objects (not only to germs around
the fixed points).
Examples are the $a$-stable manifold $W_a^\obs(f,p)$
and the Levi factor~$M_\alpha$ in a Lie group~$G$,
which is a distinguished centre manifold
(if $U_\alpha$ is closed).\\[2.5mm]
{\bf Relations to the literature.}
As just explained, our methods and results have their
roots in the theory of smooth dynamical systems over the
real ground field, and drew some inspiration
from classical sources
in this area (in the case of stable manifolds).\footnote{See
also \cite{HPS}, \cite{HaK} and the references therein.}
Complementary to our studies,
much of the literature on ultrametric dynamical systems
can rather be regarded as an offspring
of complex dynamics,
and has concentrated on the $1$-dimensional case
(see, e.g., \cite{Ben} and \cite{Bez}).
Some specific new phenomena arose there,
like the existence of wandering domains \cite{Bn2}.
It also turned out to be necessary in some situations to extend
the action of polynomials or rational functions
from the ordinary projective line
to the Berkovich projective line,
because the latter supports relevant measures
while the projective line does not,
in contrast to the classical complex case  \cite{FaR}
(further ultrametric
phenomena can be found in~\cite{AaK}
and \cite{Nil}. For relations to formal groups, see \cite{Lub}).\\[2.5mm]
While the preceding list could easily be prolonged,
papers devoted to multi-dimensional non-archimedean
dynamical systems are quite rare.
Notable exceptions
are the work of Herman and Yoccoz \cite{HaY}
on the analytic linearizability problem
in several variables (over ultrametric fields
of characteristic zero)
and the closely related recent Ph.D.-thesis
\cite{Vie} by D. Vieugue.\footnote{Concerning the single variable case,
compare also K.-O. Lindahl's thesis \cite{Lin}.}
Further works include \cite{Aga} and \cite{Ag2}.
We mention that if a (finite-dimensional)
ultrametric dynamical system is analytically conjugate
to a linear system (at least locally around a fixed point),
then it is very easy to obtain invariant manifolds
(as the images of the \mbox{corresponding} vector subspaces).
However, an analytic linearization is only possible in
special situations,
and hence the existing results are insufficient to
deal with the Lie-theoretic problems\linebreak
described above.
By contrast, the results we provide are quite general,
and apply just as well if a linearization is not available
(and in any characteristic).\\[2.5mm]
{\bf Perspectives.}
Having started on this road, it would be natural to\linebreak
proceed
and take further steps towards a non-archimedean
analogue of the theory of partially hyperbolic dynamical systems.
One essential point would be the study of
invariant foliations
(e.g., locally around a fixed point),
which would give
refined information on the dynamics.
Such extensions auto\-matically lead us outside
the class of analytic functions,
and necessitate the consideration
of functions with weaker regularity properties,
like functions which are only $C^k$,
Lipschitz, or H\"{o}lder
(cf.\ also \cite[p.\,133]{HaK}).\footnote{See \cite{DIF}, \cite{IMP}, \cite{CMP}
and \cite{FIO} for the basic theory of such functions.
For the single-variable case, consult
\cite{Sch} and the references therein.}
In fact, in the real case it is well-known
that smooth (or analytic) dynamical systems
may give rise to foliations which are
H\"{o}lder, but not $C^1$ (nor Lipschitz).\\[2.5mm]
From this point of view, it is very natural
to construct invariant manifolds also for $C^k$-dynamical
systems over ultrametric fields.\footnote{These
are also needed to adapt the above Lie-theoretic
results to non-analytic $C^k$-Lie groups
or non-analytic $C^k$-automorphisms,
as constructed in \cite{NOA}.}
These constructions (which are
more complicated than the analytic case)
are in preparation.
In a nutshell,
Irwin's method still provides $a$-stable manifolds
in the case of $C^k$-dynamical systems
modelled on an ultrametric Banach space,\footnote{With some
precautions if $k=1$ and the modelling space is infinite-dimensional.}
using the ultrametric inverse function theorem
for $C^k$-maps provided in \cite{FIO} (cf.\ \cite{IMP} for weaker results).
In the $C^k$-case,
centre-stable manifolds (for finite $k$) are constructed
as $a$-pseudo-stable manifolds
with $a>1$ close to $1$.
The latter are available through an ultrametric analogue
of Irwin's method from \cite{Ir2} (cf.\ also \cite{LaW}).\\[2.5mm]
Let us remark in closing that part of the theory becomes
nicer and easier if the real field is replaced by an ultrametric field.
For example, the
astute
reader may have noticed
that part~(c) in Theorem \ref{csthm}
and \ref{cthm} (and also Definition~\ref{defacentre}\,(b))
would be too much to ask for
in the real case.
However, some other aspects become
more complicated in the non-archimedean setting
(for example, the discussion $C^k$-dynamical systems).\vfill\pagebreak

\noindent
{\bf Acknowledgements.}
The article is an elaboration of a mini-course
on\linebreak
ultrametric invariant manifolds given by the author
at the CIRM (Luminy),
as
part of the \emph{Session th\'{e}matique autour
de la dynamique non-archim\'{e}dienne} in July 2008.
The author appreciated the invitation
and support.
A remark
by Jean-Yves Briend
inspired the treatment of $a$-centre-stable manifolds
for $a\in \;]0,1]$ (originally, the author assumed \mbox{$a=1$).}
Juan Rivera-Letelier
gave valuable advice
on the importance
of weakened regularity
properties, and widened the author's horizon of research perspectives.
The studies were started during a research visit to the University
of Newcastle (N.S.W.) in 2004,
supported by the German Research Foundation
(DFG,
project 447 \mbox{AUS-113/22/0-1})
and the Australian Research Council (project LX 0349209).
The manuscript was completed during a visit
to Newcastle in 2008,
supported by the DFG (project GL 357/6-1)
and ARC (project DP0556017).
\section{Preliminaries and notation}\label{secprepa}
In this section, we fix notation
and recall some basic (but essential) facts concerning
analytic functions between open subsets of ultrametric
Banach spaces.
First of all, let us mention that
$\N:=\{1,2,\ldots\}$ and $\N_0:=\N\cup\{0\}$
in this article. We write
$\Z$
for the integers
and $\R$ for the field of real numbers.
If $f\colon M\to M$ and $n\in \N$,
we write $f^n:=f\circ\, \cdots\, \circ f$
for the $n$-fold composition,
and $f^0:=\id_M$. If $f$ is invertible,
we define $f^{-n}:=(f^{-1})^n$.\\[2.5mm]
{\bf Ultrametric Banach spaces.}
Recall that an \emph{ultrametric field}
is a field $\K$, together with
an absolute value $|.|\colon \K\to [0,\infty[$
which satisfies the ultrametric inequality.
We shall always assume that the metric $d\colon \K \times \K \to [0,\infty[$,
$d(x,y):=|x-y|$, defines a non-discrete
topology on $\K$. If the metric space
$(\K,d)$ is complete,
then the ultrametric field $(\K,d)$ is called \emph{complete}.
A totally disconnected, locally compact, non-discrete
topological field is called a \emph{local field}.
Any such admits an ultrametric absolute value
making it a complete ultrametric field \cite{Wei}.
See, e.g., \cite{Sch} for background concerning
complete ultrametric fields.\\[2.5mm]
An \emph{ultrametric Banach space} over an ultrametric field
$\K$ is a complete normed space $(E,\|.\|)$ over $\K$
whose norm $\|.\| \colon E\to [0,\infty[$ satisfies the \emph{ultrametric inequality},
$\|x+y\|\leq \max\{\|x\|, \|y\|\}$ for all $x,y\in E$.
The ultrametric inequality entails
the following \emph{domination principle}:
%
\begin{equation}\label{domi}
\|x+ y\|= \|x\|\quad\mbox{for all $x,y\in E$ such that $\|y\|<\|x\|$.}
\end{equation}
Given $x\in E$ and $r\in \;]0,\infty]$, we set
$B^E_r(x):=\{y\in E\colon \|y-x\|<r\}$.\\[2.5mm]
{\bf Linear operators.}
Given an ultrametric Banach space $E$,
we let $\cL(E)$ be the set of all continuous linear
self-maps of $E$. Then the operator norm
makes $\cL(E)$ an ultrametric Banach space,
and it is a unital $\K$-algebra under composition.
We write
\[
\GL(E):=\cL(E)^\times:=\{A\in \cL(E)\colon (\exists B\in \cL(E))\;\;
AB=BA=\id_E\}
\]
for its unit group.
%
\begin{numba}\label{lininv}
The domination principle entails that $\id_E-A$ is an
isometry for each $A\in \cL(E)$ of operator norm $\|A\|<1$.
Moreover, $\id_E-A$ is invertible, because it is easy to see that
the Neumann series
$\sum_{k=0}^\infty A^k$ provides an inverse for $\id_E-A$. Then
also $(\id_E-A)^{-1}$ is an isometry. In particular,
%
\begin{equation}\label{nonexpa}
\|(\id_E-A)^{-1}\|\leq 1\quad\mbox{for all $A\in \cL(E)$ such that $\|A\|<1$.}
\end{equation}
\end{numba}
See, e.g., \cite{Roo} for background concerning
ultrametric Banach spaces.\\[2.5mm]
{\bf Spaces of homogeneous polynomials.}
We now discuss continuous homogeneous polynomials
and analytic functions between ultrametric Banach spaces.
As we are only dealing with a special case
of the situation in \cite{Bo1} (our main reference),
simpler notation will be sufficient.\\[2.5mm]
Let $(E,\|.\|_E)$ and $(F,\|.\|_F)$ be ultrametric Banach spaces over
a complete ultrametric field $\K$.
If $k\in \N_0$, we let $\cL^k(E,F)$ be the set of all continuous
$k$-linear mappings $\beta\colon E^k\to F$.
Thus $\cL^0(E,F)\isom F$, and $\cL^k(E,F)$ for $k\geq 1$
is an ultrametric Banach space
with norm given by
\[
\|\beta\|:=\sup\left\{
\frac{\|\beta(x_1,\ldots, x_k)\|_F}{\|x_1\|_E\ldots\|x_k\|_E}\colon x_1\ldots, x_k\in E\setminus \{0\}\right\}
\in [0,\infty[\,.
\]
If $k\geq 1$, we write $\Delta_k$ (or $\Delta_k^E$) for
the diagonal map $E\to E^k$, $x\mto (x,\ldots, x)$.
If $k=0$, define $\Delta_0:=\Delta_0^E\colon E\to E^0=\{0\}$, $x\mto 0$.
A map $p\colon E\to F$ is called a \emph{continuous homogeneous polynomial
of degree $k$} if there exists $\beta\in \cL^k(E,F)$ such that
$p=\beta\circ \Delta_k$. We let $\Pol^k(E,F)$ be the space of all
continuous homogeneous polynomials $p\colon E\to F$ of degree $k$.
Then
\[
\cL^k(E,F)\to \Pol^k(E,F)\,,\quad \beta\mto \beta\circ \Delta_k
\]
is a surjective linear map. We equip $\Pol^k(E,F)$
with the quotient norm, which makes it an ultrametric Banach space.
Thus
\[
\|p\|=\inf\{\|\beta\|\colon \mbox{$\beta\in \cL^k(E,F)$ such that $p=\beta\circ\Delta_k$}\}\,.
\]
{\bf Pullbacks and pushforwards.}
If $E_1$, $E_2$ and $F$ are ultrametric Banach spaces and $A\colon E_1\to E_2$
is a continuous linear map, we obtain a linear map
\[
\Pol^k(A,F)\colon \Pol^k(E_2,F)\to \Pol^k(E_1,F)\,,\quad
p\mto p\circ A\,.
\]
Similarly,
if $E$, $F_1$ and $F_2$ are ultrametric Banach spaces and $B\colon F_1\to F_2$
is a continuous linear map, we obtain a linear map
\[
\Pol^k(E,B)\colon \Pol^k(E,F_1)\to \Pol^k(E,F_2)\,,\quad
p\mto B\circ p\,.
\]
%
\begin{la}\label{pbpfw}
The linear mappings $A^*:=\Pol^k(A,F)$ and $B_*:=\Pol^k(E,B)$
are continuous, of operator norm
\begin{eqnarray}
\|A^*\| & \leq & \|A\|^k\qquad\mbox{and}\label{eston1}\\
\|B_*\| &\leq & \|B\|\,.\label{eston2}
\end{eqnarray}
\end{la}
\begin{proof}
Let $p\in \Pol^k(E_2,F)$. If $\beta\in \cL^k(E_2,F)$ such that $p=\beta\circ \Delta_k^{E_2}$,
then $A^*(p)=p\circ A=\beta\circ \Delta_k^{E_2}\circ A=\gamma\circ \Delta_k^{E_1}$,
where $\gamma :=\beta\circ (A\times\cdots\times A)$
and $\|\gamma\|\leq \|A\|^k \|\beta\|$.
Thus $\|A^*(p)\|\leq \|A\|^k \|\beta\|$
and hence  $\|A^*(p)\|\leq \|A\|^k \|p\|$ (passing to the infimum),
which entails (\ref{eston1}).
The proof of (\ref{eston2}) is similar.
\end{proof}
{\bf Analytic functions.}
Let $E$ and $F$ be ultrametric Banach spaces over a complete
ultrametric field.
Given $(p_k)_{k\in \N_0}\in \prod_{k\in \N_0} \Pol^k(E,F)$,
let $\rho$ be the supremum of the set of all $r\geq 0$
such that
\[
\lim_{k \to \infty } \|p_k\|r^k \; =\; 0\,.
\]
Then $\rho$ is called the \emph{radius of strict convergence}
of the series $\sum_{k\in\N_0}p_k$,
and $B_\rho^E(0)$ its \emph{domain of strict convergence}.\\[2.5mm]
A function $f\colon U\to F$ on an open subset $U\sub E$
is called (locally) \emph{analytic} if, for each $x\in U$,
there exist
$(p_k)_{k\in \N_0}\in \prod_{k\in \N_0} \Pol^k(E,F)$
such that the series $\sum_{k\in\N_0}p_k$
has a positive radius $\rho$ of strict convergence
and there exists $r \in \;]0,\rho]$ such that $B_r^E(x)\sub U$
and
\[
f(x+y)=\sum_{k=0}^\infty p_k(y)\quad \mbox{for all $\, y\in B_r^E(0)$.}
\]
We recall that if $B_\rho^E(0)$ is the domain of strict
convergence of $\sum_{k=0}^\infty p_k$ with $p_k\in \Pol^k(E,F)$
(and $\rho>0$), then the corresponding function
\[
f\colon B_\rho^E(0)\to F\,,\quad f(z):= \sum_{k=0}^\infty p_k(z)
\]
is analytic \cite[4.2.4]{Bo1}.
It is well-known that compositions of composable\linebreak
analytic functions are
analytic (see \cite[4.2.3 and 3.2.7]{Bo1}).
It is important that
quantitative information is available:\\[2.5mm]
Let $E$, $F$ and $H$
be ultrametric Banach spaces.
Assume that the series\linebreak
corresponding to
$(f_k)_{k\in \N_0}\in \prod_{k\in \N_0} \Pol^k(E,F)$
has radius of strict convergence $\rho_1>0$
and the series corresponding to
$(g_k)_{k\in \N_0}\in \prod_{k\in \N_0} \Pol^k(F,H)$
has radius of strict convergence $\rho_2>0$.
Let $f\colon B_{\rho_1}^E(0)\to F$ and
$g\colon B^{F}_{\rho_2}(0)
\to H$ be the corresponding analytic
functions. We assume that $f(0)\in B^{F}_{\rho_2}(0)$
and choose $r\in\; ]0,\rho_1]$ such that
\[
\sup\{\|f_k\| r^k\colon k\in \N\}\; \leq \; \rho_2\,.
\]
Then \cite[4.1.5]{Bo1} ensures:
%
\begin{la}\label{quanticomp}
There exists
$(h_k)_{k\in \N_0}\in \prod_{k\in \N_0} \Pol^k(E,H)$
such that $\sum_{k=0}^\infty h_k$ has radius of strict convergence
at least $r$, and such that
\[
g(f(z))=\sum_{k=0}^\infty h_k(z)\quad \mbox{for all $\, z\in B_r^E(0)$.}
\]
\end{la}
{\bf Ultrametric inverse function theorem.}
The domination principle (\ref{domi}) implies that the inverse function theorem
over ultrametric fields is much nicer than its real counterpart.
To formulate the theorem, let us define
\[
\Lip(f):=\sup\left\{\frac{\|f(y)-f(z)\|_F}{\|y-z\|_E}
\colon y\not= z\in U\right\}\in [0,\infty]
\]
if
$E$ and $F$ a ultrametric Banach
spaces and $f\colon U\to F$ is a function on a subset $U\sub E$.
The function $f$ is called (globally) \emph{Lipschitz} if $\Lip(f)<\infty$.
If $U$ is open, $f$ is analytic and $x\in U$, we write $f'(x)\colon E\to F$ for the total differential of $f$ at $x$.
The next fact combines \cite[Thm.\,5.8]{FIO} and \cite[5.7.6]{Bo1}.
%
\begin{thm}[Ultrametric Inverse Function Theorem]\label{IFT}
Let $(E,\|.\|)$ be an ultrametric Banach space over a complete ultrametric field,
$x\in E$, $r>0$ and $f\colon B^E_r(x)\to E$ be an analytic map.
Let $A \in \GL(E)$ and assume
that the function $\wt{f}\colon B^E_r(x)\to E$ determined by
\[
f(y)=f(x)+A .(y-x)+\wt{f}(y)
\]
is Lipschitz, with
%
\begin{equation}\label{lipdomi}
\Lip(\wt{f})\; <\; \frac{1}{\|A^{-1}\|}\,.
\end{equation}
Then the following holds:
\begin{itemize}
\item[\rm(a)]
$f(B^E_r(x))$ is open, $f$ is injective and $f^{-1}\colon f(B^E_r(x))\to E$ is analytic.
\item[\rm(b)]
$f(B^E_s(y))=f(y)+A.B^E_s(0)$, for all $y\in B^E_r(x)$ and $s\in \;]0,r]$.\,\Punkt
\end{itemize}
\end{thm}
%
%
\begin{rem}\label{remstrict}
\begin{itemize}
\item[(a)]
Condition (b) in Theorem~\ref{IFT} means that $f$ behaves on balls like an affine-linear map.
\item[(b)]
$\frac{1}{\|A^{-1}\|}$
can be interpreted as an expansion factor,
in the sense that $\|A y\|\geq \frac{1}{\|A^{-1}\|}\|y\|$ for all $y\in E$.
\item[(c)] Condition (\ref{lipdomi}) means
that the remainder term $\wt{f}$
is dominated by the linear map~$A$.
\item[(d)]
Condition (\ref{lipdomi}) is automatically satisfied if we take $A:=f'(x)$
and choose $r>0$
small enough, since the analytic map $f$ is ``strictly differentiable''
at $x$ and thus $\lim_{s\to 0}\Lip(\wt{f}|_{B^E_s(x)})=0$
(see 4.2.3 and 3.2.4 in \cite{Bo1}).
\end{itemize}
\end{rem}
%
%
\begin{rem}\label{behaveba}
Let $f\colon B^E_r(0)\to F$ be analytic, with $f(0)=0$.
\begin{itemize}
\item[(a)]
If $\|f'(0)\|\leq a$, then Remark \ref{remstrict}\,(d)
and (\ref{domi}) imply that $f(B_s^E(0))\sub B^F_{as}(0)$
for all sufficiently small $s>0$.
\item[(b)]
In particular, $f(B_s^E(0))\sub B^F_s(0)$ for small $s>0$ if $\|f'(0)\|\leq 1$.
\item[(c)]
If $E=F$ and $f'(0)$ is a surjective isometry, then $f(B^E_s(0))=B^E_s(0)$
for small $s>0$
and $f|_{B_s^E(0)}$ is an isometry,
by (\ref{domi}), Remark~\ref{remstrict}\,(d)
and Theorem~\ref{IFT}\,(b).\footnote{In fact,
this holds for all $s\in \;]0,r]$ such that $\Lip(\wt{f}|_{B^E_s(0)})<1$.} 
\end{itemize}
\end{rem}
{\bf Manifolds and Lie groups.}
An \emph{analytic manifold} modelled on an
ultrametric Banach space $E$ over a complete ultrametric field $\K$ is defined
as usual (as a Hausdorff topological
space~$M$, together with a (maximal) set $\cA$ of homeomorphisms (``charts'')
$\phi\colon U_\phi\to V_\phi$
from open subsets of~$M$\linebreak
onto open subsets of~$E$,
such that $M=\bigcup_{\phi\in \cA}U_\phi$
and the mappings\linebreak
$\phi\circ \psi^{-1}$
are analytic for all $\phi,\psi\in \cA$).
Also the tangent space $T_pM$ of $M$ at $p\in M$,
analytic maps $f\colon M\to N$ between analytic manifolds,
and the tangent maps $T_pf\colon T_pM\to T_{f(p)}N$
can be defined as usual (cf.\ \cite{Bo1}),
as well as the tangent bundle $TM$
and $Tf\colon TM\to TN$.
If $f\colon M\to E$ is\linebreak
an analytic map to a Banach space,
we write $df$ for the second component of $Tf\colon TM\to TE\isom E\times E$.
An \emph{analytic Lie group}~$G$ over $\K$
is a group, equipped with an analytic manifold structure
modelled on an ultrametric\linebreak
Banach space over $\K$,
such that the group inversion and group
multiplication are analytic (cf.\ \cite{Bo2}).
As usual, we write $L(G):=T_1(G)$
and $L(\alpha):=T_1(\alpha)$, if $\alpha\colon G\to H$
is an analytic homomorphism between analytic Lie groups.
Let $M$ be an analytic manifold modelled on an ultrametric Banach space~$E$.
A subset $N\sub M$ is called a \emph{submanifold} of
$M$ if there exists a complemented vector subspace $F$ of the modelling space
of $M$ such that
each point $p\in N$ is contained in
the domain~$U$ of some chart $\phi\colon U\to V$ of $M$
such that $\phi(N\cap U)=F\cap V$.
By contrast, an analytic manifold $N$ is called an \emph{immersed submanifold} of $M$
if $N\sub M$ as a set and the inclusion map $\iota \colon N\to M$
is an immersion.
Subgroups of Lie groups with analogous properties are called
\emph{Lie subgroups}
and \emph{immersed Lie subgroups}, respectively.
\section{Centre-stable manifolds}\label{seccs}
%
%
In this section, we prove the Ultrametric Centre-Stable
Manifold Theorem (Theorem \ref{csthm}),
and discuss related topics.
We first regard the local situation.\\[2.5mm]
Let $(E,\|.\|)$ be an ultrametric Banach space over a complete ultrametric field $(\K,|.|)$,
such that $E=E_1\oplus E_2$ as a topological vector space
with closed vector subspaces $E_1$ and $E_2$,
such that
%
\begin{equation}\label{thussup}
\|x+y\|=\max\{\|x\|,\|y\|\}\quad\mbox{for all $x\in E_1$ and $y\in E_2$.}
\end{equation}
Given $r>0$,
we have
$B_r^E(0)=B_r^{E_1}(0)\times B_r^{E_2}(0)$, by (\ref{thussup}).
Let $f=(f_1,f_2)\colon$ $B_r^E(0)\to E=E_1\oplus E_2$ be an analytic map such that
$f(0)=0$ and $f'(0)$ leaves $E_1$ and $E_2$ invariant.
Thus
\[
f'(0)=A\oplus B
\]
with certain continuous linear maps $A\colon E_1\to E_1$
and $B\colon E_2\to E_2$. Let $a\in \;]0,1]$.
We assume that
%
\begin{equation}\label{goodA}
\|A\|\leq a
\end{equation}
and we assume that there exists a right inverse $C\in \cL(E_2)$ to $B$
(i.e., $B\circ C=\id_{E_2}$) such that\footnote{We are only interested
in the case where $B$ is invertible, but this hypothesis is not needed
for the following construction.}
\begin{equation}\label{goodC}
\frac{1}{\|C \|} \, >  \, a \, .
\end{equation}
Then
\[
f(x,y)=(A x, B y)+\wt{f}(x,y)
\]
determines an analytic map
$\wt{f}=(\wt{f}_1,\wt{f}_2) \colon B_r^E(0)\to E$
such that $\wt{f}(0)=0$
and $\wt{f}\, '(0)=0$.
After shrinking $r$, we may assume
that $\wt{f}$ is Lipschitz with
%
\begin{equation}\label{ftlip}
\Lip(\wt{f})<a
\end{equation}
(see Remark \ref{remstrict}\,(d)), and that
\[
\wt{f}_1(x,y)=\sum_{k=2}^\infty a_k(x,y)\quad\mbox{ and }
\quad
\wt{f}_2(x,y)=\sum_{k=2}^\infty b_k(x,y)
\]
for all $x\in B^{E_1}_r(0)$ and $y\in B^{E_2}_r(0)$,
for suitable continuous homogeneous polynomials
$a_k\colon E\to E_1$ and $b_k\colon E\to E_2$ of degree $k$
such that
\[
\lim_{k\to\infty} \|a_k\|r^k=0\quad\mbox{ and }\quad
\lim_{k\to\infty}\|b_k\|r^k=0\,.
\]
After replacing $f(x)$ by $\lambda^{-1} f(\lambda x)$ with $0\not=\lambda\in \K$
sufficiently small, we can achieve that
\[
\|a_k\|,\|b_k\|< 1\quad\mbox{for all $k\geq 2$.}
\]
After decreasing $r$, we may assume that
\begin{equation}\label{smr}
r\leq 1\,.
\end{equation}
Then the following holds:
%
%
\begin{prop}\label{thmloccs}
There exists an
analytic function $\phi \colon B^{E_1}_{ar}(0)\to E_2$
with the following properties:
\begin{itemize}
\item[\rm(a)]
$\phi(B^{E_1}_{ar}(0))\sub B^{E_2}_{ar}(0)$
and the graph of $\phi$ is $f$-invariant, more precisely
%
\begin{equation}\label{locinvgs}
f(\Gamma_s)\sub \Gamma_{a s}\sub \Gamma_s
\quad \mbox{for all $s\in \;]0, a r]$,}
\end{equation}
where
$\Gamma_s:=\{(x,\phi(x))\colon x\in B^{E_1}_s(0)\}$
for $s\in \;]0,ar]$;
\item[\rm(b)]
$\phi(0)=0$ and $\phi'(0)=0$
$($whence $T_{(0,0)}(\Gamma_{a r})=E_1)$; and
\item[\rm(c)]
There are continuous homogeneous polynomials
$c_k\colon E_1\to E_2$ of
degree~$k$ with
$\|c_k\| < a^{1-k}$
and $\phi(x)=\sum_{k=2}^\infty c_k(x)$ for all~$x\in B^{E_1}_{a r}(0)$.
\end{itemize}
If $B$ is invertible, then $\phi$ is uniquely determined.
\end{prop}
\begin{proof}
For all integers $k\geq 2$, we choose
$\alpha_k\in \cL^k(E,E_1)$
and $\beta_k\in \cL^k(E,E_2)$ such that $a_k=\alpha_k\circ\Delta^E_k$,
$b_k=\beta_k\circ\Delta^E_k$, and $\|\alpha_k\|$, $\|\beta_k\|<1$.\\[2.5mm]
If $\phi$ is an analytic function of the form described in (c),
then
\[
\sup\{\|c_k\|(a r)^k\colon k\geq 2\}\; \leq \; r
\]
(using (\ref{smr}))
and $\phi(0)=0$. Hence $f(x,\phi(x))$ is defined
for all $x\in B^{E_1}_{a r}(0)$
and given globally by its Taylor series around $0$
(by Lemma \ref{quanticomp}).
Now let $x\in B^{E_1}_s(0)$, where $s\in \;]0,a r]$.
Then
\[
\|\phi(x)\|\,\leq\, \|x\|\,,
\]
since $\|c_k(x)\|\leq \|c_k\|\cdot\|x\|^k\leq a^{1-k}\|x\|^k
=\big(\frac{\|x\|}{a}\big)^{k-1}\|x\|\leq \|x\|$.
Hence $\phi(B^{E_1}_{ar}(0))\sub B_{ar}^{E_2}(0)$.
Moreover,
$\|f_1(x,\phi(x))\|=\|Ax+\wt{f}_1(x,\phi(x))\|\leq a \|x\|< a s\leq s$
(using~(\ref{ftlip})),
and hence
%
\begin{eqnarray}
f(x,\phi(x))\in \Gamma_s &\aeq & f(x,\phi(x))=\big(f_1(x,\phi(x)),\, \phi(f_1(x,\phi(x)))\big)\notag\\
&\aeq& f_2(x,\phi(x)) = \phi(f_1(x,\phi(x)))\,.\label{fstimpl}
\end{eqnarray}
We mention that also the right hand side of (\ref{fstimpl})
is given on all of $B_{ar}^{E_1}(0)$ by its Taylor series around $0$,
because the homogeneous polynomials $\eta_j$
of the Taylor series of $f_1\circ (\id,\phi)$ around $0$
vanish if $j=0$ and
have norm $\|\eta_j\|\leq a^{2-j}$ if $j\geq 1$
(as will be verified in (\ref{verifeta})),
whence $\|\eta_j\| (ar)^j\leq  ar$
and so Lemma~\ref{quanticomp} applies.
Hence the validity of~(\ref{fstimpl}) for all $x\in B^{E_1}_{ar}(0)$
is equivalent to an identity of formal series:
%
\begin{eqnarray}
\lefteqn{B(c_2(x)+c_3(x)+\cdots)+b_2(x,c_2(x)+\cdots)
+b_3(x,c_2(x)+\cdots)+\cdots}\notag \\
&=&
c_2\big(Ax+a_2(x,c_2(x)+c_3(x)+\cdots)
+a_3(x,c_2(x)+\cdots)+\cdots\big)\notag\\
& + & 
c_3\big(Ax+a_2(x,c_2(x)+c_3(x)+\cdots)
+a_3(x,c_2(x)+\cdots)+\cdots\big)\notag\\
& + & \cdots \label{huge}
\end{eqnarray}
Comparing the lowest order term (of second order) on both sides,
we see that
\[
Bc_2(x) + b_2(x,0)=c_2(Ax)
\]
is required, which can be rewritten as $(B_*-A^*)(c_2)=-b_2(\sbull,0)$ or
%
\begin{equation}\label{wahlp}
B_*(\id -C_*A^*)(c_2)=-b_2(\sbull,0)\, ,
\end{equation}
writing
$A^*:=\Pol^2(A,E_2)$,
$B_*:=\Pol^2(E_1,B)$
and
$C_*:=\Pol^2(E_1,C)$
as in Lemma \ref{pbpfw}.
Since $\|C_*A^*\|\leq \|C\|\cdot \|A\|^2\leq \|C\|\cdot \|A\| <1$
by Lemma \ref{pbpfw}, using~\ref{lininv} it follows 
that $\id-C_*A^*$ is invertible
and $\|(\id - C_* A^* )^{-1}\| \leq 1$.
Thus
\[
c_2:=(\id -C_*A^*)^{-1} C_* (-b_2(\sbull,0))\in \Pol^2(E_1,E_2)
\]
has norm $\|c_2\|\leq \|C_*\|\cdot\|b_2(\sbull,0)\|\leq \|C\|<  \frac{1}{a}$.
Moreover, (\ref{wahlp}) holds for this choice of~$c_2$,
and if $B$ is invertible, then~$c_2$ is determined by (\ref{wahlp}).\\[2.5mm]
Let $n\geq 3$ now and, by induction, suppose that we have already found
$c_k\in\Pol^k(E_1,E_2)$ for $k=2,\ldots, n-1$
such that $\|c_k\| < a^{1-k}$
and~(\ref{huge})
holds up to order $n-1$ 
if these $c_2,\ldots, c_{n-1}$ are used
(and that these are unique if $B$ is invertible).
For $k=2,\ldots, n-1$, let
 $\gamma_k\in \cL^k(E_1,E_2)$
such that $c_k=\gamma_k\circ\Delta^{E_1}_k$
and $\|\gamma_k\|<a^{1-k}$.
Define 
$\gamma_1(x):=x$ for $x\in E_1$.
Identifying $E_1$ with the vector subspace
$E_1\times \{0\}$ of $E$
and $E_2$ with $\{0\}\times E_2$,
the previous maps take $x$ to $(0,\gamma_k(x))$ and $(x,0)$,
respectively.\\[2.5mm]
Define  $A^*:=\Pol^n(A,E_2)$,
$B_*:=\Pol^n(E_1,B)$ and
$C_*:=\Pol^n(E_1,C)$.
The $n$-th order term of (\ref{huge}) then reads
%
\begin{equation}\label{lesshuge}
B_*(c_n) + r_n= A^* (c_n)+ s_n
\end{equation}
with
%
\begin{equation}\label{fsterr}
r_n=\sum_{k=2}^n \sum_{\stackrel{\scriptstyle j_1,\ldots, j_k\in \N}{j_1+\cdots+ j_k=n}}
\beta_k\circ(\gamma_{j_1},\ldots, \gamma_{j_k})
\end{equation}
and
%
\begin{equation}\label{secerr}
s_n=\sum_{k=2}^{n-1} 
\sum_{\stackrel{\scriptstyle j_1,\ldots, j_k\in \N}{j_1+\cdots+ j_k=n}}
\gamma_k\circ (\eta_{j_1},\ldots, \eta_{j_k})\,,
\end{equation}
where $\eta_1:=A$ and
%
\begin{equation}\label{simfml}
\eta_j=\sum_{\ell=2}^j \sum_{\stackrel{\scriptstyle i_1,\ldots, i_\ell \in \N}{i_1+\cdots+ i_\ell=j}}
\alpha_\ell\circ (\gamma_{i_1},\ldots, \gamma_{i_\ell})
\end{equation}
for $j=2,\ldots, n-1$.
For $j_1,\ldots, j_k$ as
in (\ref{fsterr}), we have
\[
\|\beta_k\circ(\gamma_{j_1}\times \cdots\times \gamma_{j_k})\|
\leq \|\beta_k\| \cdot a^{1-j_1}\cdots a^{1-j_k}
< a^{k-n}\leq a^{2-n}\,.
\]
Since
$\beta_k\circ(\gamma_{j_1},\ldots, \gamma_{j_k})=
\beta_k\circ(\gamma_{j_1}\times \cdots \times \gamma_{j_k})\circ \Delta^{E_1}_n$,
we conclude that
%
\begin{equation}\label{fibas}
\| r_n\| \, < \, a^{2-n}\,.
\end{equation}
Likewise, the norm of each summand in (\ref{simfml})
is $<a^{\ell-j}\leq a^{2-j}$, and thus
\begin{equation}\label{verifeta}
\|\eta_j\|\; <\;  a^{2-j}\, .
\end{equation}
Consequently, the norm of each summand in (\ref{secerr})
is at most $\|\gamma_k\|\cdot a^{2k-n}
< a^{1-k} a^{2k-n}=a^{k-n+1}\leq a^{2-n}$.
Therefore,
%
\begin{equation}\label{secbas}
\|s_n\|\, < \, a^{2-n}\,.
\end{equation}
In view of Lemma \ref{pbpfw},
(\ref{nonexpa}), (\ref{fibas}) and (\ref{secbas}),
\[
c_n:=(\id -C_*A^*)^{-1} C_* (s_n-r_n)\in \Pol^n(E_1,E_2)
\]
is a solution to
\[
(B_*-A^*)(c_n)=B_*(\id-C_*A^*)(c_n)=s_n-r_n
\]
(an hence to (\ref{lesshuge})),
of norm $\|c_n\|\leq \|C\|\cdot a^{2-n}  <  a^{1-n}$.
If $B$ is invertible, then  
$B_*-A^*=B_*(\id-(B^{-1})_*A^*)$ is invertible,
entailing that $c_n$ is uniquely determined by (\ref{lesshuge}).
\end{proof}
$\;$\vfill\pagebreak

\noindent
{\bf Proof of Theorem \ref{csthm}.}
\begin{numba}\label{reusable}
Let $E:=T_p(M)$ and $\kappa\colon P\to U\sub E$ be a chart of $M_0$ around~$p$
such that $\kappa(p)=0$ and $d\kappa(0)=\id_E$.
Let $Q\sub P$ be an
open neighbourhood of~$p$ such that $f(Q)\sub P$;
after shrinking $Q$, we may assume that $\kappa(Q)=B_r^E(0)$ for some $r>0$.
Then $g:=\kappa\circ f|_Q\circ\kappa^{-1}|_{B_r^E(0)}\colon B^E_r(0)\to E$
expresses $f|_Q$ in the local chart $\kappa$.\vspace{-1.3mm}
\end{numba}

(a) Let $E=E_1\oplus E_2$, with the norm $\|.\|$,
be the decomposition of $E$ into
an $a$-centre-stable subspace~$E_1$
and an $a$-unstable subspace~$E_2$
with respect to $\alpha:=T_p(f)=g'(0)$
(as in Definition~\ref{defcsub}).
Applying Proposition \ref{thmloccs} to $g$ (instead of $f$),
we see that, possibly after shrinking $r$,
there is an analytic map \mbox{$\phi\colon B^{E_1}_{ar}(0)\to B^{E_2}_{ar}(0)$}
as described there.
Then the graph $\Gamma_{a r}$ of $\phi$ is a submanifold of $B_r^E(0)$
tangent to $E_1$ at~$0$,
and hence $N:=\kappa^{-1}(\Gamma_{a r})$
is a submanifold of $Q$ (and hence of $M_0$)
tangent to $T_p(M)_{a,\obcs}$ at~$p$.
Now $g(\Gamma_{ar}) \sub \Gamma_{ar}$,
where $\Gamma_{ar}$ is a submanifold
and~$g$ is analytic.
Hence $g$ restricts to an analytic
map $\Gamma_{ar}\to \Gamma_{ar}$. Thus~$f$
restricts to an analytic map $N\to N$.\vspace{1mm}

(c) Let $\mu\colon V\to B^{E_1}_\tau(0)$
be a chart of~$N$ around~$p$ such that
$\mu(p)=0$ and $d\mu(p)=\id_{E_1}$.
There exists $\sigma\in \;]0,\tau]$
such that $h:=\mu \circ f\circ \mu^{-1}$
is defined on all of
$B^{E_1}_\sigma(0)$.
Since $h'(0)=T_p(f|_N)=:A$
with $\|A\|\leq a$,
Remark~\ref{behaveba}\,(a)
shows that $h(B^{E_1}_s(0))\sub
B^{E_1}_{as}(0)
\sub B^{E_1}_s(0)$ for all $s\in \;]0,\sigma]$,
after possibly shrinking~$\sigma$.
Moreover,
we may assume that
$\Omega_s:=\mu^{-1}(B^{E_1}_s(0))$
is a submanifold of~$M$ for each $s\in \;]0,\sigma]$
(after shrinking~$\sigma$ further if necessary).
Then the sets $\Omega_s$ with $s\in \;]0,\sigma]$
form a basis of open neighbourhoods of~$p$ in~$N$,
and each of them is an $a$-centre-stable
manifold and a submanifold of~$M$.\vspace{1mm}

(b) Let $N$ (and other notation) be as in the proof of~(a)
and $N_1$ be any $a$-centre-stable
manifold. Write
$g=(g_1,g_2)=g'(0)+\wt{g}\colon B^E_r(0)\to E_1\oplus E_2$,
where $\Lip(\wt{g})<a$
and $g'(0)=A\oplus B$ with $\|A\|\leq a$ and $\frac{1}{\|B^{-1}\|}>a$.
Then $P\cap N_1$ is an immersed
submanifold of~$P$ tangent to~$E_1$
and hence, after replacing~$N_1$ by an open subset of~$N_1$
(justified by~(c)),
we may assume that~$N_1$ is a submanifold of~$P$.
Since~$\kappa(N_1)$ is tangent to~$E_1$ at $0\in E$,
the inverse function theorem implies
that $\kappa(N_1)$ is the graph
of an analytic map $\psi\colon W\to E_2$
on some open $0$-neighbourhood
$W\sub B^{E_1}_r(0)$,
with $\psi(0)=0$ and $\psi'(0)=0$
(after shrinking~$N_1$ if necessary).
By Remark \ref{remstrict}\,(d),
we may assume that $\Lip(\psi)\leq 1$.
Then $\mu:=\pr_1\circ \kappa|_{N_1}$ is a chart
for~$N_1$ such that $\mu(0)=0$ and
$d\mu(p)=\id_{E_1}$
(where $\pr_1\colon E_1\oplus E_2\to E_1$).
Hence, by the proof of~(c), there exists $\sigma\in \;]0,r]$
such that $B^{E_1}_\sigma(0)\sub W$
\[
g(\Theta_s)\;\sub\; \Theta_{as}\quad \mbox{for all $\, s\in \;]0,\sigma]$,}
\]
where
$\Theta_s:=\{(x,\psi(x))\colon x\in B^{E_1}_s(0)\}$
for $s\in\;]0,\sigma]$.
After shrinking~$\sigma$
and conjugation with a homothety if necessary,
we may assume that~$\psi$
is given globally by a power series
and satisfies conditions analogous to~(b) and~(c) in Proposition~\ref{thmloccs}.
After replacing~$r$ and~$\sigma$
with $\min\{r,\sigma\}$,
we may assume that $r=\sigma$.
Then $\phi=\psi|_{B^{E_1}_{ar}(0)}$
(by the uniqueness part of Proposition \ref{thmloccs}),
and hence $N$ is an open submanifold
of~$N_1$.\,\Punkt

\begin{rem}\label{unicst}
We mention that $E_{s,\obcs}$ in Definition \ref{defcsub}
is uniquely determined.
In fact, $E_{a,\obcs}$ is the set of all $x\in E$ such that
$a^{-n}\|\alpha^n(x)\|$ is bounded
for the specified (and hence any compatible) norm.
If~$\alpha$ is invertible, then also~$E_{a,\obu}$
is unique, because $E_{a,\obu}$ is the $c$-centre-stable subspace with respect
to $\alpha^{-1}$ in this case, for each $c\in\;]\|\alpha_2^{-1}\|, \frac{1}{a}[$.
Also note that $E_{a,\obcs}=E_{b,\obcs}$  and $E_{a,\obu}=E_{b,\obu}$
for all positive real numbers $b\in [\|\alpha_1\|, \frac{1}{\|\alpha_2^{-1}\|}[$.
\end{rem}
\section{Centre  manifolds}\label{secc}
%
In this section, we prove the Ultrametric Centre Manifold Theorem.\\[2.5mm]
{\bf Proof of Theorem \ref{cthm}.}
Since $E:=T_p(M)$ admits a centre subspace with respect to $T_p(f)$,
we have a decomposition $E=E_\obs \oplus E_\obc \oplus E_\obu$
and corresponding norm $\|.\|$.\vspace{1mm}

(c) Let $N\sub M$ be a centre manifold around $p$ with respect to $f$.
Or, more
generally (for use in the proof of (a)),
let $N\sub M$ be an immersed submanifold\linebreak
containing $p$ which is
tangent to $T_p(M)_\obc=E_\obc$ at $p$,
and assume that $p$ has an open neighbourhood $P\sub N$
such that $f(P)\sub N$ and $f|_P\colon P\to N$ is analytic.
Let $W\sub N$ be a given neighbourhood of~$p$.
We choose a chart
$\mu\colon V\to B^{E_\obc}_\tau(0)$
of~$N$ around~$p$ such that $V\sub W\cap P$,
$\mu(p)=0$ and $d\mu(p)=\id_{E_\obc}$.
There exists $\sigma\in \;]0,\tau]$
such that $f(\mu^{-1}(B^{E_\obc}_\sigma(0)))\sub V$.
Then $h:=\mu \circ f\circ \mu^{-1}$
defines an analytic map
$B^{E_\obc}_\sigma(0)\to E_\obc$, such that
$h'(0)=T_p(f)|_{E_\obc}$ is a surjective linear isometry.
Remark~\ref{behaveba}\,(c)
shows that $h(B^{E_\obc}_s(0))=
B^{E_\obc}_s(0)$
for all $s\in \;]0,\sigma]$
and $h|_{B^{E_\obc}_s(0)}\colon
B^{E_\obc}_s(0)\to B^{E_\obc}_s(0)$
is an analytic diffeomorphism,
after possibly shrinking~$\sigma$.
By the same token,
we may assume that
$\Omega_s:=\mu^{-1}(B^{E_\obc}_s(0))$
is a submanifold of~$M$ for each $s\in \;]0,\sigma]$.
Then the sets $\Omega_s$ with $s\in \;]0,\sigma]$
form a basis of open neighbourhoods of~$p$ in~$N$,
and each of them is a centre
manifold around $p$ with respect to $f$,
a submanifold of~$M$, stable under $f$,
and $f|_{\Omega_s}\colon \Omega_s\to \Omega_s$
is a diffeomorphism.\vspace{1mm}

(a) After shrinking $M$ and $M_0$, we may assume that $M_1:=f(M_0)$ is open and $f$
an analytic diffeomorphism onto $M_1$.
Observe that $E_\obs\oplus E_\obc$ is a centre-stable
subspace with respect to
$T_p(f)$, and
$E_\obc\oplus E_\obu$ is a centre-stable
subspace with respect to
$T_p(f^{-1})=T_p(f)^{-1}$.
Hence Theorem \ref{csthm}
provides a centre-stable manifold $N_1$
around $p$ with respect to $f$, and a centre-stable manifold $N_2$
around $p$ with respect to $f^{-1}\colon M_1\to M$, which are submanifolds of $M$.
Since $N_1$ and $N_2$ are transversal at $p$,
after shrinking $N_2$, we may assume that $N_1\cap N_2$
is a submanifold of $N_1$ and hence of $M$
(retaining that $N_2$ is a centre-stable manifold
by means of Theorem \ref{csthm}\,(c)).
After shrinking $N_2$ further if necessary,
we may assume that $S:=f^{-1}(N_2)$ is open in $N_2$
and that $f^{-1}|_{N_2}\colon N_2\to S$ is a diffeomorphism
(by the inverse function theorem
and Theorem \ref{csthm}\,(c)).
Define $P:=N_1\cap S$.
Then $f(P)\sub N_1\cap N_2$
and since $N:=N_1\cap N_2$ is a submanifold of $M_0$,
the restriction of $f$ to a map $h\colon P \to N$
is analytic. By the proof of (c),
there exists an open subset $\Omega\sub N$
which is a centre manifold around $p$ with respect to $f$.\vspace{1mm}

(b) can be proved like Theorem \ref{csthm} (b)
once we know that also centre manifolds
can be described as the graph of a unique power series
in a given chart (the simple adaptation of the argument is left to
the reader).
The following proposition
shows that such a description is
always possible.\,\vspace{3mm}\Punkt

\noindent
We first fix the setting.
\begin{numba}\label{thesetaltc}
We consider an ultrametric Banach space $E$ over a complete
ultrametric field $(\K,|.|)$
and an automorphism $\alpha\colon E\to E$
for which there exists a centre subspace $E_2\sub E$.
We let $E_1$ be the stable subspace,
$E_3$ be the unstable subspace and $\|.\|$ be an ultrametric norm
on $E=E_1\oplus E_2\oplus E_3$ as described in Definition~\ref{defacentre}.
Thus $\alpha=A  \oplus B \oplus C$ in terms of automorphisms
of $E_1$, $E_2$ and $E_3$, respectively.
Let $r>0$ and $f=(f_1,f_2, f_3)\colon B_r^E(0)\to E$ be an analytic mapping such that
$f(0)=0$ and $f'(0)=\alpha$.
Thus
\[
f(x,y,z)=(A x, B y, C z)+\wt{f}(x,y, z)
\]
with an analytic function
$\wt{f}=(\wt{f}_1,\wt{f}_2,\wt{f}_3) \colon B_r^E(0)\to E$
such that $\wt{f}(0)=0$
and $\wt{f}\, '(0)=0$.
After shrinking $r$, we may assume
that $\wt{f}$ is Lipschitz with
%
\begin{equation}\label{ftlipbb}
\Lip(\wt{f})<1 \,.
\end{equation}
We may also assume that $\wt{f}_1(x,y,z)= \sum_{k=2}^\infty a_k(x,y,z)$,
\[
\wt{f}_2(x,y,z)\; =\; \sum_{k=2}^\infty b_k(x,y,z)
\quad
\mbox{ and }\quad \wt{f}_3(x,y,z) =\sum_{k=2}^\infty c_k(x,y,z)
\]
for all $x\in B^{E_1}_r(0)$, $y\in B^{E_2}_r(0)$
and $z\in B^{E_3}_r(0)$,
for suitable continuous homogeneous polynomials
$a_k \colon E\to E_1$,
$b_k\colon E\to E_2$
and $c_k \colon E\to E_3$
of degree $k$
such that
\[
\lim_{k\to\infty} \|a_k\|r^k=0\,,\quad
\lim_{k\to\infty}\|b_k\|r^k=0
\quad\mbox{ and }\quad
\lim_{k\to\infty} \|c_k\|r^k=0\,.
\]
After replacing $f(x)$ by $\lambda^{-1} f(\lambda x)$ with $0\not=\lambda\in \K$
sufficiently small, we can achieve that
\[
\|a_k\|,\|b_k\|, \|c_k\| < 1\quad\mbox{for all $\, k\geq 2$.}
\]
After decreasing $r$, we may assume that
$r\leq 1$. Then the following holds:
\end{numba}
%
%
%
%
\begin{prop}\label{thmloccsbb}
In the setting of {\rm\ref{thesetaltc}},
there exists a unique
analytic function
$\phi=(\phi_1,\phi_3)  \colon B^{E_2}_r(0)\to E_1\times E_3$
with properties {\rm(a)--(c)}:
\begin{itemize}
\item[\rm(a)]
$f(\Gamma_t)\sub \Gamma_r$
for some $t\in\;]0,r]$,
where
$\{(\phi_1(y),y,\phi_3(y))\colon y\in B^{E_2}_s(0)\}$ $=:\Gamma_s$
for $s\in \;]0, r]$;
\item[\rm(b)]
$\phi(0)=0$ and $\phi'(0)=0$
$($whence $T_0(\Gamma_r)=E_2)$; and
\item[\rm(c)]
There exist continuous homogeneous polynomials
$d_k\colon E_2\to E_1$ and $e_k \colon \!E_2\!\to \!E_3$
of degree $k$ with
$\|d_k\|, \!\|e_k\|\! < \!1$
and $\phi(y)\!=\!\sum_{k=2}^\infty (d_k(y),e_k(y))$ for all
$y\in B^{E_2}_r(0)$.
\end{itemize}
Moreover, $\phi(B_s^{E_1}(0))\sub B^{E_1\times E_3}_s(0)$
and $f(\Gamma_s) = \Gamma_s$,
for all $s\in \;]0,r]$.
\end{prop}
Because the proof of Proposition~\ref{thmloccsbb}
is very similar to that of Proposition~\ref{thmloccs},
we relegate it to an appendix (Appendix \ref{sec1app}).
Note that Proposition~\ref{thmloccsbb}
also provides a second (more involved)
proof for the existence of centre manifolds.
The above transversality argument
can be re-used nicely in the $C^k$-case.
%
%
\begin{rem}\label{udecen}
$E_{a,\obs}$, $E_{a,\obc}$ and $E_{a,\obu}$
are uniquely determined
in the situation of Definition~\ref{defacentre}
(if they exist).
In fact, $E_{a,\obs}$ (resp., $E_{a,\obu}$) is the set of all $x\in E$ such that
$a^{-n}\|\alpha^n(x)\|\to 0$
as $n\to\infty$ (resp., as $n\to-\infty$),
and $E_{a,\obc}$ is the set of all $x\in E$ such that
$\{a^{-n}\|\alpha^n(x)\|\colon n\in \Z\}$
is bounded.
\end{rem}
\section{Mappings between sequence spaces}\label{secseq}
%
%
%
%
%
We now prove some basic facts concerning
Banach spaces of
sequences and non-linear mappings
between them. They will be used later
(in Section~\ref{secirw})
to construct local $a$-stable manifolds.\\[2.5mm]
For the moment, let $(\K,|.|)$ be a valued field.
%
\begin{defn}\label{defSaE}
If $(E,\|.\|)$ is a normed space
over $\K$ and $a$ a positive real number,
we let
$\cS_a(E)$
be the set of all sequences
$(x_n)_{n\in \N_0}$ in $E$ such that
$\lim_{n\to\infty} a^{-n}\|x_n\| =0$.
Clearly $\cS_a(E)$ is a vector subspace of
$E^{\N_0}$, and
\[
\|x\|_a\, :=\, \max\{a^{-n}\|x_n\|\colon n\in \N_0\}\quad
\mbox{for $\, x=(x_n)_{n\in \N_0}\in \cS_a(E)$}
\]
is a norm on $\cS_a(E)$.
Given a subset $U\sub E$,
we write
\[
\cS_a(U):=\{x=(x_n)_{n\in \N_0}\in \cS_a(E)\colon
(\forall n\in \N_0)\; x_n\in U\}\, .
\]
\end{defn}
\begin{rem}
The following assertions are obvious:
\begin{itemize}
\item[(a)]
If $|.|$ and $\|.\|$
are ultrametric, then also $\|.\|_a$
is ultrametric.
\item[(b)]
If $E$ is a Banach space, then also
$\cS_a(E)$ is a Banach space.
\item[(c)]
If $a\in \;]0,1]$,
then $\lim_{n\to\infty}x_n=0$ for all $x=(x_n)_{n\in \N_0}\in \cS_a(E)$.
\end{itemize}
\end{rem}
Our first lemma compiles various basic facts.
%
\begin{la}\label{basicsseq}
Let $(E,\|.\|_E)$ be a normed space over $(\K,|.|)$ and $a>0$.
\begin{itemize}
\item[\rm(a)]
The left shift $\lambda \colon
\cS_a(E)\to \cS_a(E)$, $\lambda(x):=(x_{n+1})_{n\in \N_0}$
for $x=(x_n)_{n\in \N_0}\in \cS_a(E)$,
and the right shift $\rho\colon \cS_a(E)\to\cS_a(E)$,
$\rho(x):=(0,x_0,x_1,\ldots)$
are continuous linear maps, of operator norm
%
\begin{equation}\label{opnormsh}
\|\lambda\|\leq a\quad \mbox{ and }\quad
\|\rho\|\leq a^{-1}\, .
\end{equation}
\item[\rm (b)]
For each $m\in \N_0$, the maps
$\pi_m\colon \cS_a(E)\to E$, $x=(x_n)_{n\in \N_0}\mto x_m$
and $\mu_m\colon E\to\cS_a(E)$,
$\mu_m(x):=(0,\ldots, 0,x,0,0,\ldots)$ $($with $m$ zeros at the beginning$)$
are continuous linear, of operator norm $\|\pi_m\|\leq a^m$
and $\|\mu_m\|\leq a^{-m}$.
\item[\rm(c)]
For each $m\in \N$, the map $\cS_a(E)\to E^m\times \cS_a(E)$,
\[
(x_n)_{n\in \N_0}\mto ((x_0,x_1,\ldots, x_{m-1}), (x_{m+n})_{n \in \N_0})
\]
is an isomorphism of topological vector spaces.
\item[\rm(d)]
If $a\in \;]0,1]$ and $U\sub E$ is an open $0$-neighbourhood,
then $\cS_a(U)$ is an open $0$-neighbourhood in $\cS_a(E)$.
\item[\rm(e)]
If also $(F,\|.\|_F)$ is a normed space over $\K$,
equip $E\oplus F$ $($which we treat as an internal direct sum$)$
with the maximum norm, $\|x+y\|:=\max\{\|x\|_E,\|y\|_F\}$
for $x\in E$, $y\in F$.
Then $\cS_a(E\oplus F)=\cS_a(E)\oplus \cS_a(F)$
and $\|x+y\|_a=\max\{\|x\|_a,\|y\|_a\}$
for all $x\in \cS_a(E)$ and $y\in \cS_a(F)$.
\end{itemize}
\end{la}
\begin{proof}
(a) Given $x=(x_n)_{n\in \N_0}\in \cS_a(E)$,
\[
a^{-n}\|\lambda(x)_n\|=a^{-n}\|x_{n+1}\|=
a\cdot a^{-(n+1)}\|x_{n+1}\|\leq a \|x\|_a
\]
for each $n\in \N_0$. Hence $\| \lambda(x)\|_a \leq a \|x\|_a$
and thus $\|\lambda\|\leq a$. The second assertion can be shown
analogously.

(b) For $x=(x_n)_{n\in \N_0}$, we have $\|\pi_m(x)\|_E=\|x_m\|_E=a^ma^{-m}\|x_m\|_E \leq a^m\|x\|_a$.
Hence $\|\pi_m\|\leq a^m$. If $x\in E$, then $\|\mu_m(x)\|_a=a^{-m}\|x\|_E$
and thus $\|\mu_m\|\leq a^{-m}$.

(c) Using the continuous mappings introduced
in (a) and (b), the map in question can be written as $\Phi=(\pi_0,\ldots,\pi_{m-1}, \lambda^m)$.
It therefore is continuous linear. For $j\in \{0,\ldots, m-1\}$, let
$\pr_j\colon E^m\times \cS_a(E)\to E$
be the projection onto the $j$-th component (which we count starting with $0$).
Moreover, let\linebreak
$\pr_m\colon
E^m\times \cS_a(E)\to \cS_a(E)$ be the projection onto the final component.
Then $\Psi:=\sum_{j=0}^{m-1} \mu_j\circ \pr_j + \rho^m\circ \pr_m$
is continuous linear (by (a) and (b)),
and it is easy to see that $\Phi$ and $\Psi$ are the inverse mappings of one
another.

(d) Let $x=(x_n)_{n\in \N_0}\in \cS_a(U)$.
There is $s>0$ such that $B^E_s(0)\sub U$, and
$m\in \N_0$ such that $a^{-n}\|x_n\|_E<s$
for all $n>m$.
Then
$\|(x_{m+1+n})_{n\in \N_0}\|_a = a^{m+1}\max\{a^{-n}\|x_n\|_E\colon n>m\}
\leq \max\{a^{-n}\|x_n\|_E\colon n>m\}<s$.
Hence
$U^{m+1} \times B^{\cS_a(E)}_s(0)$
is an open neighbourhood of~$x$,
identifying $\cS_a(E)$ with $E^{m+1}\times \cS_a(E)$ (as in (c)).
Since $U^{m+1}\times B_s^{\cS_a(E)}(0)\sub \cS_a(U)$ and~$x$
was arbitrary,
it follows that $\cS_a(U)$ is open.

(e) follows from the fact that the maximum of
$\max\{a^{-n}\|x_n\|_E\colon n\in \N_0 \}$ and $\max \{a^{-n}\|y_n\|_F\colon  n\in \N_0 \}$
coincides with the maximum of the numbers $a^{-n}\max\{\|x_n\|_E,\|y_n\|_F\}$,
for $n\in \N_0$.
\end{proof}
Various types of maps
operate on sequence spaces.
%
%
\begin{la}\label{lamapsop}
Let $(E,\|.\|_E)$ and $(F,\|.\|_F)$ be normed spaces over $\K$ and $a>0$.
\begin{itemize}
\item[\rm(a)]
If $U\sub E$ is a subset such that $0\in U$
and $f\colon U\to F$ is a Lipschitz map such that
$f(0)=0$, then
$\cS_a(f)(x):=(f(x_n))_{n\in \N_0}\in \cS_a(F)$
for all $x=(x_n)_{n\in \N_0}\in \cS_a(U)$,
and the map $\cS_a(f)\colon \cS_a(U)\to \cS_a(F)$
so obtained is Lipschitz,
with $\Lip(\cS_a(f)) \leq \Lip(f)$.
\item[\rm(b)]
If $\alpha \colon \!E\!\to \!F$ is continuous linear, then
$\cS_a(\alpha)(x)\!:=\!(\alpha(x_n))_{n\in \N_0}\in \cS_a(F)$
for all $x=(x_n)_{n\in \N_0}\in \cS_a(E)$,
and the map $\cS_a(\alpha)\colon \cS_a(E)\to \cS_a(F)$
so obtained is continuous linear,
of operator norm $\|\cS_a(\alpha)\|\leq \|\alpha\|$.
\item[\rm(c)]
If $a\in \;]0,1]$ and
$p\colon E\to F$ is a continuous homogeneous polynomial
of degree $k\in \N$, then
$\cS_a(p)(x):=(p(x_n))_{n\in \N_0}\in \cS_a(F)$
for all $x=(x_n)_{n\in \N_0}\in \cS_a(E)$,
and the map $\cS_a(p)\colon \cS_a(E)\to \cS_a(F)$
so obtained is a\linebreak
continuous homogeneous polynomial of degree~$k$, of norm $\|\cS_a(p)\|\!\leq\! \|p\|$.
\end{itemize}
\end{la}
\begin{proof}
(a) Given $x=(x_n)_{n\in \N_0}$ and $y=(y_n)_{n\in \N_0}$ in $\cS_a(U)$,
we have $a^{-n}\|f(x_n)-f(y_n)\|_F\leq a^{-n} \Lip(f)\|x_n-y_n\|_E$
for each $n\in \N_0$, showing that
$\cS_a(f)(x)-\cS_a(f)(y)\in \cS_a(F)$ and
$\|\cS_a(f)(x)-\cS_a(f)(y)\|_a\leq \Lip(f)\|x-y\|_a$.
Taking $y=0$, we obtain that $\cS_a(f)(x)\in \cS_a(F)$.
Thus $\cS_a(f)$ makes sense, and it is Lipschitz of constant $\leq \Lip(f)$
by the preceding estimate.

(b) By (a), $\cS_a(\alpha)$ makes sense.
Since, apparently,
the map $\cS_a(\alpha)$ is linear, its Lipschitz constant
(estimated in (a)) coincides with its operator norm.

(c) Let $\beta\in\cL^k(E,F)$ such that $p=\beta\circ\Delta^E_k$.
Write $\Delta_k$ for the diagonal map $\cS_a(E)\to\cS_a(E)^k$.
Given $x^1,\ldots, x^k\in \cS_a(E)$ with $x^j=(x^j_n)_{n\in \N_0}$,
we define $B(x^1,\ldots, x^k):=(\beta(x^1_n,\ldots,x^k_n))_{n\in\N_0}$.
Then
\begin{eqnarray}
a^{-n}\|\beta(x^1_n,\ldots,x^k_n)\|_F &\leq& a^{-n}\|\beta\|\cdot \|x^1_n\|_E\cdots\|x^k_n\|_E\notag\\
&\leq & \|\beta\|(a^{-n}\|x^1_n\|)\cdots (a^{-n}\|x^k_n\|)\label{newtg}\\
&\leq & \|\beta\|\cdot \|x^1\|_a\cdots\|x^k\|_a\,.\label{gvemulti}
\end{eqnarray}
Since the right hand side in (\ref{newtg}) tends to $0$ as $n\to \infty$,
we have $B(x^1,\ldots x^k)\in \cS_a(F)$.
Then~$B$ is $k$-linear and
$B \colon \cS_a(E)^k\to  \cS_a(F)$ has norm at most $\|\beta\|$,
because (\ref{gvemulti}) implies that
$\|B(x^1,\ldots, x^k)\|_a\leq \|\beta\|\cdot \|x^1\|_a\cdots\|x^k\|_a$.
Since $\cS_a(p)=B\circ \Delta_k$, the assertion~follows.
\end{proof}
%
%
%
\begin{prop}\label{propseqana}
Let $E$ and $F$ be ultrametric Banach spaces
over a complete ultrametric field $\K$
and $f\colon U\to F$ be an analytic mapping on an open
$0$-neighbourhood $U\sub E$,
such that $f(0)=0$.
Let $a\in \;]0,1]$.
Then
\[
\cS_a(f)(x):=(f(x_n))_{n\in \N_0}\in \cS_a(F)
\]
for all $x=(x_n)_{n\in \N_0}\in \cS_a(U)$,
and the map
\[
\cS_a(f)\colon \cS_a(U)\to \cS_a(F)
\]
so obtained is analytic.
\end{prop}
\begin{proof}
There are $s>0$ and polynomials
$p_k\in\Pol^k(E,F)$ for $k\in \N$ such that $\lim_{k\to\infty} \|p_k\|s^k=0$,
$B^E_s(0)\sub U$
and $f(x)=\sum_{k=1}^\infty p_k(x)$ for all $x\in B^E_s(0)$.
Then $\cS_a(p_k)\colon \cS_a(E)\to \cS_a(F)$ is a continuous
homogeneous polynomial for each $k\in \N$
(by Lemma \ref{lamapsop}\,(c)).
Since, moreover, $\|\cS_a(p_k)\|\leq \|p_k\|$,
we see that $\lim_{k\to\infty}\|\cS_a(p_k)\|s^k=0$.
Then $\sum_{k=1}^\infty \cS_a(p_k)$
defines an analytic function $B^{\cS_a(E)}_s(0)\to \cS_a(F)$
(by \cite[4.2.4]{Bo1}).
Since this function coincides with
\begin{equation}\label{auxilh}
h:=\cS_a(f)|_{B^{\cS_a(E)}_s(0)}\, ,
\end{equation}
we see that $\cS_a(f)$ is analytic on $B^{\cS_a(E)}_s(0)$.

Now let $x=(x_n)_{n\in \N_0}\in \cS_a(U)$ be arbitrary.
There is $m\in \N_0$ such that $a^{-n}\|x_n\|_E<s$
for all $n>m$.
Identifying $\cS_a(E)$ with $E^{m+1}\times \cS_a(E)$ (as in Lemma \ref{basicsseq}\,(c)),
we may consider $U^{m+1} \times B^{\cS_a(E)}_s(0)$
as an open neighbourhood of~$x$ in~$\cS_a(U)$.
Considered as a mapping
\[
U^{m+1} \times B^{\cS_a(E)}_s(0)\to F^{m+1}\times \cS_a(F)\isom\cS_a(F)\,,
\]
$\cS_a(f)$ coincides with the analytic map
$f\times\cdots \times f\times h$
(which involves $m+1$ factors $f$ at the beginning, and the map $h$ from (\ref{auxilh})).
Thus $\cS_a(f)$ is analytic on an open neighbourhood of $x$
and hence analytic (as $x$ was arbitrary).
\end{proof}
\section{Construction of local stable manifolds}\label{secirw}
%
%
We construct local stable manifolds
by an adaptation of Irwin's method.
%
\begin{numba}\label{thesitu}
Let $(E,\|.\|)$ be an ultrametric Banach space over a complete ultrametric field $(\K,|.|)$,
such that $E=E_1\oplus E_2$ as a topological vector space
with closed vector subspaces $E_1$ and $E_2$, and
such that
%
\begin{equation}\label{thussup2}
\|x+y\|=\max\{\|x\|,\|y\|\}\quad\mbox{for all $x\in E_1$, $y\in E_2$.}
\end{equation}
We interpret $E$ also as the direct product
$E_1\times E_2$, in which case its elements are written as pairs
$(x,y)$.
Given $r>0$,
we have
$B_r^E(0)=B_r^{E_1}(0)\times B_r^{E_2}(0)$, by (\ref{thussup2}).
Let $f=(f_1,f_2)\colon B_r^E(0)\to E=E_1\times E_2$
be an analytic map such that
$f(0)=0$ and $f'(0)$ leaves $E_1$ and $E_2$ invariant.
Thus
\[
f'(0)=A\oplus B
\]
with certain continuous linear maps $A\colon E_1\to E_1$
and $B\colon E_2\to E_2$. Let $a\in \;]0,1]$.
We assume that
%
\begin{equation}\label{goodA2}
\|A\|< a
\end{equation}
and we assume that $B$ is invertible, with
%
\begin{equation}\label{goodC2}
\frac{1}{\|B^{-1}\|}>a \,.
\end{equation}
Then
\[
f(x,y)=(A x, B y)+\wt{f}(x,y)
\]
determines an analytic function
$\wt{f}=(\wt{f}_1,\wt{f}_2) \colon B_r^E(0)\to E$
such that $\wt{f}(0)=0$
and $\wt{f}\, '(0)=0$.
By Remark~\ref{remstrict}\,(d),
after shrinking $r$ if necessary, we may assume
that $\wt{f}$ is Lipschitz with
%
\begin{equation}\label{ftlip2}
\Lip(\wt{f})<a\, .
\end{equation}
\end{numba}
Then we have:
%
%
%
\begin{thm}\label{mainthm}
In the situation of {\rm\ref{thesitu}}, the set
\begin{eqnarray}
\Gamma & :=& \{ z\in B_r^E(0)\colon \mbox{$f^n(z)$
is defined and $a^{-n}\|f^n(z)\|< r$
for}\notag \\
& & \;\; \mbox{all $n\in \N_0$, and $\lim_{n\to\infty} a^{-n} \| f^n(z) \| =0$}\}\label{thebiggam}
\end{eqnarray}
has the following properties:
\begin{itemize}
\item[\rm (a)]
$f(\Gamma)\sub \Gamma$, i.e., $\Gamma$ is $f$-invariant.
\item[\rm (b)]
$\Gamma$
is the graph
of an analytic
map $\phi\colon B_r^{E_1}(0)\to B_r^{E_2}(0)$
with $\phi(0)=0$
and $\phi'(0)=0$
$($and hence $T_0 \Gamma=E_1)$.
\end{itemize}
Moreover, the following holds:
\begin{itemize}
\item[\rm(c)]
$\phi$ is Lipschitz, with $\Lip(\phi)\leq 1$.
\item[\rm(d)]
$\Gamma$ is independent of the choice of $a$ such that $\|A\|<a<\frac{1}{\|B^{-1}\|}$
and $\Lip(\wt{f})<a$.
\item[\rm (e)]
$\Gamma\cap B_s^E(0)$ is $f$-invariant
for each $s \in \;]0,r]$
and has properties analogous to those of $\Gamma$
described in {\rm(\ref{thebiggam})} and {\rm(b)} if we replace
$r$ with~$s$ there.
\item[\rm (f)]
For each $b>0$ such that $\|A\|\leq b< \frac{1}{\|B^{-1}\|}$
and $\Lip(\wt{f})\leq b$, we have
\begin{eqnarray*}
\Gamma
& = & \{z\!\in \! B^E_r(0) \colon \! \mbox{$f^n(z)$ is defined and
$\|f^n(z)\|  \leq\!  \, b^n \hspace*{.2mm} r$
for all $n  \in   \N_0$}\}\\
& = & \{z\!\in \! B^E_r(0) \colon \! \mbox{$f^n(z)$ is defined and
$\|f^n(z)\|\!\leq\!  b^n \|z\|$
for all $n\! \in \! \N_0$}\}.
\end{eqnarray*}
\item[\rm(g)]
$\|f(z)\|\leq c\|z\|$ for each $z\in \Gamma$,
where $c:=\max\{\|A\|,\Lip(\wt{f}_1)\}<a$.
\item[\rm(h)]
If $A$ is invertible and $\Lip(\wt{f}_1)<\frac{1}{\|A^{-1}\|}$,
then $f(\Gamma)$ is open in $\Gamma$
and $f|_\Gamma\colon \Gamma\to f(\Gamma)$ is a
diffeomorphism.
\end{itemize}
\end{thm}
\begin{proof}
(a) is obvious from the definition of $\Gamma$.

(b) Identifying $\cS_a(E)$ with $\cS_a(E_1)\times \cS_a(E_2)$
as in Lemma \ref{basicsseq}\,(e), we can write elements $z=(z_n)_{n\in \N_0}\in \cS_a(E)$
in the form $z=(x,y)$ with $x=(x_n)_{n\in \N_0}\in \cS_a(E_1)$
and $y=(y_n)_{n\in \N_0}\in \cS_a(E_2)$.
We abbreviate $\cU:=B^{\cS_a(E)}_r(0)$
and consider the map $g\colon \cU\to\cS_a(E)$
taking $z=(z_n)_{n\in \N_0}\in \cU$ with $z_n=(x_n ,y_n)$
to the sequence $g(z)$ with $n$-th entry
%
\begin{equation}\label{dfnsmg}
g(z)_n:=
\left\{
\begin{array}{cl}
\big(0,\, B^{-1}(y_1-\wt{f}_2(z_0)) \big) &\mbox{if $\, n=0$;}\\
\big(f_1(z_{n-1}),\, B^{-1}(y_{n+1}-\wt{f}_2(z_n))\big) & \mbox{if $\, n\geq 1$}
\end{array}
\right.
\end{equation}
for $n\in \N_0$.
Then $g(0)=0$.
Using the right shift $\rho$ on $\cS_a(E_1)$,
the left shift~$\lambda$ on $\cS_a(E_2)$ and
the projection $\pr_2\colon E=E_1\oplus E_2\to E_2$,
we can write~$g$ as
\[
g\,=\,  \big( \rho\circ \cS_a(f_1),\, \cS_a(B^{-1})\circ (\lambda\circ \cS_a(\pr_2)-\cS_a(\wt{f}_2))\big)\,.
\]
In view of Lemma \ref{basicsseq} (a) and~(e), Lemma \ref{lamapsop}\,(a)
and Proposition \ref{propseqana}, this formula
shows that $g$ is analytic and Lipschitz with $\Lip(g)<1$.
Then
\[
G\, :=\, \id_\cU-g\colon \cU\to\cU
\]
is an analytic diffeomorphism and an isometry,
by the Ultrametric Inverse Function Theorem (Theorem \ref{IFT})
and the domination principle~(\ref{domi}).
Now
\[
w\colon B_r^{E_1}(0)\to\cU\,,\quad
w(x):= G^{-1}\big( (x,0),(0,0), \ldots\big)
\]
is an analytic map such that
$(\id_\cU-g)(w(x))=G(w(x))=((x,0),(0,0),\ldots)$
and thus
%
\begin{equation}\label{leftright}
w(x)=\big((x,0),(0,0),\ldots\big) + g(w(x))
\end{equation}
for all $x\in B^{E_1}_r(0)$. Comparing the $0$-th component
on both sides of (\ref{leftright}),
we see that $w(x)_0=(x,0)+\big(0,B^{-1}(\pr_2(w(x)_1)-\wt{f}_2(w(x)_0))\big)$
and thus
\[
w(x)_0=(x,\phi(x))\,,
\]
where $\phi\colon B^{E_1}_r(0)\to B^{E_2}_r(0)$ is the analytic function
given by
\begin{equation}\label{defirwphi}
\phi(x):=B^{-1}(\pr_2(w(x)_1)-\wt{f}_2(w(x)_0))\,.
\end{equation}
Then $\phi(0)=0$,
and since
\[
g'(0)=\big(\rho\circ \cS_a(A)\circ \cS_a(\pr_1),\,  \cS_a(B^{-1})\circ \lambda \circ \cS_a(\pr_2))=D_1\oplus D_2
\]
with $D_1:= \rho\circ \cS_a(A)\in\cL(\cS_a(E_1))$ and
$D_2:=\cS_a(B^{-1})\circ \lambda
\in \cL(\cS_a(E_2))$
of operator norm $<1$
(cf.\ proof of Proposition \ref{propseqana}),
we have
\[
(G^{-1})'(0)=(G'(0))^{-1}=(\id - (D_1\oplus D_2))^{-1}=\id+ \sum_{k=1}^\infty D^k_1\oplus D_2^k
\]
and thus
$w'(0).y=(G^{-1})'(0).\big((y,0),(0,0),\ldots\big)=
((y,0),(Ay,0),(A^2y,0),\ldots)$
for all $y\in E_1$, entailing that
\[
\phi'(0)=0\,.
\]
\emph{We claim that $w(x)$ is the $f$-orbit of $w(x)_0=(x,\phi(x))$,
for each $x\in B_r^{E_1}(0)$.} If this is true,
then $w(x)\in \cU$
implies that $a^{-n}\|f^n(x,\phi(x))\|=a^{-n}\|w(x)_n\|$
is $<r$ and tends to $0$ as $n\to \infty$,
showing that $(x,\phi(x))\in \Gamma$
and hence
\[
\mbox{graph}(\phi)\sub \Gamma\,.
\]
To prove the claim, we need only show
that $w(x)_{n+1}=f(w(x)_n)$ for each $n\in \N_0$
(because $w(x)_0=(x,\phi(x))$).
Looking at the second component of the entry in (\ref{leftright}) indexed
by $n$ and the first component of the entry indexed by $n+1$,
we see that
%
\begin{equation}\label{nothpart}
\pr_2(w(x)_n)=B^{-1}(\pr_2(w(x)_{n+1})-\wt{f}_2(w(x)_n)
\end{equation}
and $\pr_1(w(x)_{n+1})= f_1(w(x)_n)$.
Multiplying (\ref{nothpart})
with $B$, we obtain
\[
B\pr_2(w(x)_n)=\pr_2(w(x)_{n+1})-\wt{f}_2(w(x)_n)\, ,
\]
whence
$\pr_2(w(x)_{n+1})=B\pr_2(w(x)_n)
+\wt{f}_2(w(x)_n)=f_2(w(x)_n)$.
Thus $w(x)_{n+1}=f(w(x)_n)$, confirming the claim.

\noindent
For a full proof of (b), it only remains to show
that $\Gamma\sub \mbox{graph}(\phi)$.
To prove this inclusion, let $z_0=(x_0,y_0)\in \Gamma$.
Then $z:=(f^n(z_0))_{n\in \N_0}
\in \cU$ (since $a^{-n}\|f^n(z_0)\|<r$
and $a^{-n}\|f^n(z_n)\|\to 0$ by definition of $\Gamma$).
We claim that
%
\begin{equation}\label{convdir}
z=((x_0,0),(0,0),\ldots)+g(z)\,.
\end{equation}
If this is true, then $G(z)=z-g(z)=((x_0,0),(0,0),\ldots)$
and hence $z=G^{-1}((x_0,0),(0,0),\ldots)=w(x_0)$.
As a consequence, $z_0=w(x_0)_0=(x_0,\phi(x_0))$
and thus $z_0\in \mbox{graph}(\phi)$.
To prove the claim, note first that $\pr_1(z_0)=x_0$,
which is also the first component
of the $0$-th entry of the right hand side of~(\ref{convdir}).
Given $n\in \N$,
equality of the first component of the index $n$ entry
of the sequences on the left and right of (\ref{convdir})
means that
\[
\pr_1(z_n)=f_1(z_{n-1})\,,
\]
which is valid since $z_n=f^n(z_0)=f(f^{n-1}(z_0))=f(z_{n-1})$.
Next, we use that $\pr_2(z_n)=\pr_2 (f(z_{n-1}))=B\pr_2(z_{n-1}) + \wt{f}_2(z_{n-1})$,
whence $B\pr_2(z_{n-1})=\pr_2(z_n)-\wt{f}_2(z_{n-1})$
and hence
\[
\pr_2(z_{n-1})=
B^{-1}(\pr_2(z_n)-\wt{f}_2(z_{n-1}))\,.
\]
Therefore the second
components of the $(n-1)$st entries of the sequences on the left and right
of (\ref{convdir}) coincide.
As $n$ was arbitrary, (\ref{convdir}) holds.

(c) Given $x,y\in B^{E_1}_r(0)$, we have
\begin{eqnarray*}
\|\phi(x)\!-\!\phi(y)\| \!& \!\leq\!&\!
\|(x,\phi(x))\!-\!(y,\phi(y))\| = \|w(x)_0\!-\!w(y)_0\|  \leq \|w(x)\!-\!w(y)\|_a\\
\!&\!=\!&\! \|G^{-1}((x,0),(0,0),\ldots)-G^{-1}((y,0),(0,0),\ldots)\|_a\\
\!& \! = \! &\!  \|((x-y,0),(0,0),\ldots)\|_a =  \|x-y\|,
\end{eqnarray*}
using that $G$ is an isometry.

(d)  Let $b$ be a real number such that $\|A\|<b<\frac{1}{\|B^{-1}\|}$
and $\Lip(\wt{f})<b$; after interchanging $a$
and $b$ if necessary, we may assume that $b\leq a$.
Let $\Gamma'$ be the set of all $z\in B_r^E(0)$ such that
$f^n(z)$ is defined for all $n\in \N_0$,
$b^{-n}\|f^n(z)\|<r$ and $b^{-n}\|f^n(z)\|\to 0$ as $n\to\infty$.
Then also $a^{-n}\|f^n(z)\|<r$ and $a^{-n}\|f^n(z)\|\to 0$,
whence $z\in \Gamma$. Thus $\Gamma'\sub \Gamma$.
Since each of $\Gamma'$ and $\Gamma$
is the graph of a function on the same domain
(by (b)), it follows that $\Gamma=\Gamma'$.

(e) Given $s\in \;]0,r]$, let $\Gamma_s$ be the set
of all $z\in B^E_s(0)$ such that $(f|_{B^E_s(0)})^n(z)$ is defined for
all $n\in \N_0$, $a^{-n}\|f^n(z)\|<s$, and $a^{-n}\|f^n(z)\|\to 0$ as $n\to\infty$.
Applying (b) to $f|_{B^E_s(0)}$ instead of $f$,
we find that $\Gamma_s$ is the graph of a function
$\psi\colon B^{E_1}_s(0)\to B^{E_2}_s(0)$.
Since $\Gamma_s\sub \Gamma$, it follows
that $\psi$ is the restriction of $\phi$ to $B^{E_1}_s(0)$
and $\Gamma\cap B^E_s(0)=\Gamma_s$.
Since $\Gamma_s$ has been obtained in the same way
as $\Gamma$, it has analogous properties.

(f) After replacing $a$ by an element in $\; ]b,\frac{1}{\|B^{-1}\|}[$
(which is legitimate by~(d)), we may assume that $b<a<\frac{1}{\|B^{-1}\|}$.
Given $z\in \Gamma$,
let $c\in \;]b,\frac{1}{\|B^{-1}\|}[$ and $s\in \;] \|z\|,r]$.
Then $z\in \Gamma_s$ (by (e) and its proof),
and since the latter can be obtained using $c$ in place of $a$ (by (d)),
it follows that $\|f^n(z)\|< c^n s$ for each $n\in \N_0$.
Letting $c\to b$ and $s\to \|z\|$,
it follows that $\|f^n(z)\|\leq b^n\|z\|$ for each
$n\in \N_0$.
Hence $z$ belongs to the final set described in (f).
This set, in turn, is a subset of the penultimate set
occurring in (f) (as is clear from the definition of these sets).
To complete the proof, let $z\in B_r^E(0)$ be in the penultimate set;
thus $f^n(z)$ is defined for each $n\in \N_0$ and $\|f^n(z)\|\leq b^nr$,
for each $n\in \N_0$.
Then $a^{-n}\|f^n(z)\|\leq \Big(\frac{b}{a}\Big)^n r$, where $\frac{b}{a}<1$.
Thus $a^{-n}\|f^n(z)\|< r$ and $a^{-n}\|f^n(z)\|\to 0$,
whence $z\in \Gamma$. Therefore all three sets in (f)
coincide.

(g) For each $z=(x,\phi(x))\in \Gamma$,
we have
%
\begin{equation}\label{maybeusf}
\|z\|=\max\{\|x\|,\|\phi(x)\|\}=\|x\|\,,
\end{equation}
because $\phi(0)=0$ and $\Lip(\phi)\leq 1$ (by (c)).
Since $f(z)\in\Gamma$ by~(a),
we deduce that $\|f(z)\|=\|(f_1(z),\phi(f_1(z)))\|=\|f_1(z)\|
=\|A x+\wt{f}_1(z)\|\leq \max\{\|Ax\|,\|\wt{f}_1(z)\|\}\leq c\|x\|$,
using that $\|\wt{f}_1(z)\|\leq \Lip(\wt{f}_1)\|z\|=
\Lip(\wt{f}_1)\|x\|$.

(h) The map $\kappa\colon \Gamma\mto B^{E_1}_r(0)$,
$(x,y)\mto x$ is a global chart for the submanifold $\Gamma$ of $E$,
with inverse $\kappa^{-1}\colon B^{E_1}_r(0)\to\Gamma$,
$x\mto (x,\phi(x))$.
The map
$\kappa\circ f\circ\kappa^{-1}
\colon B^{E_1}_r(0)\to E_1$
takes $x\in B^{E_1}_r(0)$ to
\[
(\kappa\circ f\circ\kappa^{-1})(x)
=f_1(x,\phi(x))=Ax+\wt{f}_1(x,\phi(x))
\]
with $\Lip(\wt{f}_1\circ (\id,\phi))\leq\Lip(\wt{f}_1)<\frac{1}{\|A^{-1}\|}$.
By the Ultrametric Inverse Function Theorem
(Theorem~\ref{IFT}),
$\kappa\circ f\circ\kappa^{-1}$ has open image
and is a diffeomorphism onto its image,
from which the assertion follows.
\end{proof}
\begin{rem}
Note that the mere existence of an analytic function $\phi$ with
$\mbox{graph}(\phi)\sub \Gamma$ remains valid if $B$ is not invertible,
but merely admits a right inverse $C$ such that $\frac{1}{\|C\|}>a$
(simply replace $B^{-1}$ by $C$ in the proof of (b)).
However, if $B$ is not invertible,
it can happen that $\mbox{graph}(\phi)$
is a proper subset of $\Gamma$
(as the following  example shows).
Thus Wells \cite{Wel} was mistaken when
he thought that equality can always
be achieved in his Theorem 1.
\end{rem}
\begin{example}
Let $\K$ be a complete ultrametric field,
$a:=\frac{1}{2}$,
$E_1:=\{0\}$,\linebreak
$A:=0$,
$E:=E_2:=c_0:=\cS_1(\K)$,
and $f:=B:=\lambda$ be the left shift on~$c_0$,
which admits the right shift $C:=\rho$ as a right inverse,
with \mbox{$\frac{1}{\|C\|}=1>a$.}
Then, for any $r>0$, the set $\Gamma$ defined in (\ref{thebiggam})
coincides with the ball $B^{\cS_a(\K)}_r(0)$,
which certainly is not the graph of a function on $E_1=\{0\}$.
The same pathology occurs if we replace $\K$ by the field
$\R$ of real numbers.
\end{example}
As a first step towards Theorem \ref{ultrastabm},
we now use the preceding results to
construct local $a$-stable manifolds
for locally defined maps on a manifold.
\begin{defn}
In the situation of \ref{thesituintr}
(with $a\in\;]0,1]$),
assume that $T_p(f)$ is $a$-hyperbolic.
We call an immersed submanifold $N\sub M_0$
a \emph{local $a$-stable manifold} around~$p$ with respect to~$f$
if (a), (c) and (d) from Definition \ref{defcsta} are satisfied
as well as
\begin{itemize}
\item[(b)$''$]
$N$ is tangent at $p$ to
the $a$-stable subspace $T_p(M)_{a,\obs}$ with respect to $T_p(f)$, i.e.,
$T_p(N)=T_p(M)_{a,\obs}$.
\end{itemize}
If $a=1$, we simply speak of a \emph{local stable manifold}.
\end{defn}
Before we discuss this concept, let us clarify two related points
mentioned in the introduction.
%
%
\begin{rem}\label{hypsuniq}
$E_{a,\obs}$
is uniquely determined
in the situation of Definition~\ref{defwas},
and if $\alpha$ is an automorphism of $E$,
then also $E_{a,\obu}$ is determined
(by the argument given in Remark~\ref{udecen}).
\end{rem}
The second point concerns Definition~\ref{defwas}.
%
%
\begin{rem}\label{indep}
Note that if (\ref{dewaseq})
holds for one chart $\kappa\colon U\to V$ as
described in Definition~\ref{defwas}
and ultrametric norm $\|.\|$
on~$E$ defining its topology, then also
for every other chart $\wt{\kappa}\colon \wt{U}\to \wt{V}\sub E$
and norm $\|.\|\wt{\;}$ with analogous
properties. In fact,
since $\tau:=\wt{\kappa}\circ\kappa^{-1}$ is analytic,
it is Lipschitz on some $0$-neighbourhood
$W\sub V$.
Moreover, there exists $D>0$ such that $\|.\|\wt{\;}\leq D\|.\|$.
If $x\in M$ such that $f^n(x)\to p$ as $n\to\infty$
and $a^{-n}\|\kappa(f^n(x))\|\to 0$,
then $f^n(x)\in \kappa^{-1}(W)$ for large $n$ and thus
$a^{-n}\|\wt{\kappa}(f^n(x))\|\wt{\;}
\leq  a^{-n}D \|\tau(\kappa(f^n(x)))\|$
$\leq D\Lip(\tau|_W) a^{-n}\|\kappa(f^n(x))\|\to 0$
(as required).
\end{rem}
We can prove the following result on local $a$-stable manifolds.
%
%
\begin{thm}\label{locstamfd}
Let $M$ be an analytic manifold modelled
on an ultrametric Banach space over a complete ultrametric field.
Let $M_0\sub M$ be open, $f\colon M_0\to M$ be
an analytic mapping and $p\in M_0$ be a fixed point
of~$f$, such that $T_p(f)$ is $a$-hyperbolic for some
$a\in \;]0,1]$. Then the following holds:
\begin{itemize}
\item[\rm(a)]
There exists a local $a$-stable manifold $N$ around $p$
with respect to $f$;
\item[\rm(b)]
The germ of $N$ at $p$ is uniquely determined;
\item[\rm(c)]
If $N$ is a local $a$-stable manifold around $p$ with respect to $f$,
then every neighbourhood of $p$ in $N$ contains an
open neighbourhood $\Omega\sub N$
which is a local $a$-stable manifold around $p$,
a submanifold of $M$, and has the following
additional properties:
\begin{itemize}
\item[\rm(i)]
There exist a chart $\kappa\colon P\to U\sub T_p(M)$ of $M_0$
around $p$ such that $\kappa(p)=0$,
an ultrametric
norm $\|.\|$ on $E:=T_p(M)$ defining its topology,
and an open neighbourhood $Q\sub P$ with $f(Q)\sub P$
and $\kappa(Q)=B^E_r(0)$ for some $r>0$,
such that $\Omega=\kappa^{-1}(\Gamma)$
where
\begin{eqnarray*}
\Gamma & :=& \{ z\in B_r^E(0)\colon \mbox{$g^n(z)$
is defined and $a^{-n}\|g^n(z)\|< r$
for}\notag \\
& & \;\; \mbox{all $n\in \N_0$, and $\lim_{n\to\infty} a^{-n} \| g^n(z) \| =0$}\},
\end{eqnarray*}
with $g:=\kappa\circ f\circ\kappa^{-1}|_{B^E_r(0)}$.
\item[\rm(ii)]
$f^n(x)\to p$ for all $x\in \Omega$; and
\item[\rm(iii)]
$a^{-n}\|\kappa(f^n(x))\|\to 0$ as $n\to\infty$,
for each $x\in \Omega$ and some $($and hence any$)$
chart $\kappa$ of $M_0$ around $p$, such that
$\kappa(p)=0$.
\end{itemize}
\end{itemize}
\end{thm}
\begin{proof}
(a) and (c):
Let $E_1$ be the $a$-stable subspace
and~$E_2$ be the $a$-unstable subspace
of~$E:=T_p(M)$ with respect to $T_p(f)$,
and $\|.\|$ be a norm on $E=E_1\oplus E_2$
as in Definition \ref{defahyp}.
Let $\kappa\colon P\to U$, $Q$, $r>0$ and
$g\colon B^E_r(0)\to E$
be as in~\ref{reusable};
thus $g'(0)=T_p(f)$.
For $s\in \;]0,r]$,
let $\Gamma_s$ be the set
of all $z\in B_s^E(0)$ such that
$g^n(z)$ is defined and $a^{-n}\|g^n(z)\|< s$
for all $n\in \N_0$, and $\lim_{n\to\infty}
a^{-n} \| g^n(z) \| =0$.
After shrinking~$r$ if necessary,
the construction of Section~\ref{secirw}
can be applied with $g$ in place of $f$
(cf.\ Remark \ref{remstrict}\,(d)).
Hence, there exists an analytic map
\[
\phi\colon B_r^{E_1}(0)\to B_r^{E_2}(0)
\]
with $\phi(0)=0$
and $\phi'(0)=0$,
such that $\Gamma_s$ is the graph of $\phi|_{B_s^{E_1}(0)}$
for each $s\in \;]0,r]$,
and $g(\Gamma_s)\sub \Gamma_s$ (see Theorem \ref{mainthm}).
Then $\Gamma_s$ is a submanifold
of $B_s^E(0)$ and $g$ restricts to an analytic
map $\Gamma_s\to\Gamma_s$ for each $s\in\;]0,r]$.
Now $\Omega_s:=\kappa^{-1}(\Gamma_s)$
is a submanifold of $\kappa^{-1}(B^E_s(0))$
(and hence of $M$),
such that $f(\Omega_s)\sub \Omega_s$
and $f|_{\Omega_s}\colon \Omega_s\to\Omega_s$
is analytic.
Also, $T_p(\Omega_s)=E_1$
because $T_0(\Gamma_s)=E_1$,
and hence each $\Omega_s$ is a local $a$-stable manifold
around $p$ with respect to $f$. Moreover, the $\Omega_s$ with $s\in \;]0,r]$
form a basis of open neighbourhoods of $p$ in $\Omega_r$.
Finally, $f^n(x)\to p$ holds for all $x\in \Omega_r$
and $a^{-n}\|\kappa(f^n(x))\|=a^{-n}\|g^n(\kappa(x))\|\to 0$
by definition of $\Gamma_r$ (and we have analogous behaviour
with respect to other charts and norms, by Remark \ref{indep}).
Thus (a) is established and (c) holds for~$N=\Omega_r$
(and hence for any local $a$-stable manifold~$N$, once we have~(b)).

(b)  We retain the notation from the proof of (a),
let $N$ be another local $a$-stable manifold around $p$ with respect to $f$,
and pick $b\in\;]\|A\|,a[$ with $A:=(T_p(f))|_{E_1}$.
Let $\mu\colon V\to B^{E_1}_\tau(0)$
be a chart of~$N$ around~$p$ such that
$\mu(p)=0$ and $d\mu(p)=\id_{E_1}$.
There is $\sigma\in \;]0,\tau]$
such that $h:=\mu \circ f\circ \mu^{-1}$
is defined on all of
$B^{E_1}_\sigma(0)$.
Since $h'(0)=T_p(f|_N)=A$ with $\|A\|\leq b$,
Remark~\ref{behaveba}\,(a)
shows that
\begin{equation}\label{relvainfo}
h(B^{E_1}_s(0))\; \sub\;
B^{E_1}_{bs}(0)\; \sub\; B_s^{E_1}(0)
\quad\mbox{for all $\, s\in \;]0,\sigma]$,}
\end{equation}
after possibly shrinking~$\sigma$.
Thus $\mu^{-1}(B^{E_1}_s(0))$ is a local $a$-stable manifold.\\[2.5mm]
Now write
$g=(g_1,g_2)=g'(0)+\wt{g}\colon B^E_r(0)\to E_1\oplus E_2$,
where $\Lip(\wt{g})<a$
and $g'(0)=A\oplus B$ with $A$ as before and $\frac{1}{\|B^{-1}\|}>a$.
Then $Q\cap N$ is an immersed
submanifold of~$Q$ tangent to~$E_1$.
After replacing~$N$ by an
$f$-invariant open $p$-neighbourhood in $N$ (cf.\ (\ref{relvainfo})),
we may assume that~$N$ is a submanifold of~$Q$.
Since~$\kappa(N)$ is tangent to~$E_1$ at $0\in E$,
the inverse function theorem implies
that $\kappa(N)$ is the graph
of an analytic map $\psi\colon W\to E_2$
on some open $0$-neighbourhood
$W\sub B^{E_1}_r(0)$,
with $\psi(0)=0$, $\psi'(0)=0$ and $\Lip(\psi)\leq 1$
(after shrinking~$N$ if necessary).
Then $\nu:=\pr_1\circ \kappa|_N$ is a chart
for~$N$ with $\nu(p)=0$ and
$d\nu(p)=\id_{E_1}$
(where $\pr_1\colon E_1\oplus E_2\to E_1$).
Hence, by the discussion leading to (\ref{relvainfo}),
there exists $\sigma\in \;]0,r]$
such that $B^{E_1}_\sigma(0)\sub W$ and
\begin{equation}\label{thsdfnd}
g(\Theta_s)\;\sub\; \Theta_{bs}\quad \mbox{for all $\, s\in \;]0,\sigma]$,}
\end{equation}
where
$\Theta_s:=\{(x,\psi(x))\colon x\in B^{E_1}_s(0)\}$
for $s\in\;]0,\sigma]$.
Note that $\|z\|=\|x\|<s$
for all $s\in\;]0,\sigma]$ and $z=(x,\psi(x))\in\Theta_s$,
since $\Lip(\psi)\leq 1$.
By (\ref{thsdfnd}),
$g^n(x,\psi(x))$ is defined
for each $x\in B^{E_1}_\sigma(0)$.
Moreover,
$\|g^n(x,\psi(x))\|<b^n\sigma$
(since $g^n(x,\psi(x))\in\Theta_{b^n\sigma}$),
entailing that
$a^{-n}\|g^n(x,\psi(x))\|<\sigma$ for each $n\in \N_0$
and $a^{-n}\|g^n(x,\psi(x)\|\to 0$ as $n\to\infty$.
Hence $\Theta_\sigma\sub \Gamma_\sigma$
and since both sets are graphs of functions on the same domain,
it follows that $\Theta_\sigma=\Gamma_\sigma$.
Hence $\Theta_\sigma$ is an open submanifold
of $\Gamma_r$.
Then $\Omega_\sigma$ is an
open submanifold of $\Omega_r$ which contains $p$,
and it is also an open submanifold of $N$ because
$\Omega_\sigma=
\kappa^{-1}(\Gamma_\sigma)=\kappa^{-1}(\Theta_\sigma)
=\nu^{-1}(B^{E_1}_\sigma(0))$.
This completes the proof.
\end{proof}
Note that, in contrast to Theorem \ref{ultrastabm},
$T_p(f)$ is not assumed to be an automorphism in Theorem \ref{locstamfd}.
\section{Global stable manifolds}\label{secglob}
%
%
We now prove Theorem~\ref{ultrastabm}
from the introduction.\\[2.5mm]
{\bf Proof of Theorem~\ref{ultrastabm}.}
Let $r>0$, $E$, $\kappa$, $Q$ with $\kappa(Q)=B^E_r(0)$, $g$  and the submanifolds
$\kappa^{-1}(\Gamma_s)$ of $M$ (for $s\in\; ]0,r]$)
be as in the proof of Theorem \ref{locstamfd}\,(c).
By Theorem \ref{mainthm}\,(h) and Remark~\ref{remstrict}\,(d),
we may assume that
$g(\Gamma)$ is an open subset of $\Gamma:=\Gamma_r$
and $g|_{\Gamma}\colon \Gamma\to g(\Gamma)$
is an analytic diffeomorphism.
Then $\Omega=\kappa^{-1}(\Gamma)$ is
a submanifold of~$M$ which is tangent to $T_p(M)_{a,\obs}$ at $p$,
the image $f(\Omega)$ is open in~$\Omega$,
and~$f$ restricts to a diffeomorphism $\Omega\to f(\Omega)$.
We now show that
%
\begin{equation}\label{impostp}
W_a^\obs=\bigcup_{n\in \N_0}f^{-n}(\Omega)\,.
\end{equation}
By (ii) and (iii) in Theorem~\ref{locstamfd}\,(c),
we have $\Omega\sub W_a^\obs$
and hence $\bigcup_{n\in \N_0}f^{-n}(\Omega)$ $\sub W_a^\obs$,
recalling that $f(W_a^\obs)=W_a^\obs$. \\[2.5mm]
Conversely, suppose that $x\in W_a^s$.
Then, by the definition of $W_a^s$
(and Remark~\ref{indep}),
there is $m\in \N_0$
such that $f^n(x)\in Q$ for all $n\geq m$
and $a^{-n}\|\kappa(f^n(x))\|\to 0$.
After increasing $m$ if necessary, we may assume
that $a^{-n}\|\kappa(f^n(x))\|<r$
for all $n\geq m$.
Then
\[
a^{-k}\|\kappa(f^k(f^m(x)))\|
\; \leq \; a^{-(k+m)}\|\kappa(f^{k+m}(x))\| \; <\; r
\]
for all $k\in \N_0$
and $a^{-k}\|g^k(\kappa(f^m(x)))\|=a^{-k}\|\kappa(f^{k+m}(x))\|\to 0$,
showing that
$\kappa(f^m(x))\in \Gamma$.
Therefore $f^m(x)\in\Omega$
and thus $x\in f^{-m}(\Omega)\sub \bigcup_{n\in \N_0}f^{-n}(\Omega)$,
completing the proof of~(\ref{impostp}).\\[2.5mm]
Since $\Omega$ is a submanifold of~$M$ and $f$ a diffeomorphism,
also $\Omega_n:=f^n(\Omega)$
is a submanifold of~$M$, for each $n\in \Z$.
Because~$f(\Omega)$ is an open submanifold of~$\Omega$
and~$f$ restricts to a diffeomorphism
from~$\Omega$ to~$f(\Omega)$,
it follows that $\Omega_m$ is an open submanifold
of $\Omega_n$, for all $m\geq n$,
and that the map
\[
h_{m,n}\colon \Omega_n\to \Omega_m
\]
induced by $f^{m-n}$ is a diffeomorphism.
We give $W_a^\obs$ the unique
analytic manifold structure making
each $\Omega_n$ an open submanifold of $W_a^\obs$.
Since each $\Omega_n$ is open in $W_a^\obs$ and the inclusion map
$\Omega_n\to M$ is an immersion, $W_a^\obs$
is an immersed submanifold of~$M$.
Moreover, $T_p(W_a^\obs)=T_p(\Omega)$
is the $a$-stable subspace of $T_p(M)$ with respect to $T_p(f)$,
and
$f(W_a^\obs)=W_a^\obs$.
The restriction $h$ of $f$ to a map
$W_a^\obs\to W_a^\obs$ is analytic, because
$h|_{\Omega_n}=\lambda_{n+1}\circ h_{n+1,n}$
is analytic for each $n\in \Z$
(where $\lambda_n\colon \Omega_n\to W_a^\obs$
is the inclusion map).
Also $h^{-1}$ is analytic, because
$h^{-1}|_{\Omega_n}=\lambda_{n-1}\circ (h_{n,n-1})^{-1}$
is analytic.
Thus, the existence part of
Theorem~\ref{ultrastabm} is established.
Moreover, the final assertion holds because
$\Omega$ might be replaced by its open subset $\kappa^{-1}(\Gamma_s)$
for each $s\in \;]0,r]$.

Uniqueness: We let $W_a^\obs$ be the manifold just constructed
and write $N$ for $W_a^\obs$, equipped with
any manifold structure for which conditions (a)--(c)
of the theorem are satisfied. Then both $N$ and $W_a^\obs$
are local $a$-stable manifolds.
Hence, there is an open neighbourhood $W\sub N$ of~$p$
which is also an open neighbourhood of $p$ in $W_a^\obs$,
and on which~$N$ and $W_a^\obs$
induce the same analytic manifold structure
(Theorem~\ref{locstamfd}\,(b)).
Let $\Omega\sub W$ be as before.
Since $f^n$ restricts to diffeomorphism of both $W_a^\obs$
and~$N$, 
it follows that $f^{-n}(\Omega)$
is an open submanifold of both $W_a^\obs$ and~$N$
(with the same induced analytic manifold structure),
for each $n\in\N_0$.
Since the sets $f^{-n}(\Omega)$ form an open cover
of $W_a^\obs$ and~$N$, it follows that
$W_a^\obs=N$ as an analytic manifold.\,\Punkt
\section{Local unstable manifolds}\label{secirw3}
%
%
We construct local unstable manifolds
by an adaptation of Irwin's method
(cf.\ also \cite{Wel}). Our general setting is the following:
\begin{numba}\label{newsett7}
Let $M$ be an analytic manifold modelled on an ultrametric Banach space
over a complete ultrametric field.
Let $M_0\sub M$ be open,
$f\colon M_0\to M$
be analytic,
$a>0$
and $p$ be a fixed point of $f$
such that $T_p(f)$ is $a$-hyperbolic.
\end{numba}
\begin{defn}
In the situation of \ref{newsett7},
an immersed submanifold $N\sub M_0$
is called a \emph{local $a$-unstable manifold} around~$p$ with respect to~$f$
if (a)--(c) hold:
\begin{itemize}
\item[(a)]
$p\in N$;
\item[(b)]
$N$ is tangent at $p$ to the $a$-unstable subspace
$T_p(M)_{a,\obu}$ of $T_p(M)$ with respect to $T_p(f)$;
\item[(c)]
There exists an open neighbourhood $U$ of $p$
in $N$ such that $f(U)\sub N$
and $f|_U\colon U\to N$ is analytic.
\end{itemize}
\end{defn}
We can show:
%
%
\begin{thm}\label{locumfd}
If $a\geq 1$ in the situation of {\rm\ref{newsett7}},
then the following holds:
\begin{itemize}
\item[\rm(a)]
There exists a local $a$-unstable manifold $N$ around $p$
with respect to $f$;
\item[\rm(b)]
The germ of $N$ at $p$ is uniquely determined.
\end{itemize}
\end{thm}
\begin{rem}
If $T_p(f)$ is invertible,
we may assume that $M_1:=f(M_0)$ is open and $f\colon M_0\to M_1$
a diffeomorphism, after possibly shrinking $M_0$.
Now Theorem \ref{locstamfd}
provides a local $\frac{1}{a}$-stable manifold $N\sub M_0\cap M_1$
around $p$ with respect to  $f^{-1}\colon M_1\to M$.
Then $N$ is a local $a$-unstable manifold
with respect to~$f$. Hence, we already have most of
Theorem \ref{locumfd} if $T_p(f)$ is invertible.
The interesting point is
that the theorem remains valid
if $T_p(f)$ is not invertible.
\end{rem}
Because the proof of Theorem \ref{locumfd}
is quite similar to that of Theorem \ref{locstamfd}, we relegate it to
an appendix (Appendix \ref{sec2app}).
\section{Spectral interpretation of hyperbolicity}\label{secfin}
%
In this section, we consider
the special case where~$\alpha$ is an automorphism
of a \emph{finite-dimensional} vector space
over a complete ultrametric field
$(\K,|.|)$.
We shall interpret $a$-hyperbolicity
as the absence of eigenvalues
of absolute value $a$ (in an
algebraic closure of $\K$).
Moreover, we shall see that an $a$-centre subspace
and an $a$-centre-stable subspace always exist.
%
%
\begin{numba}\label{findiset}
Let $(\K,|.|)$ be a complete ultrametric
field, $E$ be a finite-dimensional $\K$-vector
space, and $\alpha\colon E\to E$
be a linear map.
We let $\wb{\K}$ be an algebraic closure of~$\K$,
and use the same symbol,
$|.|$, for the unique extension
of the given absolute value to~$\wb{\K}$
(see \cite[Theorem~16.1]{Sch}).
We let $R(\alpha)$ be the set
of all absolute values $|\lambda|$,
where $\lambda\in \wb{\K}$
is an eigenvalue
of the $\wb{\K}$-linear self-map
$\alpha_{\wb{\K}}:=\alpha\otimes \id_{\wb{\K}}$
of the $\wb{\K}$-vector space
$E_{\wb{\K}}:=E\tensor_\K \wb{\K}$
obtained from~$E$ by extension of scalars.
For each $\lambda\in \wb{\K}$,
we let
\[
(E_{\wb{\K}})_{(\lambda)}\, :=\, \{ x\in E_{\wb{\K}}\!:
(\alpha_{\wb{\K}}-\lambda)^dx=0\}
\]
be the generalized eigenspace of $\alpha_{\wb{\K}}$
in $E_{\wb{\K}}$ corresponding to~$\lambda$
(where $d$ is the dimension of the $\K$-vector space~$E$).
Given $\rho\in [0,\infty[$,
we define
\begin{equation}\label{dfspacerho}
(E_{\wb{\K}})_\rho\; :=\;
\bigoplus_{|\lambda|=\rho} (E_{\wb{\K}})_{(\lambda)}\, ,\vspace{-.7mm}
\end{equation}
where the sum is taken over all
$\lambda\in \wb{\K}$
such that $|\lambda|=\rho$.
\end{numba}
The following fact (cf.\ (1.0) on p.\,81 in \cite[Chapter~II]{Mar})
is important:\footnote{In \cite[p.\,81]{Mar},
$\K$ is a local field,
but the proof works also for complete ultrametric fields.}
\begin{la}
For each $\rho\in R(\alpha)$,
the vector
subspace $(E_{\wb{\K}})_\rho$ of $E_{\wb{\K}}$ is defined
over~$\K$, i.e.,
$(E_{\wb{\K}})_\rho= (E_\rho)_{\wb{\K}}$
with $E_\rho:=(E_{\wb{\K}})_\rho\cap E$.
Thus
\begin{equation}\label{isdsum}
E\; =\; \bigoplus_{\rho\in R(\alpha)} E_\rho\,,
\end{equation}
and each $E_\rho$
is an $\alpha$-invariant vector subspace of~$E$.\,\Punkt
\end{la}
It is essential for us that certain well-behaved norms
exist on~$E$ (as in~\ref{findiset}).
%
%
\begin{defn}\label{defnadpt}
A norm $\|.\|$
on $E$ is \emph{adapted to $\alpha$} if the following holds:
\begin{itemize}
\item[(a)]
$\|.\|$ is ultrametric;
\item[(b)]
$\big\|\sum_{\rho\in R(\alpha)} x_\rho\big\|=\max\{\|x_\rho\|
\colon \rho\in R(\alpha)\}$\vspace{.5mm}
for each $(x_\rho)_{\rho\in R(\alpha)}
\in \prod_{\rho\in R(\alpha)} E_\rho$; and
\item[(c)]
$\|\alpha (x)\|=\rho\|x\|$ for each $0\not= \rho\in R(\alpha)$
and $x\in E_\rho$.
\end{itemize}
\end{defn}
%
%
\begin{prop}\label{propadapt}
Let $E$ be a finite-dimensional vector space over
a complete ultrametric field $(\K,|.|)$ and
$\alpha\colon E\to E$ be a linear map.
Let $\ve>0$
and $E_0:=\{x\in E\colon (\exists n\in \N)\;\alpha^n(x)=0\}$.
Then $E$ admits a norm $\|.\|$ adapted to~$\alpha$,
such that $\alpha|_{E_0}$ has operator norm
$<\ve$ with respect to $\|.\|$.
\end{prop}
The proof uses \cite[Lemma~4.4]{SUR}
(the proof of which does not
require that $\alpha$ is an automorphism,
as assumed in \cite{SUR}):
%
\begin{la}\label{normrho}
For each $\rho\in R(\alpha)\setminus\{0\}$,
there exists an ultrametric norm
$\|.\|_\rho$ on $E_\rho$ such that $\|\alpha(x)\|_\rho=\rho\|x\|_\rho$
for each $x\in E_\rho$.\,\Punkt
\end{la}
The next lemma takes care of the case $\rho=0$.
%
\begin{la}\label{rhozero}
Let $E$ be a finite-dimensional vector space over a complete
ultrametric field $(\K,|.|)$
and $\alpha\colon E\to E$ be a nilpotent
linear map. Let $\ve>0$.
Then there exists an ultrametric norm $\|.\|$ on~$E$
with respect to which $\alpha$ has operator norm $<\ve$.
\end{la}
\begin{proof}
Assume first that there exists a basis $v_1,\ldots, v_m$
of $E$ with respect to which $\alpha$
has Jordan normal form with a single Jordan block,
i.e., $\alpha(v_1)=0$ and $\alpha(v_k)=v_{k-1}$
for $k\in \{2,\ldots, m\}$.
The case $E=\{0\}$ being trivial,
we may assume that $m\geq 1$.
Choose $\lambda\in \K$ such that $0<|\lambda|<\ve$
and define $w_k:=\lambda^k v_k$
for $k\in \{1,\ldots, m\}$.
Then $\alpha(w_k)=\lambda^k v_{k-1}=\lambda w_{k-1}$
for $k\in\{2,\ldots, m\}$ and $\alpha(w_1)=0$,
entailing that $\alpha$ has operator norm
$<\ve$ with respect to the maximum norm $\|.\|$
on $E$ with respect to the basis $w_1,\ldots, w_m$,
\[
\left\|\sum_{k=1}^m t_kw_k\right\|\; :=\;
\max\{|t_k|\colon k=1,\ldots, m\}\quad \mbox{for $\, t_1,\ldots, t_m\in \K$.}
\]
In the general case, we write
$E$ as a direct sum $\bigoplus_{j=1}^nE_j$
of $\alpha$-invariant
vector subspaces $E_j\sub E$
such that the Jordan decomposition of $\alpha|_{E_j}$
has a single Jordan block.
For each $j$,
there exists
an ultrametric norm $\|.\|_j$ on $E_j$
with respect to which $\alpha|_{E_j}$ has operator
norm $<\ve$,
by the above special case.
Then $\alpha$ has operator norm $<\ve$
with respect to the ultrametric norm $\|.\|$ on~$E$
given by $\|v_1+\cdots +v_n\|:=\max\{\|v_j\|_j\colon j=1,\ldots, n\}$
for $v_j\in E_j$.
\end{proof}
{\bf Proof of Proposition~\ref{propadapt}.}
For each $\rho\in R(\alpha)\setminus\{0\}$, we choose a norm $\|.\|_\rho$
on $E_\rho$ as described in Lemma~\ref{normrho}.
Lemma~\ref{rhozero}
provides an ultrametric norm $\|.\|_0$
on $E_0$, with respect to which
$\alpha|_{E_0}$ has operator norm $<\ve$.
Then
\[
\Big\| \sum_{\rho\in R(\alpha)} x_\rho\Big\| \; :=
\; \max\, \big\{ \,\|x_\rho\|_\rho\colon \rho\in R(\alpha)\,\big\}\quad
\mbox{for $\,(x_\rho)_{\rho\in R(\alpha)} \in \prod_{\rho\in R(\alpha)}
E_\rho$}
\]
defines a norm $\|.\|\colon E\to[0,\infty[$
which, by construction, is adapted to~$\alpha$
and with respect to which
$\alpha|_{E_0}$ has operator norm $<\ve$.\,\vspace{3mm}\Punkt

\noindent
In the finite-dimensional case,
the next corollary (and its proof) provide
lucid interpretations for $a$-hyperbolicity,
and also for
$a$-centre-stable and $a$-centre subspaces.
%
\begin{cor}\label{spechyper}
Let $E$ be a finite-dimensional
vector space over a complete ultrametric field $(\K,|.|)$.
Let $\alpha\colon E\to E$ be a linear map,
$R(\alpha)$ be as in {\rm\ref{findiset}},
and $a>0$.
Then the following holds:
\begin{itemize}
\item[\rm(a)]
$E$ admits an $a$-centre-stable
subspace with respect to $\alpha$.
\item[\rm(b)]
If $\alpha$ is invertible,
then $E$ admits an $a$-centre subspace
with respect~to~$\alpha$.
\item[\rm(c)]
$\alpha$ is $a$-hyperbolic
if and only if $a\not\in R(\alpha)$.
\end{itemize}
\end{cor}
\begin{proof}
By Proposition~\ref{propadapt},
there exists an ultrametric norm $\|.\|\wt{\;}$
on $E$ which is adapted to~$\alpha$,
and with respect to which
$\alpha|_{E_0}$ has operator norm $<a$.

(a) The conditions
from Definition~\ref{defcsub}
are satisfied with $\|.\|:=\|.\|\wt{\;}$
and
%
\begin{equation}\label{sodecacs}
E_{a,\obcs}:=\bigoplus_{\rho\leq a} E_\rho\quad\mbox{ and }\quad
E_{a,\obu}:= \bigoplus_{\rho >a} E_\rho\,.
\end{equation}

(b)
The conditions
of Definition~\ref{defacentre}
are satisfied with $\|.\|:=\|.\|\wt{\;}$ and
%
\begin{equation}\label{sodecac}
E_{a,\obs}:=\bigoplus_{\rho< a} E_\rho,\quad
E_{a,\obc}:=E_a,\quad
\quad\mbox{ and }\quad
E_{a,\obu}:= \bigoplus_{\rho >a} E_\rho\,.
\end{equation}

(c) If $a\not\in R(\alpha)$,
then the conditions
of Definition~\ref{defahyp}
are satisfied with $\|.\|:=\|.\|\wt{\;}$
and
%
\begin{equation}\label{soahyp}
E_{a,\obs}:=\bigoplus_{\rho< a} E_\rho \quad
\quad\mbox{ and }\quad
E_{a,\obu}:= \bigoplus_{\rho >a} E_\rho\,.
\end{equation}
If $a\in R(\alpha)$,
then $\alpha$ cannot be $a$-hyperbolic.
In fact,
if $\alpha$ was $a$-hyperbolic,
we obtain a norm $\|.\|$
and a splitting $E=E_{a,\obs}\oplus E_{a,\obu}$
as in Definition~\ref{defahyp}.
Let $\alpha_1:=\alpha|_{E_{a,\obs}}$ and
$\alpha_2:=\alpha|_{E_{a,\obu}}$.
Since $\|.\|$ and $\|.\|\wt{\;}$
are equivalent, there exists
$C>0$ such that $C^{-1}\|.\|\leq \|.\|\wt{\;}\leq C\|.\|$.
Let $0\not=v\in E_a$.
Write $v=x+y$ with $x\in E_{a,\obs}$ and $y\in E_{a,\obu}$.
If $y\not=0$, then
\[
\|v\|\wt{\;}=a^{-n}\|\alpha^n(v)\|\wt{\;}
\geq
a^{-n} C^{-1}\|\alpha^n(v)\|
\geq C^{-1} \left(\frac{1}{a \|\alpha_2^{-1}\|}\right)^n \|y\|
\]
for all $n\in \N$, which is absurd because
$\frac{1}{a\|\alpha_2^{-1}\|}>1$.
Hence $y=0$ and thus $x=v\not=0$.
But then
\[
\|v\|\wt{\;}=a^{-n}\|\alpha^n(v)\|\wt{\;}
\leq a^{-n} C \|\alpha^n(v)\|
\leq C \left(\frac{\|\alpha_1\|}{a}\right)^n\|v\|
\quad
\mbox{for all $\, n\in \N$.}
\]
Since $\frac{\|\alpha_1\|}{a}<1$,
this is absurd. Thus $\alpha$
cannot be $a$-hyperbolic.
\end{proof}
\begin{rem}
If $E$ is an infinite-dimensional
Banach space over a complete
ultrametric field
and $\alpha\colon E\to E$
an automorphism,
let $E_{\wb{\K}}$
be the completed
projective tensor product
$E\, \wb{\otimes}_\K \, \wb{\K}$
(which is an ultrametric Banach space
over $\wb{\K}$)
and $\alpha_{\wb{\K}}:=\alpha\, \wb{\tensor} \, \id_{\wb{\K}}$.
It is natural to define
$R(\alpha):=\{|\lambda|\colon \lambda\in
\sigma(\alpha_{\wb{\K}})\}$ in this case,
where $\sigma(\alpha_{\wb{\K}})$
is the set of all
$\lambda\in \wb{\K}$
such that $\alpha_{\wb{\K}}-\lambda \id$
is not invertible.
The author conjectures
that $\alpha$ is $a$-hyperbolic if
and only if $a\not\in R(\alpha)$,
and moreover that $E$ has an $a$-centre-stable
subspace (resp., an $a$-centre subspace)
if and only if there exists $\ve>0$ such that
$]a,a+\ve[\;\cap R(\alpha)=\emptyset$
(resp.,
$]a,a+\ve[\;\cap R(\alpha)=\emptyset$
and $]a-\ve,a[\;\cap R(\alpha)=\emptyset$).
These topics require further investigation.
The functional calculus
from \cite{Esc} should be useful.
\end{rem}
\section{Behaviour close to a fixed point}\label{seccon}
%
We now relate
the behaviour of a dynamical system $(M,f)$
around a fixed point~$p$
and properties of the linear map $T_p(f)$.
The results (and those from Sections~\ref{notonly} through~\ref{seclie})
are useful for Lie theory
(see \cite{MaZ} and \cite{SPO};
cf.\ \cite{SUR}).
%
%
\begin{numba}\label{situnow}
Let $M$ be an analytic manifold modelled on an ultrametric Banach space
over a complete ultrametric field $(\K,|.|)$.
Let $f\colon M_0\to M$ be an analytic
mapping on an open subset $M_0\sub M$ and $p\in M_0$ be a fixed
point of $f$, such that $T_p(f)\colon T_p(M)\to T_p(M)$
is an automorphism.
\end{numba}
%
%
%
%
\begin{prop}\label{sin}
In \,{\rm\ref{situnow}},
the following conditions are equivalent:
\begin{itemize}
\item[\rm(a)]
$T_p(M)$ admits a centre-stable
subspace with respect to $T_p(f)$,
and
each neighbourhood $P$ of $p$ in $M_0$ contains a neighbourhood $Q$
of $p$ such that $f(Q)\sub Q$.
\item[\rm(b)]
There exists a norm $\|.\|$ on $T_p(M)$ defining its topology,
such that $\|T_p(f)\|\leq 1$ holds for the corresponding operator norm.
\end{itemize}
If, moreover, $M$ is a finite-dimensional
manifold, then {\rm(a)} and {\rm(b)}
are also equivalent to the following condition:
\begin{itemize}
\item[\rm(c)]
Each eigenvalue $\lambda$ of $T_p(f)\otimes_\K \id_{\wb{\K}}$
in an algebraic closure $\wb{\K}$ of $\K$ has absolute value
$|\lambda|\leq 1$.
\end{itemize}
\end{prop}
\begin{proof}
(b) means that $E:=T_p(M)$ coincides with
its centre-stable subspace with respect to $\alpha:=T_p(f)$.
If $E$ is finite-dimensional,
this property is equivalent to $R(\alpha)\sub \;]0,1]$
and hence to~(c),
by Remark~\ref{unicst}
and~(\ref{sodecacs}).
If~(b) holds, then (a) follows with Theorem~\ref{csthm}\,(c).

(a)$\impl$(b): If (a) holds,
then $E$ admits a decomposition
$E=E_{1,\obcs}\oplus E_{1,\obu}$
and a norm $\|.\|$, as described in Definition~\ref{defcsub}
(with $a=1$). After shrinking $M_0$, we may
assume that $M_1:=f(M_0)$ is open in $M$ and $f \colon M_0\to M_1$
is a diffeomorphism (by the Inverse Function Theorem).

If $E_{1,\obu}\not=\{0\}$,
we let $P\sub M_0\cap M_1$
be an open neighbourhood of $p$ such that $f(P)\sub P$,
and consider the map $g:=f^{-1}\colon M_1\to M$.
Then $E_{1,\obu}$ is the stable subspace of~$E$
with respect to $T_p(g)=\alpha^{-1}$.
Pick $b\in \;] \|\alpha^{-1}|_{E_{1,\obu}}\| ,1[$.
Then $\alpha^{-1}$ is $b$-hyperbolic, and
\[
E_{b,\obs}=E_{1,\obu} \quad \mbox{ as well as  }\quad E_{b,\obu}=E_{1,\obcs}
\]
(with respect to the automorphisms $\alpha^{-1}$ and $\alpha$
on the left and right of the equality signs, respectively).
By Theorem~\ref{locstamfd}
(applied to $g|_P\colon P\to M$),
there exists a local $b$-stable
manifold $N\sub P$ with respect to~$g$,
such that $g^n(x)\to p$ as $n\to\infty$,
for all $x\in N$.
Since $N$ is tangent to
$E_{1,\obu}\not=\{0\}$,
we have $N\not=\{p\}$
and thus find a point $x\in N\setminus\{p\}$.
By hypothesis~(a), there is
an open $p$-neighbourhood $Q\sub P\setminus \{x\}$
with $f(Q)\sub Q$.
Since $g^n(x)\to p$,
there exists $m\in \N$ with
$y:=g^m(x)\in Q$.
Then $x=f^m(y)\in f^m(Q)\sub Q$,
contradicting the choice of~$Q$.
Hence $E_{1,u}=\{0\}$ (and thus (b) holds).
\end{proof}
%
%
%
\begin{prop}\label{typeR}
In \,{\rm\ref{situnow}},
the following conditions are equivalent:
\begin{itemize}
\item[\rm(a)]
$T_p(M)$ admits a centre subspace
with respect to $T_p(f)$, and
each\linebreak
neighbourhood $P$ of $p$ in $M_0$ contains a neighbourhood $Q$
of $p$ such that $f(Q)=Q$.
\item[\rm(b)]
There exists a norm $\|.\|$ on $T_p(M)$ defining its topology,
which makes $T_p(f)$ an isometry.
\end{itemize}
If, moreover, $M$ is a finite-dimensional
manifold, then {\rm(a)} and {\rm(b)}
are also equivalent to the following condition:
\begin{itemize}
\item[\rm(c)]
Each eigenvalue $\lambda$ of $T_p(f)\otimes_\K \id_{\wb{\K}}$
in an algebraic closure $\wb{\K}$ of $\K$ has absolute value
$|\lambda|=1$.
\end{itemize}
\end{prop}
\begin{proof}
(b) means that $E:=T_p(M)$ coincides with
its centre subspace with respect to $\alpha:=T_p(f)$.
If $E$ is finite-dimensional,
this property is equivalent to $R(\alpha)\sub \{1\}$
and hence to~(c),
by Remark~\ref{udecen}
and~(\ref{sodecac}).
If~(b) holds, then~(a) follows with
Theorem~\ref{cthm}\,(c).

(a)$\impl$(b):
After shrinking $M_0$, we may
assume that $M_1:=f(M_0)$ is open in $M$ and $f \colon M_0\to M_1$
is a diffeomorphism.
If (a) holds,
then there is a decomposition
$E=E_{1,\obs}\oplus E_{1,\obc}\oplus E_{1,\obu}$
and a norm $\|.\|$, as in Definition~\ref{defacentre}
(with $a=1$). By ``(a)$\impl$(b)''
in Proposition~\ref{sin},
we have $E_{1,\obu}=\{0\}$.
Applying
Proposition~\ref{sin}
to $g:=f^{-1}\colon M_1\to M$,
we see that also $E_{1,\obs}=\{0\}$
(because this is the unstable subspace of $T_p(M)$
with respect to $T_p(g)=\alpha^{-1}$).
Thus $E=E_{1,\obc}$,
establishing~(b).
\end{proof}
The proofs show that $Q$ can always be chosen
as an open subset of $M_0$, in part (a)
of Proposition \ref{sin} and \ref{typeR}.
%
\begin{defn}\label{typesfp}
In the situation of {\rm\ref{situnow}},
we use the following terminology:
\begin{itemize}
\item[(a)]
$p$ is said to be an \emph{attractive} fixed point
of $f$ if $p$  has a neighbourhood
$P\sub M_0$ such that $f^n(x)$ is defined
for all $x\in P$ and $n\in \N$,
and $\lim_{n\to\infty}f^n(x)=p$ for all $x\in P$.
\item[(b)]
We say that $p$ is \emph{uniformly attractive}
if it is attractive and, moreover,
every neighbourhood of~$p$ in~$M_0$
contains a neighbourhood~$Q$ of~$p$ such that
$f(Q)\sub Q$.
\end{itemize}
\end{defn}
%
%
\begin{prop}\label{unicon}
In \,{\rm\ref{situnow}},
the following conditions are equivalent:
\begin{itemize}
\item[\rm(a)]
$T_p(M)$ admits a centre subspace with respect
to $T_p(f)$, and $p$ is\linebreak
uniformly attractive;
\item[\rm(b)]
There exists a norm $\|.\|$ on $T_p(M)$ defining its topology,
such that $\|T_p(f)\|<1$ holds for the corresponding operator norm.
\end{itemize}
If, moreover, $M$ is a finite-dimensional
manifold, then {\rm(a)} and {\rm(b)}
are also equivalent to the following condition:
\begin{itemize}
\item[\rm(c)]
Each eigenvalue $\lambda$ of $T_p(f)\otimes_\K \id_{\wb{\K}}$
in an algebraic closure $\wb{\K}$ of $\K$ has absolute value
$|\lambda| < 1$.
\end{itemize}
\end{prop}
\begin{proof}
(b) means that $E:=T_p(M)$ coincides with
its stable subspace with respect to $\alpha:=T_p(f)$.
If $E$ is finite-dimensional,
this property is equivalent to $R(\alpha)\sub \;]0,1[$
and hence to~(c),
by Remark~\ref{hypsuniq}
and~(\ref{sodecac}).
If (b) holds and $P\sub M_0$
is an open neighbourhood of~$p$,
then Theorem~\ref{locstamfd}
(applied to $f|_P$ instead of $f$)
provides a local stable manifold $N\sub P$
such that $\lim_{n\to\infty}f^n(x)=p$ for all $x\in N$. 
Because $T_p(N)=E=T_p(M)$,
it follows that~$N$ is open in~$M$.
Since, moreover, $f(N)\sub N$ by definition of $N$,
we have verifed that $p$ is uniformly attractive.
\end{proof}
\begin{rem}
If $p$ is merely attractive (but possibly not
uniformly) and $E:=T_p(M)$
admits a centre subspace with respect to $T_p(f)$,
we can still conclude that $E_{1,\obc}=\{0\}$.\\[2.5mm]
[After shrinking $M_0$, we may assume
that $f$ is injective.
Let $P\sub M_0$ be as in Definition~\ref{typesfp}\,(a).
If $E_{1,\obc}\not=\{0\}$, we let $Q\sub P$
be a centre manifold with respect to~$f$,
such that $f(Q)=Q$ (see Theorem~\ref{cthm}\,(c)).
Since $E_{1,\obc}\not=\{0\}$, we must have $Q\not=\{p\}$,
enabling us to pick $x_0\in Q\setminus\{p\}$.
Using
Theorem~\ref{cthm}\,(c)
again, we find
a centre manifold $S\sub Q\setminus\{x_0\}$
with respect to~$f$,
such that $f(S)=S$.
Since $f$ is injective, it follows that
$f(Q\setminus S)=Q\setminus S$
and thus $f^n(x_0)\in Q\setminus S$ for all $n\in \N_0$.
As $Q$ is a neighbourhood of~$p$,
we infer $f^n(x_0)\not\to p$ as $n\to\infty$.
Since $x_0\in P$, this contradicts the choice of~$P$.\,]
\end{rem}
\section{When {\boldmath$W_a^\obs(f,p)$} is
not only immersed}\label{notonly}
%
%
In general, $W_a^\obs$ is only an \emph{immersed}
submanifold of $M$,
not a submanifold
(cf.\
\cite[\S7.1]{SUR}
for an easy example).
We now describe
a criterion
(needed in~\cite{MaZ})
which prevents such pathologies.
%
%
\begin{prop}\label{Bangd}
Let $M$ be an analytic manifold modelled on an
ultrametric Banach space over a complete
ultrametric field.
Let $p\in M$ be a fixed point of
an analytic diffeomorphism
$f\colon M\to M$, such that
$E:=T_p(M)$ admits a centre-stable subspace
with respect to $T_p(f)$,
and $E_{1,\obu}=\{0\}$.
Then $W_a^\obs (f,p)$ is a submanifold
of~$M$, for each $a\in \;]0,1]$
such that $T_p(f)$ is $a$-hyperbolic.
\end{prop}
\begin{proof}
Let $W_a^\obs:=W_a^\obs(f,p)$ and
$\Omega \sub W_a^\obs$ be as
in Theorem~\ref{ultrastabm}.
Since $f$ restricts to a diffeomorphism of $W_a^\obs$,
the image $f(\Omega)$ is relatively
open in $\Omega$.
Hence,
there exists an open $p$-neighbourhood $Q\sub M$
such that $\Omega\cap Q\sub f(\Omega)$.
By ``(b)$\impl$(a)'' in Proposition \ref{sin},
we may assume that $f(Q)\sub Q$,
after replacing $Q$ with a smaller
neighbourhood of $p$ is necessary.
We claim that
%
\begin{equation}\label{goclaim}
W_a^\obs \cap Q
=\Omega\cap Q\,.
\end{equation}
If this is true, then
$W_a^\obs \cap Q$
is a submanifold
of $M$, and hence also
\[
f^{-n}(W_a^\obs)\cap f^{-n}(Q)=W_a^\obs\cap f^{-n}(Q)
\]
is a submanifold of~$M$ (as $f^{-n}\colon M\to M$
is a diffeomorphism).
Since $\bigcup_{n\in \N_0}f^{-n}(Q)$ is an open subset
of~$M$ which contains $W_a^\obs$ (exploiting that
$f^n(x)\in Q$ for large~$n$, for each $x\in W_a^\obs$),
we deduce that $W_a^\obs$ is a submanifold of~$M$
(and the submanifold structure
coincides with the immersed submanifold structure
constructed earlier).\\[2.5mm]
To prove (\ref{goclaim}), suppose
that $x\in W_a^\obs\cap Q$ but $x\not\in \Omega\cap Q$
(and hence $x\not\in \Omega$).
Since $f(Q)\sub Q$, we then have
\[
f^n(x)\in Q\quad\mbox{for all $\,n\in \N_0$.}
\]
By definition of $\Omega$, there exists
$n\in \N_0$ such that $f^n(x)\in \Omega$.
We choose $n$ minimal and note that
$n\geq 1$ as $x\not\in \Omega$
by hypothesis.
Then $f^n(x)\in \Omega\cap Q\sub f(\Omega)$
and hence $f^{n-1}(x)=f^{-1}(f^n(x)) \in f^{-1}(f(\Omega))=\Omega$,
contradicting the minimality of~$n$.
Hence $x$ cannot exist
and thus $W_a^\obs\cap Q\sub \Omega\cap Q$.
The converse inclusion, $\Omega\cap Q\sub W_a^\obs\cap Q$,
being trivial,
(\ref{goclaim}) is proved.
\end{proof}
\section{Further conclusions in the\\
\hspace*{.2mm}finite-dimensional case}\label{adepend}
%
%
We collect further results
which are easily available
in the finite-dimensional case
(and required in \cite{MaZ}).
In particular, we
study the dependence of $a$-stable manifolds
on the parameter~$a$.
%
%
\begin{prop}\label{adepprop}
Let $M$ be an analytic manifold modelled on a
finite-\linebreak
dimensional vector space over a complete
ultrametric field $(\K,|.|)$.
Let $p\in M$ be a fixed point of
an analytic diffeomorphism
$f\colon M\to M$.
Abbreviate $\alpha:=T_p(f)$
and define $R(\alpha)$ as in {\rm\ref{findiset}}.
Then the following holds:
\begin{itemize}
\item[\rm(a)]
If $R(\alpha)\sub \;]0,1]$, then
$W_a^\obs(f,p)$ is a submanifold
of $M$, for each $a\in \;]0,1]\setminus R(\alpha)$.
\item[\rm(b)]
If $0<a<b\leq 1$ and $[a,b]\cap R(\alpha)=\emptyset$,
then $W_a^\obs(f,p)=W_b^\obs(f,p)$.
\item[\rm(c)]
If $a\in \;]0,1]$ and $\,]0,a]\cap R(\alpha)=\emptyset$,
then $W_a^\obs(f,p)=\{p\}$.
\end{itemize}
\end{prop}
\begin{proof}
(a) follows from Proposition~\ref{Bangd}
(using Corollary~\ref{spechyper} and~(\ref{sodecac})).

(b) Let $E$ and $\kappa$ be as in \ref{reusable}
(with $M_0:=M$), $\|.\|$ be a norm on
$E$ adapted to $\alpha:=T_p(f)$,
and $R(\alpha)$ as well as the subspaces
$E_\rho\sub E$ for $\rho>0$ be as in~\ref{findiset}.
By hypothesis on $a$ and $b$, we have
\[
X :=\bigoplus_{\rho<a}E_\rho=\bigoplus_{\rho<b}E_\rho
\quad\mbox{ and }\quad
Y :=\bigoplus_{\rho>a}E_\rho=\bigoplus_{\rho>b}E_\rho\,.
\]
Hence $E_{a,\obs}=E_{b,\obs}=X$
and $E_{a,\obu}=E_{b,\obu}=Y$,
by~(\ref{soahyp}).
Now let $\Omega_a$ and $\Omega_b$
be an $\Omega$ as in Theorem~\ref{ultrastabm},
applied with $a$ and $b$, respectively.
By Theorem~\ref{mainthm}\,(f) and the proof of
Theorem~\ref{ultrastabm},
we may assume that $\Omega_a=\kappa^{-1}(\Gamma_a)$
and $\Omega_b=\kappa^{-1}(\Gamma_b)$,
where
%
\begin{eqnarray}
\Gamma_a & =& \{ z\in B_r^E(0)\colon \mbox{($\forall n\in \N_0$) $g^n(z)$
is defined and $\|g^n(z)\|\leq a^nr$}\}\; \mbox{and}\notag\\
\Gamma_b & =& \{ z\in B_t^E(0)\colon \mbox{($\forall n\in \N_0$) $g^n(z)$
is defined and $\|g^n(z)\|\leq b^nt$}\}\label{graess}
\end{eqnarray}
for certain $r,t>0$ and $g$ is as in (\ref{reusable}).
Moreover, by Theorem~\ref{mainthm}\,(e),
we may assume that $r=t$, after replacing each
by $\min\{r,t\}$.
Then $\Gamma_a\sub \Gamma_b$
by~(\ref{graess}),
and hence $\Gamma_a=\Gamma_b$
(since both sets are graphs of functions on the
same domain, by Theorem \ref{mainthm}). 
Thus $\Omega_a=\Omega_b$,
entailing that $W_a^\obs(f,p)=W_b^\obs(f,p)$
as a set and also as an immersed submanifold
of~$M$ (cf.\ proof of Theorem~\ref{ultrastabm}).

(c) By~(\ref{soahyp}), we have
$E_{a,\obs}=\bigoplus_{\rho<a}E_\rho=\{0\}$,
whence $\Omega=\kappa^{-1}(\Gamma)=\{p\}$
in Theorem~\ref{ultrastabm} and its proof.
Thus $W_a^\obs(f,p)=\bigcup_{n\in \N_0}f^{-n}(\Omega)=\{p\}$.
\end{proof}
\section{\!\!\!Specific results concerning automorphisms\\
\hspace*{-2.9mm}of Lie groups}\label{seclie}
%
%
Throughout this section,
$G$ is an analytic Lie group
modelled on an ultrametric Banach space
over a complete ultrametric field $(\K,|.|)$,
and $\alpha\colon G\to G$ an analytic automorphism.
Then the neutral element $1\in G$ is a fixed point of~$\alpha$,
and hence our general theory applies.
We now compile some
additional conclusions which are specific to
automorphisms. Like results of the previous
sections, these are needed for the farther-reaching
Lie-theoretic applications described in the introduction.\\[2.5mm]
We begin with a corollary to Proposition~\ref{unicon}.
An automorphism $\alpha\colon G\to G$
is called \emph{contractive}
if $\lim_{n\to\infty}\alpha^n(x)=1$ for each $x\in G$.
\begin{cor}
If $G$ is finite-dimensional
and $\alpha\colon G\to G$ a contractive\linebreak
automorphism,
then every eigenvalue $\lambda$ of
$ L(\alpha) \tensor_\K  \id_{\wb{\K}}$ in an algebraic\linebreak
closure~$\wb{\K}$
has absolute value $|\lambda|<1$.
\end{cor}
\begin{proof}
$G$ is complete by \cite[Proposition~2.1\,(a)]{FOR},
and metrizable. Since every identity neighbourhood $P$ in $G$
contains an open subgroup $U$ of $G$ (see, e.g.,  \cite[Proposition~2.1\,(a)]{FOR}),
Lemma~1\,(a) in \cite{Sie}
provides an $\alpha$-invariant open subgroup $Q:=U_{(0)}\sub U\sub P$
of $G$. Hence~$1$ is a uniformly contractive fixed
point of~$\alpha$, and thus ``(a)$\impl$(c)'' in Proposition~\ref{unicon}
applies.
\end{proof}
%
%
%
\begin{prop}\label{proprev}
If $a\in \;]0,1]$
and $L(\alpha)$ is $a$-hyperbolic,
the following holds:
\begin{itemize}
\item[\rm(a)]
The $a$-stable manifold $W_a^\obs(\alpha,1)$ is an immersed
Lie subgroup of~$G$.
\item[\rm(b)]
If, moreover, $L(G)$ admits a centre subspace
with respect to $L(\alpha)$
and $L(G)_{1,\obu}=\{0\}$,
then $W_a^\obs(\alpha,1)$ is a
Lie subgroup of~$G$.
\end{itemize}
\end{prop}
\begin{proof}
(a) The proof of \cite[Proposition~4.6]{MaZ}
applies without changes.\footnote{In $\diamondsuit$, read ``$\,\leq a^n\,$''
as ``$\,< a^n r$.''}

(b) is a special case of Proposition~\ref{Bangd}.
\end{proof}
If $G$ is finite-dimensional,
then the extra hypotheses
in Proposition~\ref{proprev}\,(b)
mean that $R(L(\alpha))\sub \;]0,1]$
(see Corollary~\ref{spechyper} and (\ref{sodecac})).\\[2.5mm]
In the following situation,
hyperbolicity
is not needed to make $W^\obs$
a manifold.
%
%
%
\begin{prop}\label{notmostgen}
If $\alpha\colon G\to G$ is an automorphism
and $L(G)$ admits a centre subspace with respect
to $L(\alpha)\colon L(G)\to L(G)$,
then the following holds:
\begin{itemize}
\item[\rm(a)]
There exist a local stable manifold $V_1$
and a centre manifold $V_0$ around~$1$ with respect to~$\alpha$,
and a local stable manifold $V_{-1}$ around $1$ with respect
to $\alpha^{-1}$,
such that $V_1V_0V_{-1}$ is open in $G$ and the product map
\begin{equation}\label{thepromp}
\pi \colon V_1\times V_0\times V_{-1} \to V_1V_0V_{-1}\,\quad (x,y,z)\mto xyz
\end{equation}
is an analytic diffeomorphism.
\item[\rm(b)]
There is a unique
immersed submanifold structure on
$W^\obs(\alpha,1)$
such that
conditions {\rm (a)--(c)}
of Theorem~{\rm\ref{ultrastabm}}
are satisfied.
This immersed submanifold structure
makes $W^\obs(\alpha,1)$
an immersed Lie subgroup
of~$G$,
and the final assertion of Theorem~{\rm\ref{ultrastabm}}
holds.
Moreover,
$W^\obs(\alpha,1)=W_a^\obs(\alpha,1)$
for some $a\in \;]0,1[$ such that
$L(\alpha)$ is $a$-hyperbolic.
\end{itemize}
\end{prop}
\begin{proof}
(a) Set $E:=L(G)$ and let $E=E_1\oplus E_0\oplus E_{-1}$
be the decomposition into a
stable subspace $E_1$, 
centre subspace $E_0$ and unstable subspace $E_{-1}$
with respect to $L(\alpha)$,
and $\|.\|$ be an ultrametric norm as
in Definition \ref{defacentre}.
There is $a\in \;]0,1[$ such that
$\|L(\alpha)|_{E_1}\|<a$
and $\frac{1}{\|L(\alpha)^{-1}|_{E_{-1}}\|}>\frac{1}{a}$.
Then $L(\alpha)$ is $a$-hyperbolic
with $a$-stable subspace $E_1$
and $a$-unstable subspace
$E_0\oplus E_{-1}$
(and the norm $\|.\|$ as before).
Also $L(\alpha)^{-1}$ is $a$-hyperbolic,
with $a$-stable subspace $E_{-1}$
and $a$-unstable subspace
$E_0\oplus E_1$
(and the norm $\|.\|$ as before).
We let
$V_1$ be a local $a$-stable
manifold around~$1$ with respect to~$\alpha$
and $V_{-1}$ be a local $a$-stable
manifold around~$1$ with respect to~$\alpha^{-1}$
(see Theorem~\ref{locstamfd}\,(a));
by Theorem \ref{locstamfd}\,(c), we may assume that
$V_1\sub W_a^\obs(\alpha,1)$.
Also, we let $V_0$ be a centre manifold
around $p$ with respect to~$\alpha$ (see Theorem \ref{cthm}\,(a)).
Then $T_1(V_1)=E_1$, $T_1(V_0)=E_0$ and $T_1(V_{-1})=E_{-1}$,
whence
\[
L(G)\, =\, T_1(V_1)\oplus T_1(V_0) \oplus T_1(V_{-1})\, .
\]
Thus, after shrinking $V_1$, $V_0$ and $V_{-1}$
(which is possible by Theorem \ref{locstamfd}\,(c)
and \ref{cthm}\,(c)),
we may assume that $P:=V_1V_0V_{-1}$ is open in $G$
and the product map (\ref{thepromp})
is an analytic diffeomorphism
(by the Inverse Function Theorem).\vspace{1mm}

(b) Shrinking $V_1$, $V_0$ and $V_{-1}$ further
is necessary,
we may assume that there are $r>0$
and charts $\kappa_j\colon V_j\to B^{E_j}_r(0)$
with $\kappa_j(1)=0$ and $d\kappa_j=\id$
for $j\in \{-1,0,1\}$.
There is $s\in \;]0,r]$ such that $\alpha(\kappa^{-1}_j(B^{E_j}_s(0)))\sub
V_j$ for all $j\in \{-1,0,1\}$.
Let $g_j:=\kappa_j\circ \alpha \circ \kappa_j^{-1}|_{B^{E_j}_s(0)}$.
Shrinking~$s$, we achieve that
\begin{eqnarray}
\|g_0(x)\|\, &=& \,\;\|x\|\hspace*{1.08mm} \quad
\; \mbox{for each $x\in B^{E_0}_s(0)$,}\label{useelswh2}\\
\|g_1(x)\|\,  & < & \; a \|x\| \quad
\, \hspace*{.1mm}\mbox{for each $x\in B^{E_1}_s(0)$, and } \label{useelswh3}\\
\|g_{-1}(x)\| & > & a^{-1}  \|x\| \;\; \mbox{for each $x\in B^{E_{-1}}_s(0)$} \label{useelswh4}
\end{eqnarray}
(using (\ref{domi})
and parts (b) and (d) of Remark~\ref{remstrict}).
Then
\[
\kappa:=(\kappa_1\times \kappa_0\times \kappa_{-1})\circ \pi^{-1}\colon
P\to B^E_r(0)
\]
is a chart of~$G$
around~$1$. We set
$g:=g_1\times g_0 \times g_{-1} \colon B^E_s(0)\to
B^E_r(0)$
(where $B^E_s(0)=B^{E_1}_s(0)\times B^{E_0}_s(0) \times B^{E_{-1}}_s(0)$).
Abbreviate $Q:=\kappa^{-1}(B^E_s(0))$.
Then
%
\begin{equation}\label{locconj}
f|_Q=\kappa^{-1}\circ g \circ\kappa|_Q\,.
\end{equation}
If $z\in W^\obs (\alpha,1)$, there is $n_0\in \N_0$
such that $\alpha^n(z)\in Q$ for all $n\geq n_0$,
and
\begin{equation}\label{wllwspr}
\|\kappa(\alpha^n(z))\|\to 0\quad\mbox{as $\,n\to\infty$.}
\end{equation}
After replacing $z$ with $f^{n_0}(z)$,
we may assume that $n_0=0$.
Now $x=(x_1, x_0, x_{-1}):=\kappa(z)$
is an element of $B_s^E(0)$
such that $g^n(x)\in B^E_s(0)$
for all $n\in \N_0$
(cf.\ (\ref{locconj})).
Also
%
\begin{equation}\label{zwstp}
\lim_{n\to\infty}\|g^n(x)\|=0\,,
\end{equation}
by (\ref{wllwspr}).
Since
$\|g^n(x)\|=\max\{\|g_1(x_1)\|, \| g_0(x_0)\|,\|g_{-1}(x_{-1})\|\}$
for all $n\in \N_0$,
using (\ref{useelswh2}) and (\ref{useelswh4})
we obtain a contradiction to~(\ref{zwstp})
unless $x_0=0$ and $x_{-1}=0$.
Thus $x=x_1\in E_1$
and thus $z=\kappa_1^{-1}(x_1)\in V_1\sub W_a^\obs(\alpha,1)$,\linebreak
entailing that $W^\obs(\alpha,1)\sub W_a^\obs(\alpha,1)$.
The converse inclusion being trivial,
we deduce that
$W^\obs(\alpha,1)
=W_a^\obs(\alpha,1)$.
We give $W^\obs(\alpha,1)$
the manifold structure of
$W_a^\obs(\alpha,1)$.
It then is tangent to $E_{a,\obs}=E_1$ at~$1$.
Hence $W^\obs(\alpha,1)$
satisfies conditions (a)--(c)
of Theorem~\ref{ultrastabm},
and also the final assertion of the theorem.
To obtain the uniqueness of the immersed submanifold
structure subject to these conditions,
note that for any such structure on $W^\obs$,
each neighbourhood of $p$ in $W^\obs$ contains
am open $f$-invariant $p$-neighbourhood
(as this only requires (\ref{domi})
and Remark \ref{remstrict}\,(d)).
Now one shows as in the proof of Theorem \ref{locstamfd}\,(b)
that the germ of the latter
coincides with the germ we\linebreak
already have,
and this entails as in the proof of the uniqueness
part of\linebreak
Theorem \ref{ultrastabm}
that the new manifold structure on $W^\obs$
coincides with the one we already had
(further explanations are omitted,
because the assertion is not central).
All other assertions follow from
Proposition~\ref{proprev}.
\end{proof}
\begin{cor}
If $G$ is a finite-dimensional
Lie group, then there is a unique
immersed submanifold structure on
$W^\obs(\alpha,1)$
such that
conditions {\rm (a)--(c)}
of Theorem~{\rm\ref{ultrastabm}}
are satisfied.
This immersed submanifold structure
makes $W^\obs(\alpha,1)$
an immersed Lie subgroup
of~$G$.
Moreover,
$W^\obs(\alpha,1)=W_a^\obs(\alpha,1)$
for each $a\in \;]0,1[$ such that
$[a,1[\,\cap R(L(\alpha))=\emptyset$
and $\;]1,\frac{1}{a}]\cap R(L(\alpha))=\emptyset$.
\end{cor}
\begin{proof}
If we choose $\|.\|$ as a norm adapted to $L(\alpha)$ (as in Definition \ref{defnadpt})
in the proof of Proposition \ref{notmostgen},
then $E_1$, $E_0$ and $E_{-1}$
are the direct sum of all $L(G)_\rho$
with $\rho\in R(L(\alpha))$,
such that $\rho\in \;]0,1[$ (resp., $\rho=1$,
resp., $\rho\in \;]1,\infty[$), by (\ref{sodecac}).
If $a$ is as described in the current corollary, then $\|L(\alpha)\|<a$ and $\|L(\alpha)^{-1}\|<a$
(as is clear from (b) and (c) in Definition \ref{defnadpt}).
Therefore the proof of Proposition \ref{notmostgen}
applies with this choice of $a$.
\end{proof}
\appendix
\section{Proof of Proposition~\ref{thmloccsbb}}\label{sec1app}
For all integers $k\geq 2$, we choose
$\alpha_k\in \cL^k(E,E_1)$,
$\beta_k\in \cL^k(E,E_2)$
and $\gamma_k\in \cL^k(E,E_3)$
such that $a_k=\alpha_k\circ\Delta^E_k$,
$b_k=\beta_k\circ\Delta^E_k$,
$c_k=\gamma_k\circ\Delta^E_k$
and $\|\alpha_k\|$, $\|\beta_k\|$, $\|\gamma_k\| <1$.\\[2.5mm]
If $\phi$ is an analytic function of the form described in (c),
then (b) holds and
\begin{equation}\label{firstsr}
\sup\{\|(d_k,e_k) \| \,s^k\colon k\geq 2\}\; < \; s
\end{equation}
for each $s\in\;]0,r]$,
entailing that $\phi(B^{E_2}_s(0)) \sub B^{E_1\times E_3}_s(0)$.
Now $f(\Gamma_s)= \Gamma_s$ for all $s\in\;]0,r]$
if we can prove that $f(\Gamma_r)\sub \Gamma_r$,
exploiting that
\[
B_s^{E_2}(0)\to E_2,\quad
y\mto
f_2(\phi_1(y),y,\phi_3(y))
= By+\wt{f}_2(\phi_1(y),y,\phi_3(y))
\]
is an isometry with image $B^{E_2}_s(0)$
(by Remark \ref {behaveba}).
In particular,
\begin{equation}\label{neweqn}
\|f_2(\phi_1(y),y,\phi_3(y))\|
= \|y\| < s
\end{equation}
for all $s\in \;]0,r]$
and $y\in B_s^{E_2}(0)$.
Applying (\ref{firstsr}) with $s=r$, we see that
$f(\phi_1(y),y,\phi_3(y))$ is defined
for all $y\in B^{E_2}_r(0)$
and given globally by its Taylor series around $0$
(by Lemma~\ref{quanticomp}).
Now let $y\in B^{E_2}_r(0)$.
In view of (\ref{neweqn}) (applied with $s=r$),
we have
$f(\phi_1(y), y,\phi_3(y))\in \Gamma_r$
if and only if
\begin{eqnarray*}
\lefteqn{f(\phi_1(y),y,\phi_3(y))}\\
& = & \big(\phi_1(f_2(\phi_1(y),y,\phi_3(y))),\,
f_2(\phi_1(y),y,\phi_3(y)),\,
\phi_3(f_2(\phi_1(y),y,\phi_3(y))) \big),
\end{eqnarray*}
which holds if and only if
%
\begin{eqnarray}
f_1(\phi_1(y),y,\phi_3(y)) & = & \phi_1(f_2(\phi_1(y),y,\phi_3(y)))
\quad\mbox{and} \label{fstimpldd}\\
f_3(\phi_1(y),y,\phi_3(y))
& = & \phi_3(f_2(\phi_1(y),y,\phi_3(y))). \label{fstimplbb}
\end{eqnarray}
We mention that the right hand sides of (\ref{fstimpldd})
and (\ref{fstimplbb})
are given on all of $B_r^{E_2}(0)$ by their Taylor series around $0$,
since the homogeneous polynomial~$\zeta_j$ of degree~$j$
of the Taylor series of $f_2\circ (\phi_1,\id,\phi_3)$ around $0$
vanishes if $j=0$ and
has norm $\|\zeta_j\|\leq 1 $ if $j\geq 1$
(as will be verified in (\ref{neweqB}) and (\ref{verifetadd})),
whence $\|\zeta_j\| r^j\leq r$
and so Lemma~\ref{quanticomp} applies.
Therefore the validity of~(\ref{fstimpldd}) and (\ref{fstimplbb}) for all $y \in B^{E_2}_r(0)$
is equivalent to their validity on $B^{E_2}_t(0)$ for some $t\in\;]0,r]$
(as in (a)),
and is also equivalent to two identities
for formal series, namely
%
\begin{eqnarray}
\lefteqn{A(d_2(y)+d_3(y)+\cdots)}\notag\\
& + \!\! & \! a_2(d_2(y)+\cdots,y, e_2(y)+\cdots)
\; +\;  a_3(d_2(y)+\cdots,y, e_2(y)+\cdots)+\cdots \notag \\
&= \!\! & \!
d_2\big(By+b_2(d_2(y)+\cdots,y, e_2(y)+\cdots)
+b_3(\cdots) +\cdots\big)\notag\\
& + \!\! &\!  
d_3\big(By+b_2(d_2(y)+\cdots,y, e_2(y)+\cdots)
+b_3(\cdots)+\cdots\big)\; + \,\cdots \label{hugebb}
\end{eqnarray}
and
\begin{eqnarray}
\lefteqn{C(e_2(y)+e_3(y)+\cdots)}\notag\\
& \!\!& \! +\; c_2(d_2(y)+\cdots,y, e_2(y)+\cdots)
\; +\; c_3(d_2(y)+\cdots,y, e_2(y)+\cdots)+\cdots \notag \\
&=\!\! &\!
e_2\big(By+b_2(d_2(y)+\cdots,y, e_2(y)+\cdots)
+b_3(\cdots) +\cdots\big)\notag\\
& + \!\!&\! 
e_3\big(By+b_2(d_2(y)+\cdots,y, e_2(y)+\cdots)
+b_3(\cdots)+\cdots\big)\; + \, \cdots \, .\label{hugecc}
\end{eqnarray}
Equality of the second order terms on both sides
of (\ref{hugebb}) means that
%
\begin{equation}\label{wantunq}
A d_2(y) + a_2(0,y,0)=d_2(By)\,,
\end{equation}
which can be rewritten as
$(B^*-A_*) (d_2)=a_2(0,\sbull,0)$
with
$A_*:=\Pol^2(E_2,A)$ and
$B^*:=\Pol^2(B, E_1)$.
Since $B^*-A_*=B^*(\id-(B^{-1})^*A_*)$
with $\|(B^{-1})^*A_*\|$ $\leq \|B^{-1}\|^2 \|A\|< 1$,
we see that
%
\begin{equation}\label{wahlpbb}
d_2=(B^*-A_*)^{-1} a_2(0,\sbull,0)
= (\id-(B^{-1})^* A_*)^{-1} (B^{-1})^* a_2(0,\sbull,0)
\end{equation}
is the unique solution to (\ref{wantunq}),
and $\|d_2\|\leq \|B^{-1}\|^2 \|a_2(0,\sbull,0)\| < 1$.\\[2.5mm]
Equality of the second order terms in (\ref{hugecc})
means that
%
\begin{equation}\label{wantunqcc}
C e_2(y) + c_2(0,y,0)=e_2(By)\,,
\end{equation}
which can be rewritten as
$(C_*-B^*) (e_2)=- c_2(0,\sbull,0)$
with
$C_*:=\Pol^2(E_2,C)$ and
$B^*:=\Pol^2(B, E_3)$.
Since $C_*-B^*=C_*(\id-(C^{-1})_*B^*)$
with $\|(C^{-1})_*B^*\|$ $\leq \|C^{-1}\|\cdot\|B\|^2< 1$,
we see that
%
\begin{equation}\label{wahlpcc}
e_2 = - (C_*-B^*)^{-1} c_2(0,\sbull,0)
= (\id-(C^{-1})_* B^*)^{-1} (C^{-1})_* \, c_2(0,\sbull,0)
\end{equation}
is the unique solution to (\ref{wantunqcc}),
and $\|e_2\|\leq \|C^{-1}\|\cdot\|c_2(0,\sbull,0)\|
< 1$.\\[2.5mm]
Let $n\geq 3$ now and, by induction, suppose we have found
$d_k\in\Pol^k(E_2,E_1)$
and $e_k\in\Pol^k(E_2,E_3)$
with $\|d_k\|, \| e_k\| < 1$
for $k=2,\ldots, n-1$,
such that (\ref{hugebb})
and (\ref{hugecc})
hold up to order $n-1$
if these $d_2,\ldots, d_{n-1}$ and $e_2,\ldots, e_{n-1}$
are used.
For $k=2,\ldots, n-1$, let
$\delta_k\in \cL^k(E_2,E_1)$
and
$\eta_k\in \cL^k(E_2,E_3)$
such that
$d_k=\delta_k\circ\Delta^{E_2}_k$,
$e_k=\eta_k\circ\Delta^{E_2}_k$,
and $\|\delta_k\|,\|\eta_k\| < 1$.
Define $\xi_1(y):=(0,y,0)\in E$ for $y\in E_2$
and $\xi_k:=(\delta_k,0,\eta_k)\in \Pol^k(E_2,E)$
for $k\in\{2,\ldots, n-1\}$.\\[2.5mm]
Let  $B^*:=\Pol^n(B,E_1)$ and
$A_*:=\Pol^n(E_2,A)$.
Equality of the $n$-th order terms on both
sides of (\ref{hugebb}) amounts to
%
\begin{equation}\label{lesshugebb}
A_*(d_n) + r_n= B^* (d_n)+ s_n
\end{equation}
with
%
\begin{eqnarray}
r_n & = &
\sum_{k=2}^n \sum_{\stackrel{\scriptstyle j_1,\ldots, j_k\in \N}{j_1+\cdots+ j_k=n}}
\alpha_k\circ(\xi_{j_1},\ldots, \xi_{j_k})\, \quad\mbox{and} \label{fsterrbb}\\
s_n & = & \sum_{k=2}^{n-1} 
\sum_{\stackrel{\scriptstyle j_1,\ldots, j_k\in \N}{j_1+\cdots+ j_k=n}}
\delta_k\circ (\zeta_{j_1},\ldots, \zeta_{j_k}), \quad\mbox{where}\label{secerrbb}\\
\zeta_j & = &
\sum_{\ell=2}^j \sum_{\stackrel{\scriptstyle i_1,\ldots, i_\ell \in \N}{i_1+\cdots+ i_\ell=j}}
\beta_\ell\circ (\xi_{i_1},\ldots, \xi_{i_\ell})
\quad\mbox{for $j=2,\ldots, n-1$}\label{simfmlbb}
\end{eqnarray}
and $\zeta_1:=B$, where
\begin{equation}\label{neweqB}
\|\zeta_1\|\;=\; \|B\| \; \leq \; 1\, .
\end{equation}
As in the proof of Proposition~\ref{thmloccs},
we see that
%
\begin{equation}\label{fibasbb}
\| r_n\|< 1\,.
\end{equation}
Likewise, the norm of each summand in (\ref{simfmlbb})
is $<1$, and thus
\begin{equation}\label{verifetadd}
\|\zeta_j\|\; < \; 1\, .
\end{equation}
As a consequence, the norm of each summand in (\ref{secerrbb})
is $\leq \|\delta_k\|\cdot 1
< 1$.
Therefore,
%
\begin{equation}\label{secbasbb}
\|s_n\|< 1\,.
\end{equation}
In view of Lemma \ref{pbpfw},
(\ref{nonexpa}), (\ref{fibasbb}) and (\ref{secbasbb}),
\[
d_n:=(B^*-A_*)^{-1}(r_n-s_n)=(\id-(B^{-1})^*A_*)^{-1}(B^{-1})^*(r_n-s_n)
\in \Pol^n(E_2,E_1)
\]
is the unique solution to~(\ref{lesshugebb}),
of norm $\|d_n\|< 1$.\\[2.5mm]
Define  $B^*:=\Pol^n(B,E_3)$ and
$C_*:=\Pol^n(E_2,C)$.
Equality of the $n$-th order terms on both sides
of (\ref{hugecc}) amounts to
%
\begin{eqnarray}
C_*(e_n) + \rho_n & = & B^* (e_n)+ \sigma_n\quad\mbox{with}\label{lesshugecc}\\
\rho_n & = &
\sum_{k=2}^n \sum_{\stackrel{\scriptstyle j_1,\ldots, j_k\in \N}{j_1+\cdots+ j_k=n}}
\gamma_k\circ(\xi_{j_1},\ldots, \xi_{j_k})\, \quad\mbox{and} \notag\\
\sigma_n & = & \sum_{k=2}^{n-1} 
\sum_{\stackrel{\scriptstyle j_1,\ldots, j_k\in \N}{j_1+\cdots+ j_k=n}}
\eta_k\circ (\zeta_{j_1},\ldots, \zeta_{j_k})\notag
\end{eqnarray}
(with $\zeta_j$ from above).
As before, we see that $\|\rho_n\|, \|\sigma_n\|< 1$
and that
\[
e_n:=(C_*-B^*)^{-1}(\sigma_n-\rho_n)=(\id-(C^{-1})_*B^*)^{-1}(C^{-1})_*(\sigma_n-\rho_n)
\in \Pol^n(E_2,E_3)
\]
is the unique solution to~(\ref{lesshugecc}),
of norm $\|e_n\|< 1$.
This completes the proof of Proposition~\ref{thmloccsbb}.
\section{Proof of Theorem~\ref{locumfd}}\label{sec2app}
As usual, we first
discuss the situation in a local chart.
%
%
\begin{numba}\label{thesitu5}
We retain the setting
(and notation)
from \ref{thesitu},
except that we now assume that
$a\in [1,\infty[$
(instead of $a\in \;]0,1]$).
\end{numba}
%
%
%
%
\begin{thm}\label{mainthm5}
In the situation of {\rm\ref{thesitu5}}, consider the set
$\Gamma$ of all $z_0\in B_r^E(0)$
for which there exists a sequence $(z_n)_{n\in \N_0}$
in $E$ such that
\begin{equation}\label{thebiggam5}
\mbox{$a^n\|z_n\|< r\,$ and $\,f(z_{n+1})=z_n\,$
for all $\,n\in \N_0$, \,and}\;\,
\lim_{n\to\infty} a^n\|z_n\| = 0.
\end{equation}
Then $\Gamma$ has properties {\rm(a)} and {\rm(b):}
\begin{itemize}
\item[\rm (a)]
$\Gamma$ is locally $f$-invariant, more precisely
$f\bigl(\Gamma\cap B^E_{r/c}(0)\bigr)\sub \Gamma$
with $c:=\max\{1,\Lip(f)\}$.
\item[\rm (b)]
$\Gamma=\{(\phi(y),y)\colon y\in B_r^{E_2}(0)\}$
for an analytic
map $\phi\colon B_r^{E_2}(0)\to B_r^{E_1}(0)$
with $\phi(0)=0$
and $\phi'(0)=0$
$($whence $T_0 \Gamma=E_2)$.
\end{itemize}
Moreover, the following holds:
\begin{itemize}
\item[\rm(c)]
$\phi$ is Lipschitz, with $\Lip(\phi)\leq 1$.
\item[\rm (d)]
For each $s \in \;]0,r]$,
the set $\Gamma\cap B_s^E(0)$
has properties analogous to those of $\Gamma$
described in {\rm(\ref{thebiggam5})} and {\rm(b)} if we replace
$r$ with~$s$ there.
\end{itemize}
\end{thm}
\begin{proof}
(b) Let $b:=\frac{1}{a}$.
We abbreviate $\cU:=B^{\cS_b(E)}_r(0)$
and consider the map $g\colon \cU\to\cS_b(E)$
taking $z=(z_n)_{n\in \N_0}\in \cU$ with $z_n=(x_n ,y_n)$
to the sequence $g(z)$ with $n$-th entry
%
\begin{equation}\label{dfnsmg5}
g(z)_n:=
\left\{
\begin{array}{cl}
\bigl( f_1(z_1), \, 0 \bigr) &\mbox{if $\, n=0$;}\\
\bigl(f_1(z_{n+1}),\, B^{-1}(y_{n-1}-\wt{f}_2(z_n))\bigr) & \mbox{if $\, n\geq 1$}
\end{array}
\right.
\end{equation}
for $n\in \N_0$.
Then $g(0)=0$.
Using the left shift $\lambda$ on $\cS_b(E_1)$,
the left shift $\Lambda$ on $\cS_b(E_2)$,
the right shift~$\rho$ on $\cS_b(E_2)$ and
the projection $\pr_2\colon E=E_1\oplus E_2\to E_2$,
we can write~$g$ in the form
\[
g\, =\,  \big( \lambda \circ \cS_b(f_1),\, \rho\circ \cS_b(B^{-1})\circ (\cS_b(\pr_2)- \Lambda
\circ \cS_b(\wt{f}_2))\big)\,.
\]
In view of Lemma \ref{basicsseq} (a) and~(e), Lemma \ref{lamapsop}\,(a)
and Proposition \ref{propseqana}, this formula
shows that $g$ is analytic and Lipschitz with $\Lip(g)<1$.
Now
\[
G\, :=\, \id_\cU-g\colon \cU\to\cU
\]
is an analytic diffeomorphism and an isometry,
by the Ultrametric Inverse Function Theorem (Theorem \ref{IFT})
and the domination principle~(\ref{domi}).
\mbox{The map}
\[
w\colon B_r^{E_2} (0)\to\cU\,,\quad
w(y):= G^{-1}\big( (0,y),(0,0),\ldots\big)
\]
is analytic and
$(\id_\cU-g)(w(y))=G(w(y))=((0,y),(0,0),\ldots)$, i.e.,
%
\begin{equation}\label{leftright5}
w(y)=\big((0,y),(0,0),\ldots\big) + g(w(y))
\end{equation}
for all $y\in B^{E_2}_r(0)$. Comparing the $0$-th component
on both sides of (\ref{leftright5}),
we see that $w(y)_0=(0,y)+\bigl( f_1(w(y)_1),\, 0 \bigr)$
and thus
\[
w(y)_0=(\phi(y), y)\,,
\]
where $\phi\colon B^{E_2}_r(0)\to B^{E_1}_r(0)$ is the analytic function
given by
\begin{equation}\label{defirwphi5}
\phi(y):=f_1(w(y)_1)\,.
\end{equation}
Then $\phi(0)=0$,
and since
$g'(0) = \big( \lambda \circ \cS_b(A\circ \pr_1 ),\, \rho\circ \cS_b(B^{-1})\circ \cS_b(\pr_2)\big)
= D_1\oplus D_2$
with $D_1:= \lambda \circ \cS_b(A)$ and
$D_2:=\rho\circ \cS_b(B^{-1})$
of operator norm $<1$
(cf.\ proof of Proposition \ref{propseqana}),
we have
\[
(G^{-1})'(0)=(G'(0))^{-1}=(\id - (D_1\oplus D_2))^{-1}=\id+ \sum_{k=1}^\infty D^k_1\oplus D_2^k\,.
\]
Thus $w'(0).v  \hspace*{-.2mm}
= \hspace*{-.2mm} (G^{-1})'(0).((0,v),(0,0),\ldots) \hspace*{-.2mm}=\hspace*{-.2mm}
((0,v),(0,B^{-1}v),(0,B^{-2} v),\ldots)$
for all $v\in E_2$, and hence $\phi'(0)=0$.\\[2.5mm]
\emph{We claim that $w(y)_n=f(w(y)_{n+1})$ for each $n\in \N_0$
and each $y\in B_r^{E_2}(0)$.} If this is true,
then $w(y)\in \cU$
implies that $a^n\|w(y)_n\|$
is $<r$ and tends to $0$ as $n\to \infty$.
Hence $(\phi(y),y)\in \Gamma$
and thus $(\phi(y),y)\colon y\in B^{E_2}_r(0)\}  \sub \Gamma$.\\[2.5mm]
To prove the claim, let $n\in \N_0$.
Looking at the second component of the entry in (\ref{leftright5}) indexed
by $n+1$ and the first component of the entry indexed by $n$,
we see that
%
\begin{equation}\label{nothpart5}
\pr_2(w(y)_{n+1})=B^{-1}(\pr_2(w(y)_n)-\wt{f}_2(w(y)_{n+1}))
\end{equation}
and $\pr_1(w(y)_n)= f_1(w(y)_{n+1})$.
Multiplying (\ref{nothpart5})
with $B$, we obtain
\[
B\pr_2(w(y)_{n+1})=\pr_2(w(y)_n)-\wt{f}_2(w(y)_{n+1})\, ,
\]
whence
$\pr_2(w(y)_n)=B\pr_2(w(y)_{n+1})
+\wt{f}_2(w(y)_{n+1})=f_2(w(y)_{n+1})$ and hence indeed $w(y)_n=f(w(y)_{n+1})$.
The claim is established.

To get (b), it only remains to show
that $\Gamma\sub \{ (\phi(y),y)\colon y\in B^{E_2}_r(0)\}$.
To prove this inclusion, let $z_0=(x_0,y_0)\in \Gamma$; pick
$z:=(z_n)_{n\in \N_0}$ as in the definition of $\Gamma$.
Then $z \in \cU$ (since $a^n\|z_n\|<r$
and $a^n\|z_n\|\to 0$ by definition of $\Gamma$).
We claim that
%
\begin{equation}\label{convdir5}
z=((0,y_0),(0,0),\ldots)+g(z)\,.
\end{equation}
If this is true, then $G(z)=z-g(z)=((0,y_0),(0,0),\ldots)$
and hence $z=G^{-1}((0,y_0),(0,0),\ldots)=w(y_0)$.
As a consequence, $z_0=w(y_0)_0=(\phi(y_0),y_0)$
and thus $z_0\in \{ (\phi(y),y)\colon y\in B^{E_2}_r(0)\}$.
To prove the claim, note first that $\pr_2(z_0)=y_0$,
which is also the second component
of the $0$-th entry of the right hand side of~(\ref{convdir5}).
Given $n\in \N_0$,
equality of the first component of the index $n$ entry
of the sequences on the left and right of (\ref{convdir5})
means that
\[
\pr_1(z_n)=f_1(z_{n+1})\,,
\]
which holds by choice of $z$.
Next, if $n\geq1$ we use that $\pr_2(z_{n-1})=\pr_2 (f(z_n))=B\pr_2(z_n) + \wt{f}_2(z_n)$,
whence $B\pr_2(z_n)=\pr_2(z_{n-1})-\wt{f}_2(z_n)$
and hence
\[
\pr_2(z_n)=
B^{-1}(\pr_2(z_{n-1})-\wt{f}_2(z_n))\,.
\]
Therefore the second
components of the $n$-th entries of the sequences on the left and right
of (\ref{convdir5}) coincide.
As $n$ was arbitrary, (\ref{convdir5}) holds.

(c) Let $y,z\in B^{E_2}_r(0)$. Using that
$G$ is an isometry, we obtain
\begin{eqnarray*}
\|\phi(y)\!-\!\phi(z)\| \!& \!\leq\!&\!
\|(\phi(y),y)\!-\!(\phi(z),z)\| = \|w(y)_0\!-\!w(z)_0\|  \leq \|w(y)\!-\!w(z)\|_b\\
\!&\!=\!&\! \|G^{-1}((0,y),(0,0),\ldots)-G^{-1}((0,z),(0,0),\ldots)\|_b\\
\!&\!= \! &\!  \|((0,y-z),(0,0),\ldots)\|_b = \|y-z\|.
\end{eqnarray*}
\indent
(d) If $s\in \;]0,r]$, let $\Gamma_s$ be the set
of all $z_0\in B^E_s(0)$
for which there exists a sequence $(z_n)_{n\in \N_0}$
such that $a^n\|z_n\|<s$, $f(z_{n+1})= z_n$
and $a^n\|z_n\|\to 0$.
Applying (b) to $f|_{B^E_s(0)}$ instead of $f$,
we find that $\Gamma_s=\{(\psi(y),y)\colon y\in B^{E_2}_s(0)\}$
corresponds to the graph of an analytic function
$\psi\colon B^{E_2}_s(0)\to B^{E_1}_s(0)$.
Since $\Gamma_s\sub \Gamma$, it follows
that $\psi$ is the restriction of $\phi$ to $B^{E_2}_s(0)$
and $\Gamma\cap B^E_s(0)=\Gamma_s$.
Since $\Gamma_s$ has been obtained in the same way
as $\Gamma$, it has analogous properties.

(a) Let $s:=\frac{r}{c} \leq r$.
Then $\Gamma\cap B^E_s(0)=\Gamma_s$
(by (d)). Given
$z_0\in \Gamma_s$,
pick $(z_n)_{n\in \N_0}$
as in the proof of (d).
To see that $f(z_0)\in \Gamma$ (as required),
define $\zeta_n:=f(z_n)$
for $n\in \N_0$.
Then $a^n\|\zeta_n\|=a^n\|f(z_n)\|\leq \Lip(f) a^n  \|z_n\|<cs= r$
and $a^n\|\zeta_n\|\leq \Lip(f) a^n  \|z_n\|\to 0$,
showing that $f(z_0)=\zeta_0$ is in $\Gamma$.
\end{proof}
{\bf Proof of Theorem \ref{locumfd}.}
(a)
Let $E_1$ be the $a$-stable subspace
and~$E_2$ be the $a$-unstable subspace
of~$E:=T_p(M)$ with respect to $T_p(f)$,
and $\|.\|$ be a norm on $E=E_1\oplus E_2$
as in Definition \ref{defahyp}.
Let $\kappa\colon P\to U$, $Q$, $r>0$ and
$g\colon B^E_r(0)\to E$
be as in~\ref{reusable};
thus $g'(0)=T_p(f)$.
For $s\in \;]0,r]$,
let $\Gamma_s$ be the set
of all $z_0\in B_s^E(0)$
for which there exists a sequence $(z_n)_{n\in \N_0}$
in $E$ such that
$a^n\|z_n\|< s$ and $g(z_{n+1})=z_n$
for each $n$, and $\lim_{n\to\infty}a^n\|z_n\| = 0$.
After shrinking~$r$ if necessary,
Theorem \ref{mainthm5}
can be applied with $g$ in place of $f$
(cf.\ Remark \ref{remstrict}\,(d)).
Hence, there exists an analytic map
\[
\phi\colon B_r^{E_2}(0)\to B_r^{E_1}(0)
\]
with $\phi(0)=0$
and $\phi'(0)=0$,
such that $\Gamma_s=\{(\phi(y),y)\colon y\in B_s^{E_2}(0)\}$
and $g(\Gamma_{s/c})\sub \Gamma_s$
for each $s\in \;]0,r]$,
with $c:=\max\{1,\Lip(f)\}$.
Then $\Gamma_s$ is a submanifold
of $B_s^E(0)$ and $g$ restricts to an analytic
map $\Gamma_{s/c}\to\Gamma_s$ for each $s\in\;]0,r]$.
Now $\Omega_s:=\kappa^{-1}(\Gamma_s)$
is a submanifold of $\kappa^{-1}(B^E_s(0))$
(and hence of $M$),
such that $\Omega_{s/c}\sub \Omega_s$ is an open $p$-neighbourhood,
$f(\Omega_{s/c})\sub \Omega_s$,
and $f|_{\Omega_{s/c}}\colon \Omega_{s/c} \to\Omega_s$
is analytic.
Also, $T_p(\Omega_s)=E_2$
because $T_0(\Gamma_s)=E_2$,
and hence each $\Omega_s$ is a local $a$-unstable manifold
around $p$ with respect to $f$.

(b) We retain the notation from the proof of (a), set $B:=(T_p(f))|_{E_2}$
and pick $b\in\;]a,\frac{1}{\|B^{-1}\|}[$. 
We let $N$ be any local $a$-unstable manifold around $p$ with respect to $f$,
and $S \sub N$ be an open $p$-neighbourhood such that $f(S)\sub N$ and $f|_S\colon S \to N$
is analytic.
Consider a chart $\mu\colon V\to B^{E_2}_\tau(0)$
of~$N$ around~$p$ such that
$V\sub S$, $\mu(p)=0$ and $d\mu(p)=\id_{E_2}$.
There exists $\sigma\in \;]0,\tau]$
such that $h:=\mu \circ f\circ \mu^{-1}$
is defined on all of
$B^{E_2}_\sigma(0)$.
Since $h'(0)=T_p(f|_N)=B$ with $\frac{1}{\|B^{-1}\|} >b$,
Theorem \ref{IFT}\,(b) shows that,
after possibly shrinking~$\sigma$,
\begin{eqnarray}
h(B^{E_2}_s(0))  & = & B.B^{E_2}_s(0) \; \supseteq \;
B^{E_2}_{bs}(0)\;\; \mbox{for all $\, s\in \;]0,\sigma]$, \,and hence}\notag\\
h(B^{E_2}_{b^{-1}s}(0)) \! & \supseteq &
B^{E_2}_s(0)
\quad\mbox{for all $\, s\in \;]0,b \sigma]$.}\label{relvainfo7}
\end{eqnarray}
Now write
$g=(g_1,g_2)=g'(0)+\wt{g}\colon B^E_r(0)\to E_1\oplus E_2$,
where $\Lip(\wt{g})<a$
and $g'(0)=A\oplus B$ with $B$ as before and $A:= T_p(f)|_{E_1}$.
Then $Q\cap N$ is an immersed
submanifold of~$Q$ tangent to~$E_2$
and, after replacing~$N$ by an
open $p$-neighbourhood therein,
we may assume that~$N$ is a submanifold of~$Q$.
Since~$\kappa(N)$ is tangent to~$E_2$ at $0\in E$,
the inverse function theorem implies
that $\kappa(N)=\{(\psi(y),y)\colon y\in W\}$
for some open $0$-neighbourhood
$W\sub B^{E_2}_r(0)$
and analytic map $\psi\colon W\to E_1$
with $\psi(0)=0$, $\psi'(0)=0$ and $\Lip(\psi)\leq 1$
(after shrinking~$N$ if necessary).
Then $\mu:=\pr_2\circ \, \kappa|_N$ is a chart
for~$N$ with $\mu(p)=0$ and
$d\mu(p)=\id_{E_2}$
(where $\pr_2\colon E_1\oplus E_2\to E_2$).
Hence, by the discussion leading to (\ref{relvainfo7}),
there is $\sigma\in \;]0,r]$
with $B^{E_2}_\sigma(0)\sub W$ and
\begin{equation}\label{thsdfnd7}
g(\Theta_{b^{-1}s})\;\supseteq \; \Theta_s\quad \mbox{for all $\, s\in \;]0,\sigma]$,}
\end{equation}
where
$\Theta_s:=\{(\psi(y),y)\colon y\in B^{E_2}_s(0)\}$
for $s\in\;]0,\sigma]$.
Note that $\|z\|=\|y\|<s$
for all $s\in\;]0,\sigma]$ and $z=(\psi(y),y)\in\Theta_s$,
since $\Lip(\psi)\leq 1$.
Let $z_0\in \Theta_s$.
Recursively, using (\ref{thsdfnd7}),
we find a sequence $(z_n)_{n\in \N_0}$
such that $z_n\in \Theta_{b^{-n}s}$
and $g(z_n)=z_{n-1}$ for all $n\in \N$.
Then $\|z_n\|<b^{-n}s<a^{-n}s$
and $a^n\|z_n\|<(\frac{a}{b})^ns\to 0$ as $n\to \infty$,
whence $z_0\in \Gamma_s$.
Hence $\Theta_s\sub \Gamma_s$ and thus
$\Theta_s=\Gamma_s$ (as both sets are graphs of functions on the same domain).
Hence $\Theta_\sigma$ is an open submanifold
of $\Gamma_r$.
As a consequence, $\Omega_\sigma$ is an
open submanifold of $\Omega_r$ which contains $p$,
and it is also an open submanifold of $N$ as
$\Omega_\sigma=
\kappa^{-1}(\Gamma_\sigma)=\kappa^{-1}(\Theta_\sigma)
=\mu^{-1}(B^{E_2}_\sigma(0))$.
This completes the proof.\,\Punkt
{\footnotesize
Helge Gl\"{o}ckner, Universit\"at Paderborn, Institut f\"ur Mathematik,
Warburger Str.\ 100,\\
33098 Paderborn, Germany. E-Mail: glockner\at{}math.upb.de}

\begin{thebibliography}{99}\itemsep+3.7pt
%
\bibitem{AaM}
Abbondandolo, A. and P.
Majer,
\emph{On the global stable manifold},
Studia Math.\  {\bf 177} (2006),
113--131.
%
%
\bibitem{Aga}
Aguayo, J., M. Saavedra, M. Vallas,
\emph{Attracting and repelling points of analytic dynamical
systems of several variables in a non-archimedean formulation},
Theoret.\ and Math.\ Phys.\ {\bf 140} (2004),
1175--1181.
%
%
\bibitem{Ag2}
Aguayo, J., J. Gome\'{e}z, M. Saavedra and M. Wallace,
\emph{Perturbation of a $p$-adic dynamical systems in two variables},
pp.\ 39--51 in: Contemp.\ Math.\ {\bf 384},
AMS,
2005.
%
%
\bibitem{AaK}
Albeverio, S.,
S. De Smedt,
A. Yu.\ Khrennikov
and B. Tirotstsi,
\emph{$p$-adic dynamical systems},
Theor.\ and Math.\ Phys.\ {\bf 114} (1998), 276--287.
%
%
\bibitem{BaW}
Baumgartner, U. and G.\,A. Willis,
\emph{Contraction groups and scales of automorphisms
of totally disconnected, locally compact groups},
Israel J. Math.\ {\bf 142} (2004), 221--248.
%
%
\bibitem{Ben} Benedetto, R.\,L.,
\emph{Hyperbolic maps in $p$-adic dynamics},
Ergodic Theory Dynam.\ Systems {\bf 21}\,(2001),
1--11.
%
%
\bibitem{Bn2}
Benedetto, R.\,L., \emph{Wandering domains in
non-Archimedean polynomial dynamics}, Bull.\ London Math.\ Soc.\
{\bf 38} (2006), 937--950.
%
%
\bibitem{DIF}
Bertram, W., H. Gl\"{o}ckner and K.-H. Neeb, {\em Differential calculus
over general base fields and rings}, Expo.\ Math.\ {\bf 22} (2004), 213--282.
%
%
\bibitem{Bez}
B\'{e}zivin, J.-P.,  \emph{Sur les ensembles de Julia et Fatou des fonctions enti\`{e}res
ultram\'{e}triques}, Ann.\ Inst.\ Fourier (Grenoble) {\bf 51} (2001), 1635--1661.
%
%
\bibitem{Bo1}
Bourbaki, N., ``Vari\'{e}t\'{e}s diff\'{e}rentielles et analytiques.
Fascicule de r\'{e}sultats,'' Hermann, Paris, 1967.
%
%
\bibitem{Bo2}
Bourbaki, N., ``Lie Groups and Lie Algebras''
(Chapters 1--3), Springer-Verlag, Berlin 1989.
%
%
%
%
\bibitem{LaW}
de la Llave, R. and C.\,E. Wayne, {\em On Irwin's proof of
the pseudostable manifold theorem}, Math.\ Z. {\bf 219} (1995), 301--321.
%
%
\bibitem{Esc}
Escassut, A., ``Ultrametric Banach Algebras,''
World Scientific, 2003.
%
%
\bibitem{FaR}
Favre, C. and J. Rivera-Letelier,
\emph{Th\'{e}or\`{e}me d'\'equidistribution
de Brolin en dynamique $p$-adique},
C.\,R. Math.\ Acad.\ Sci.\ Paris {\bf 339} (2004),
271--276.
%
%
\bibitem{SCA}
Gl\"{o}ckner, H.,
{\em Scale functions on $p$-adic Lie groups},
Manuscr.\ Math.\ {\bf 97} (1998), 205--215.
%
%
%
%
\bibitem{NOA}
Gl\"{o}ckner, H., \emph{Lie groups over local fields
of positive characteristic need not be analytic},
J. Algebra {\bf  285} (2005), 356--371.
%
\bibitem{FOR}
Gl\"{o}ckner, H.,
\emph{Every smooth $p$-adic Lie group admits a compatible analytic
structure},
Forum Math.\ {\bf 18} (2006), 45--84.
%
%
\bibitem{IMP}
Gl\"{o}ckner, H.,
{\em Implicit functions from topological
vector spaces to Banach spaces},
Israel J. Math.\ {\bf 155} (2006), 205--252.
%
%
\bibitem{CMP}
Gl\"{o}ckner, H.,
\emph{Comparison of some notions of $C^k$-maps
in multi-variable non-archimedean analysis},
Bull.\ Belg.\ Math.\ Soc.\ Simon Stevin {\bf 14} (2007), 877--904.
%
%
\bibitem{MaZ}
Gl\"ockner, H., \emph{Contractible Lie groups
over local fields}, Math.\ Z.\ (online first);
cf.\ arXiv:0704.3737v1.
%
%
%
%
\bibitem{SUR}
Gl\"{o}ckner, H.,
\emph{Lectures on Lie groups over local fields},
preprint, arXiv:0804.2234v3.
%
%
\bibitem{SPO}
Gl\"{o}ckner, H.,
{\em Scale functions on Lie groups over local fields of positive
characteristic},
in preparation.
%
%
\bibitem{FIO}
Gl\"{o}ckner, H.,
\emph{Finite order differentiability properties,
fixed points and implicit functions over valued fields},
preprint, arXiv:math.FA/0511218.
%
%
%
%
%
%
%
%
\bibitem{HaK}
Hasselblatt, B. and A. Katok, ``Handbook of Dynamical Systems,''
Volume~1\,A, Elsevier, 2002.
%
%
\bibitem{HaY}
Herman, M. and J.-C. Yoccoz, \emph{Generalizations of some theorems
of small divisors to non-Archimedean fields}, pp.\ 408--447 in:
J. Palis, Jr.\ (ed.), ``Geometric Dynamics,'' Lecture Notes in Math.\ {\bf 1007},
Springer-Verlag, Berlin, 1983.
%
%
\bibitem{HPS}
Hirsch, M.\,W., C.\,C. Pugh and M. Shub,
``Invariant Manifolds,'' Springer, 1977.
%
%
\bibitem{Ir1}
Irwin, M.\,C., {\em On the stable manifold theorem},
Bull.\ London Math.\ Soc.\ {\bf 2} (1970), 196--198.
%
%
\bibitem{Ir2}
Irwin, M.\,C., {\em A new proof of the pseudostable manifold
theorem}, J.\ London Math.\ Soc.\ {\bf 21} (1980), 557--566.
%
%
\bibitem{Lin}
Lindahl, K.-O., ''On the  linearization of non-Archimedean holomorphic
functions near an indifferent fixed point,нн
Ph.D.-thesis, V\"{a}xj\"{o}, 2007;
available in electronic form at
http://urn.kb.se/resolve?urn=urn:nbn:se:vxu:diva-1713
%
%
\bibitem{Lub}
Lubin, J., {\em Non-Archimedean dynamical systems},
Compositio Math.\ {\bf 94} (1994), 321--346.
%
%
\bibitem{Mar}
Margulis, G.\,A.,
``Discrete Subgroups of Semisimple
Lie Groups,''
Springer, 1991.
%
%
\bibitem{Nil} Nilsson, M., {\em Fuzzy cycles of $p$-adic
monomial dynamical systems}, Far East J. Dyn.\ Syst.\
{\bf 5} (2003), 149--173.
%
%
\bibitem{Raj} Raja,
C.\,R.\,E.,
\emph{On classes of $p$-adic Lie groups},
New York J. Math.\  {\bf 5} (1999), 101--105.
%
%
\bibitem{Sch}
Schikhof, W.\,H., ``Ultrametric Calculus,''
Cambridge University Press, 1984.
%
%
\bibitem{Sie}
Siebert, E., \emph{Semisimple convolution semigroups
and the topology of contraction groups}, pp.\
325--343 in:
H.\ Heyer (ed.),
``Probability Measures on Groups IX''
(Oberwolfach 1988),
Springer, Berlin, 1989.
%
%
\bibitem{Roo} van Rooij, A.\,C.\,M.,
`Non-Archimedean Functional Analysis,''
Marcel Dekker, 1978.
%
%
\bibitem{Vie} Vieugue, D., ``Probl\`{e}mes de lin\'{e}arisation dans des familles
de germes analytiques,'' Ph.D.-thesis, Universit\'{e} d'Orleans,
2005; available in electronic form at
http://www.univ-orleans.fr/mapmo/publications/vieugue/these.php
%
%
\bibitem{Wan}
Wang, J.\,S.\,P., {\em The Mautner phenomenon for $p$-adic
Lie groups}, Math.\ Z. {\bf 185} (1984), 403--412.
%
%
\bibitem{Wei}
Weil, A., ``Basic Number Theory,'' Springer-Verlag, 1967.
%
%
\bibitem{Wel}
Wells, J.\,C., {\em Invariant manifolds of non-linear
operators}, Pacific J. Math.\ {\bf 62}\,(1976), 285--293.
%
%
\bibitem{Wi1}
Willis, G.\,A., {\em The structure of totally disconnected,
locally compact groups},
Math.\ Ann.\ {\bf 300} (1994), 341--363.
%
%
\bibitem{Wi2}
Willis, G.\,A, {\em Further properties of the scale function on a
totally disconnected group},
J. Algebra {\bf 237}\,(2001), 142--164.\vspace{1.5mm}
%
\end{thebibliography}
\end{document}